\documentclass{amsart}

\usepackage{amsmath,amsfonts,amssymb,amsthm,enumerate,tikz-cd,color}

\def\Z{\mathbb Z}
\def\Q{\mathbb Q}

\usepackage{hyperref} 
\hypersetup{
	colorlinks=true, 
	     urlcolor= blue,  
	     linkcolor= blue, 
     citecolor=blue,	
	bookmarksopen=true,           
}
\include{macros}

\theoremstyle{plain}
\newtheorem{theorem}{Theorem}[section]
\newtheorem{conjecture}{Conjecture}[section]
\newtheorem{lemma}{Lemma}[section]
\newtheorem{proposition}{Proposition}[section]
\newtheorem{corollary}{Corollary}[section]

\theoremstyle{definition}

\newtheorem{definition}{Definition}[section]

\newtheorem{remark}{Remark}[section]

\DeclareMathOperator{\ab}{ab}
\DeclareMathOperator{\Jac}{Jac}
\DeclareMathOperator{\Res}{Res}

\DeclareMathOperator{\Hom}{Hom}

\DeclareMathOperator{\cris}{cris}
\DeclareMathOperator{\loc}{loc}
\DeclareMathOperator{\Gal}{Gal}
\DeclareMathOperator{\rk}{rk}

\DeclareMathOperator{\Sym}{Sym}

\newcommand{\ma}{\mathrm{ma}}
\newcommand{\Nm}{\mathrm{Nm}}
\newcommand{\GSp}{\mathrm{GSp}}

\newcommand{\AJ}{\mathrm{AJ}}
\newcommand{\an}{\mathrm{an}}
\newcommand{\lie}{\mathrm{Lie}}
\newcommand{\Sel}{\mathrm{Sel}}
\newcommand{\NS}{\mathrm{NS}}
\newcommand{\End}{\mathrm{End}}
\newcommand{\Zar}{\mathrm{Zar}}
\newcommand{\ad}{\mathrm{ad}}
\newcommand{\gr}{\mathrm{gr}}
\newcommand{\Sp}{\mathrm{Sp}}
\newcommand{\spg}{\mathrm{sp}}
\DeclareMathOperator{\Ind}{Ind}

\newcommand{\et}{\mathrm{\acute{e}t}}
\DeclareMathOperator{\spec}{Spec}
\DeclareMathOperator{\Alb}{Alb}
\DeclareMathOperator{\codim}{codim}
\DeclareMathOperator{\Ker}{Ker}
\DeclareMathOperator{\Fil}{Fil}
\DeclareMathOperator{\dR}{dR}

\newcommand{\Mat}{\mathrm{Mat}}
\newcommand{\Be}{\mathrm{Be}}
\newcommand{\h}{\mathrm{h}}
\begin{document}
\title{Unlikely intersections and the Chabauty--Kim method over number fields}
\author{Netan Dogra}
\setcounter{tocdepth}{1}
\newcommand{\todo}[1]{[#1]}
\newcommand{\old}[1]{ old: [#1]}
\newcommand{\newversion}[1]{ new: [#1]}
\maketitle
\pagestyle{headings}
\markright{UNLIKELY INTERSECTIONS AND THE CHABAUTY--KIM METHOD}
\begin{abstract}
The Chabauty--Kim method is a tool for finding the integral or rational points on varieties over number fields via certain transcendental $p$-adic analytic functions arising from certain Selmer schemes associated to the unipotent fundamental group of the variety. In this paper we establish several foundational results on the Chabauty--Kim method for curves over number fields. The two main ingredients in the proof of these results are an unlikely intersection result for zeroes of iterated integrals, and a careful analysis of the intersection of the Selmer scheme of the original curve with the unipotent Albanese variety of certain $\Q _p $-subvarieties of the restriction of scalars of the curve. The main theorem also gives a partial answer to a question of Siksek on Chabauty's method over number fields, and an explicit counterexample is given to the strong form of Siksek's question.
\end{abstract}

\tableofcontents
\section{Introduction}
\label{intro}
Let $X$ be a hyperbolic curve over a number field $K$. Then, by a theorem of Siegel in the case $X$ is affine \cite{siegel} and by Faltings \cite{faltings1983} in general, $X$ has only finitely many integral points. Both these proofs are \textit{ineffective}, in the sense that they do not give a way to determine the set of integral points. The Chabauty--Kim method seeks to give a method for determining this set of points, by constructing a set of $p$-adic points containing the integral points, which one can prove is finite and compute in practice. Before explaining the Chabauty--Kim method more precisely, we clarify what we mean by integral points. Let $\overline{X}$ be a smooth projective curve over $K$, with $X\subset \overline{X}$ and complement $D:=\overline{X}-X$. We assume that $X$ is hyperbolic, i.e. that $2g(\overline{X})+\# D(\overline{K})>2$. Let $\mathcal{\overline{X}}$ be a regular model of $\overline{X}$ over $\mathcal{O}_{K,S}$, for $S$ a finite set of primes, $\mathcal{D}\subset \overline{\mathcal{X}}$ a normal crossings divisor with generic fibre equal to $D$, and $\mathcal{X}:=\overline{\mathcal{X}}-\mathcal{D}$. Then the theorems of Faltings and Siegel imply that $\mathcal{X}(\mathcal{O}_{K,S})$ is finite. 

Let $p$ be a prime which splits completely in $K$ and which is a prime of good reduction for $\mathcal{X}$ (henceforth when we say that a rational prime is a prime of good reduction for $\mathcal{X}$, we will mean that for all $v|p$, $v$ is not in $S$, $\overline{X}$ has good reduction at $v$ and $\mathcal{D}$ is \'etale over $\mathcal{O}_{K_v} $). Then Kim's method produces nested subsets $\mathcal{X}(\mathcal{O}_K \otimes \mathbb{Z}_p )_{S,n}$ and $\mathcal{X}(\mathcal{O}_{K_v } )_{S,n}$ of $\mathcal{X}(\mathcal{O}_K \otimes \mathbb{Z}_p )=\prod _{v|p}\mathcal{X}(\mathcal{O}_{K_v })$ and $\mathcal{X}(\mathcal{O}_{K_v })$ respectively, each containing the $S$-integral points of $\mathcal{X}$:
\begin{align*}
\mathcal{X}(\mathcal{O}_K \otimes \mathbb{Z}_p )\supset \mathcal{X}(\mathcal{O}_K \otimes \mathbb{Z}_p )_{S,1} \supset \mathcal{X}(\mathcal{O}_K \otimes \mathbb{Z}_p )_{S,2} \supset \ldots \supset \mathcal{X}(\mathcal{O}_{K,S}) , \\
\mathcal{X}(\mathcal{O}_{K_v } )\supset \mathcal{X}(\mathcal{O}_{K_v } )_{S,1} \supset \mathcal{X}(\mathcal{O}_{K_v} )_{S,2} \supset \ldots \supset \mathcal{X}(\mathcal{O}_{K,S}).
\end{align*}
When $S$ is empty, these sets will be written simply as $\mathcal{X}(\mathcal{O}_K \otimes \mathbb{Z}_p )_n$ and $\mathcal{X}(\mathcal{O}_{K_v })_n $, and in the case $X=\overline{X}$ is projective, they will be written simply as $X(K\otimes \Q _p )_n$ and $X(K_v )_n $.
A detailed description of $\mathcal{X}(\mathcal{O}_K \otimes \mathbb{Z}_p )_{S,n} $ and $\mathcal{X}(\mathcal{O}_{K_v })_{S,n}$ is given in section \ref{subsec:selmer}. A valuable feature of the Chabauty--Kim method is that the sets $\mathcal{X}(\mathcal{O}_K \otimes \mathbb{Z}_p )_{S,n}$ and $\mathcal{X}(\mathcal{O}_{K_v})_{S,n}$ are often computable in practice, and can be used to determine $\mathcal{X}(\mathcal{O}_{K,S})$ (see e.g. \cite{QC1}, \cite{BDCKW}, \cite{DC},\cite{DCW}). 

In this paper we establish foundational results about the scope of the Chabauty--Kim method over number fields, by establishing when we should expect the sets $\mathcal{X}(\mathcal{O}_K \otimes \mathbb{Z}_p )_{S,n}$ to be finite. The algorithmic side of the Chabauty--Kim method has also been developed in certain cases in depths one (see for example \cite{siksek2013explicit}) and two \cite{balakrishnan2020explicit}, and hence these results also suggest that the Chabauty--Kim method may be a practical method for the determination of rational points on hyperbolic curves of small genus over number fields.

\subsection{Main results}
The main result of this paper is that, essentially, the status of the Chabauty--Kim method over arbitrary number fields is the same as that over $\Q $. More precisely, we show that the Chabauty--Kim sets $\mathcal{X}(\mathcal{O}_K \otimes \mathbb{Z}_p )_{S,n}$ are finite under the same sets of hypotheses as those needed over $\Q $. As we explain below, the proofs of these results are however quite different from their analogues over $\Q $.
\begin{theorem}\label{finiteness}
Let $K$ be a number field, and $p$ a (rational) prime which splits completely in $K$, and $S$ a finite set of primes disjoint from $K$. Then $\mathcal{X}(\mathcal{O}_K \otimes \mathbb{Z}_p )_{S,n}$ is finite for all $n\gg 0 $ in each of the following cases.
\begin{enumerate}
\item $X=\mathbb{P}^1 -D$, where $D\subset \mathbb{P}^1 _K$ is a closed subscheme with $\# D(\overline{K})>2$.
\item $X/K$ is a smooth projective curve of genus $g>1$, and the conjecture of Jannsen, or the conjecture of Bloch--Kato, hold for all the product varieties $X^n $.
\end{enumerate}
In each of the following cases, there exists a finite extension $L|K$, such that for all rational primes $p$ which split completely in $L$, there is a prime $v$ of $K$ above $p$ such that $\mathcal{X}(\mathcal{O}_{K_v })_{S,n}$ is finite for all $n\gg 0$.
\begin{enumerate}
\setcounter{enumi}{2}
\item $X=E-O$,where $E/K$ is an elliptic curve with complex multiplication.
\item $X/K$ is a smooth projective curve of genus $g>1$ whose Jacobian is geometrically simple and has complex multiplication.
\end{enumerate}
\end{theorem}

\begin{remark}
The last two cases of Theorem \ref{finiteness} involve two conditions on the primes involved, which were erroneously not included in a previous version of this paper. The first, that of restricting to a positive proportion of primes, already occurs in the original work of Kim and Coates--Kim, since they require $p$ to split completely in the CM field by which the Jacobian has complex multiplication. It also occurs in our work because we pass to a finite extension of $K$, and require $p$ to split completely in that. The condition of only working with one prime above $p$ seems difficult to remove without a deeper understanding of localisation maps in Galois cohomology.
\end{remark}

\begin{remark}

When $K=\Q$, Theorem \ref{finiteness} is already known: cases (1), (2) and (3) are due to Kim \cite{kim:siegel}, \cite{kim:chabauty}, \cite{kim2010p}. Case (4) is due to Coates and Kim \cite{coates2010selmer}. As in Ellenberg and Hast \cite{ellenberg2017rational} this implies finiteness of $X(K_v )_n$ for any curve $X$ which geometrically dominates a hyperbolic curve with geometrically simple CM Jacobian (since we prove finiteness over all number fields, we can just pass to a finite extension over which the cover is defined, and note that for a cover $X\to Y$, finiteness of $Y(K_v )_n$ implies finiteness of $X(K_v )_n$). In particular, it implies finiteness of $X(K)$ for all hyperelliptic curves $X$ and number fields $K$.

When $K$ is totally real, case (1) was proved independently by Hadian \cite{hadian2011motivic} and Kim \cite{kim2012tangential}.

\end{remark}

\begin{remark}
As explained below, the main input into the proof of Theorem \ref{finiteness} is an unlikely intersection result (Proposition \ref{unlikely_intersection}) for Kim's unipotent Albanese morphism. The idea of relating unlikely intersections to the Chabauty--Kim method also appears in the thesis of Daniel Hast \cite[\S 5]{hast-thesis}, where he shows that, when $K$ has a real place, case (2) (assuming Bloch--Kato) and case (4) of the above theorem are implied by a sufficiently strong unlikely intersection result \cite[Conjecture 5.1]{hast-thesis}, which is a generalisation of a question of Siksek on Chabauty's method over number fields. In section \ref{subsec:siksek}, we give a counterexample to this strong unlikely intersection result, which also provides a negative answer to Siksek's question.

Hast has independently obtained related results pertaining to the relationship between the Chabauty--Kim method and generalisations of the Ax--Schanuel theorem \cite{hast}. Rather than proving finiteness of $\mathcal{X}(\mathcal{O}_K \otimes \mathbb{Z}_p )_{S,n}$, Hast gives a Chabauty--Kim proof of finiteness of the set $\mathcal{X}(\mathcal{O}_{K,S})$ of integral points, assuming Klingler's Ax--Schanuel conjecture \cite[Conjecture 7.5]{klingler2017hodge} for variations of mixed Hodge structure, as well as the assumptions in (1) to (4). The point of divergence in proving finiteness of $\mathcal{X}(\mathcal{O}_{K,S})$ rather than $\mathcal{X}(\mathcal{O}_K \otimes \mathbb{Z}_p )_{S,n}$ or $\mathcal{X}(\mathcal{O}_{K,v})_{S,n}$ is that, in both proofs, one wants to prove finiteness of the Zariski closure of a set of points (either $\Res _{K|\Q }(X)(\Q ) $ or $\Res _{K|\Q }(X)(\Q _p )_n $) by considering its image under the unipotent Albanese morphism and intersecting with a Selmer scheme. Whilst the former descends to $\Q$, the latter will in general just be defined over $\Q _p$ -- this phenomenon already occurs in the case of Chabauty's method over $\Q$, where the set $X(\Q _p )_1$ would be expected to typically contain non-algebraic points even when it is finite. We elaborate on the issue of fields of definition in the subsequent subsection.
\end{remark}

\begin{remark}
It would be interesting to adapt the proof of Theorem \ref{finiteness} to give a method for determining the rational points on more general higher dimensional varieties which have non-abelian unipotent fundamental groups. For example, can one apply the Chabauty--Kim method to determine the rational points on some of the higher dimensional Shimura varieties considered by Dimitrov and Ramakrishnan in \cite{DR15}?
\end{remark}
\subsection{Unlikely intersections and fields of definition}
The main new result used in the proof of Theorem \ref{finiteness} is a way of understanding the zeroes of certain transcendental functions (iterated integrals) on higher dimensional varieties. Kim's method works by constructing, under certain Galois cohomological assumptions, certain nontrivial locally analytic transcendental (Coleman) functions on $\prod _{v|p}\mathcal{X}(\mathcal{O}_v )$ whose zero locus contains $\mathcal{X}(\mathcal{O}_{K,S} )$. Since Coleman functions have only finitely many zeroes on $\mathcal{X}(\mathbb{Z}_p )$, this proves finiteness when $K=\Q $. 

Over number fields, $\prod _{v|p} \mathcal{X}(\mathcal{O}_v )$ is no longer one dimensional and hence the problem is to rule out that these functions conspire to have many zeroes in common. In this paper we resolve this problem by proving a foundational result (Proposition \ref{unlikely_intersection}) on unlikely intersections for iterated integrals. In the abelian case, the main unlikely intersection result (Proposition \ref{unlikely_intersection}) is a straightforward consequence of the Ax--Schanuel theorem for abelian varieties. 

To describe how such unlikely interesections can occur, consider the case of the product of $\mathbb{P}^1 -\{ 0 ,1,\infty\}$ with itself, with co-ordinate functions $z_1 ,z_2 $. Then the Coleman functions $\log (z_1 )-\log (z_2 )$ and $\log (1-z_1 )-\log (1-z_2 )$ are independent, in a suitable sense, but their common zero set is not codimension 2, as it contains the diagonal. The zero locus is not Zariski dense, however. 
A rather complicated way of seeing this is to observe that, on any positive dimensional component of the zero locus, $\frac{dz_1 }{z_1 }-\frac{dz_2 }{z_2 }$ and $\frac{dz_1 }{1-z_1 }-\frac{dz_2 }{1-z_2 }$ are colinear. Hence such a component must be contained in the subspace
\[
z_1 (z_2 -1)=z_2 (z_1 -1).
\]
In this paper we show that unlikely intersections of this form are, in some sense, the only thing that can go wrong. More precisely, in Proposition \ref{unlikely_intersection} we prove the following Ax--Schanuel-type statement for iterated integrals: if the codimension of the zero set of a set of iterated integral functions on a smooth geometrically connected quasi-projective variety is `smaller than expected', then this zero set is not Zariski dense in the subvariety. 
The method of proof is an elaborate version of the example above, inspired by Ax's original proof of the Ax--Schanuel theorem for tori and abelian varieties. A similar strategy is used by Bl\'azquez-Sanz, Casale, Freitag and Nagloo \cite{BSCFN}, where a stronger and vastly more general Ax--Schanuel type theorem is proved used a suitable nonabelian generalisation of Ax's proof. Indeed, their work, specialised to the specific context of unipotent connections on products of curves, gives an elegant conceptual explanation for some of the more elaborate (and apparently unmotivated) calculations in section \ref{sec:the_beef}. Although their argument is phrased in the language of connections on principal bundles for varieties over the complex numbers, standard arguments show their results imply unlikely intersection results for Coleman integrals. This will be pursued further in a subsequent work.

The second difficulty is establishing the existence of non-trivial Coleman functions vanishing on $\mathcal{X}(\mathcal{O}_K \otimes \mathbb{Z}_p )_{S,n}$. In the case $K=\Q $, by a theorem of Kim this is proved by showing that a localisation map from a Selmer scheme to local Galois cohomlogy variety is not dominant (this is a nonabelian, Galois cohomological analogue of the fact that Chabauty's method this is typically proved via (often conjectural) dimension bounds on certain Bloch--Kato Selmer groups associated to the \'etale fundamental groups of $X_{\overline{K}}$). Using Proposition \ref{unlikely_intersection}, we are then able to deduce Zariski non-density results for $\mathcal{X}(\mathcal{O}_K \otimes \mathbb{Z}_p )_{S,n}$ from results on dimension bounds on certain Bloch--Kato Selmer groups. To prove finiteness results, we argue by contradiction, assuming that $\mathcal{X}(\mathcal{O}_K \otimes \mathbb{Z}_p )_{S,n}$ is infinite, and taking $Z$ to be its Zariski closure. To apply the unlikely intersection result via the Chabauty--Kim method, we have to bound the intersection of Selmer groups associated to the fundamental group of the Weil restriction $\Res _{K|\Q }X_{\overline{\Q }}$ with local Galois cohomology groups associated to fundamental group of the $Z_{\overline{\Q }_p }$. 
 
Since the unipotent fundamental group of $Z_{\overline{\Q }_p }$ need not surject onto that of $\Res _{K|\Q }X_{\overline{\Q }}$, we are compelled to study a slightly novel local--global problem in Galois cohomology. Instead of needing to bound the image of the Bloch--Kato Selmer group of a global Galois representation $V$ in the local Galois cohomology of $V$, we need to bound the \textit{intersection} of the image of the Bloch--Kato Selmer group of $V$ with that of the local Galois cohomology of a subspace $W\subset V$ which is stable under $\Gal (\overline{\Q }_p | \Q _p )$, but which need not be stable under $\Gal (\overline{\Q }|\Q )$. This is done in several stages. First, although we cannot assume that $Z$ descends to a number field, we can descend the image of its Albanese variety in that of the Weil restriction of $X$ (see Lemma \ref{nearly_arithmetic}), and we know that the Albanese of $Z$ must surject onto the Jacobian of one of the factors of $\Res _{K|\Q }(X)$. For case (2), we then use this to show that there are pieces of the fundamental group of $\Res_{K|\Q }(X)$ whose global Galois cohomology cannot contain the fundamental group. Because our assumptions in this case imply strong bounds on the relevant $H^2$ groups, this suffices by standard duality arguments to prove finiteness.

For cases (3) and (4), this is not enough, because our unconditional bounds on $H^2$ in this case are weaker. This is why we instead prove finiteness of $X(\mathcal{O}_{K_v })_{S,n}$ for some $v$ above $p$. To explain how this simplifies the proof, consider the case where, for all distinct embeddings $\sigma _1 ,\sigma _2 :K\to \overline{K}$, $\Hom (\Alb (X_{\overline{K},\sigma _1 }),\Alb (X_{\overline{K},\sigma _2 }))=0$. Then, up to isogeny, the image of the unipotent fundamental group of $X$ in the fundamental group of $\Res (X)$ is the product of the images in the individual factors of $\Res (X)$. This allows us to reduce to the case where the image is very large. On the other hand, if $K|\Q $ is Galois and $X$ is defined over $\Q $, then there is an action of $\Gal (K|\Q )$ on the Galois cohomology over $K$ of the fundamental group of $X$. This gives an additional structure to show that the Galois cohomology of the fundamental group of $Z/\Q _p $ is not contained in the global Galois cohomology of the fundamental group of $\Res (X)$. In cases (3) and (4) we interpolate between these two techniques, using our strong assumptions on the Jacobian of $X$.

\subsection{Applications to explicit Chabauty--Kim}

Theorem \ref{finiteness} guarantees that the algorithms of Dan-Cohen--Wewers, Dan-Cohen and Corwin--Dan-Cohen \cite{DCW}, \cite{DC}, \cite{CDC} for computing $\mathcal{O}_{K,S}$-points on $\mathbb{P}^1 -\{ 0,1,\infty \}$ provably produce finite sets, extending the theoretical scope of the algorithms beyond totally real fields.
To use the Chabauty or Chabauty--Kim method to determine $X(K)$ for $X$ of genus bigger than one, at present one typically needs finiteness of $X(K\otimes \Q _p )_n$ when $n=1$ or $2$. When $K=\Q $, the foundational work of Chabauty implies that $X(\Q _p )_1$ is finite if $r<g$, where $r$ is the Mordell--Weil rank of the Jacobian of $X$, and $g$ is the genus. The `quadratic Chabauty lemma' \cite[Lemma 3.2]{QC1} states that $X(\Q _p )_2$ is finite when $r<g+\rho (\Jac (X))-1$, where $\rho (\Jac (X))$ denotes the rank of the Neron-Severi group of $\Jac (X)$ (over $\Q $). We prove a partial generalisation of this result to number fields. We also give sufficient conditions for finiteness of $X(K\otimes \Q _p )_1$, providing a partial answer to a question of Siksek (see below). The latter result, which can be phrased purely in terms of the classical Chabauty--Coleman method, is proved separately in section \ref{sec:AS-C}, although it is also a special case of the more general results obtained later in the paper.
\begin{proposition}\label{QC_numberfield}
Let $K|\Q $ be a finite extension of degree $d$, and let $r_1 (K)$ and $r_2 (K)$ denote the number of real and complex places. Let $P_{\mathbb{R}}$ denote the set of real places of $K$, and for each $v\in P_{\mathbb{R}}$, let $c_v \in G_{K}$ denote complex conjugation with respect to an embedding $\overline{K}\hookrightarrow \mathbb{C}$ extending $v$. Let $X/K$ be a smooth projective geometrically integral curve of genus $g>1$, and let $p$ be a prime which splits completely in $K$ and such that, for all $v|p$, $X$ has good reduction at $v$. Let $r$ denote the rank of $\Jac (X)(K)$. Suppose that for any two distinct embeddings $\sigma _1 ,\sigma _2 :K\hookrightarrow \overline{\Q }$, we have $\Hom (\Jac (X)_{\overline{\Q },\sigma _1 },\Jac (X)_{\overline{\Q },\sigma _2 })=0$. Then we have the following finiteness results.
\begin{enumerate}
\item If $r\leq d(g-1)$, then for all primes $p$ of good reduction and splitting completely in $K$, there is a prime $v$ of $K$ lying above $p$ such that $X(K_v )_1$ is finite.
\item If 
\[
\rk H^1 _f (G_{K,T},T_p (\Jac (X)))\otimes \Q _p )\leq d(g-1) +(r_2 (K)+1)(\rho (\Jac (X))-1)+\sum _{v\in P_{\mathbb{R}}}\rho (J_{\mathbb{C},v})-\rho (J_{\mathbb{R},v})
\] then for all primes $p$ as in part (1), there is a prime $v$ above $p$ such that $X(K_v )_2$ is finite.
\end{enumerate}
\end{proposition}

\begin{remark}
Assuming the finiteness of the $p$-primary part of the Tate--Shafarevich group for $(\Jac (X)$, we have $\rk H^1 _f (G_{K,T},T_p (\Jac (X))\otimes \Q _p )=\rk \Jac (X)(K)$.
By modifying the definition of the Selmer scheme as in \cite[Definition 2.2]{QC1} to include a condition on mapping to $J(K)\otimes \Q _p $, (condition (c) of \cite[Definition 2.2]{QC1}), one can prove finiteness of a modified version of $X(K_v )_2 $, for $v$ as above, whenever
\[
r\leq  d(g-1) +(r_2 (K)+1)(\rho (\Jac (X))-1)+\sum _{v\in P_{\mathbb{R}}}\rho (J_{\mathbb{C},v})-\rho (J_{\mathbb{R},v})
\]
(still assuming $\Hom (\Jac (X)_{\overline{\Q },\sigma _1 },\Jac (X)_{\overline{\Q },\sigma _2 })=0$ for all $\sigma _1 \neq \sigma _2 $).
This modified version of the Selmer scheme is also easier to compute with (see \cite{QC1}) but as this distinction is not needed elsewhere in the paper, we use the simpler definition.

In the case $K=\Q $, case (2) of Proposition \ref{QC_numberfield} is in fact more general than the quadratic Chabauty lemma \cite[Lemma 3.2]{QC1} mentioned above. Instead it reduces to \cite[Proposition 2]{QC2}, which states that $X(\Q _p )_2 $ is finite whenever
\[
\rk (J)<g-1+\rho (J)+\rk (\NS (J_{\overline{\Q }})^{c=-1}),
\]
where $J$ is Jacobian of $X$, and $c\in \Gal (\overline{\Q }|\Q )$ is a complex conjugation.
\end{remark}

Although the condition on $\Hom (\Jac (X)_{\overline{\Q },\sigma _1 },\Jac (X)_{\overline{\Q },\sigma _2 })$ is generic when $X$ is not defined over a subfield of $K$, and is practical to check for curves of small genus and number fields of low degree, it is natural to wonder whether a weaker condition is sufficient. There are known examples where $X(K\otimes \Q _p )_1$ is infinite and $r=d(g-1)$ coming from the fact that $X$ descends to a subfield of $K|\Q $.
In \cite{siksek2013explicit}, Siksek asks whether a sufficient condition for finiteness of $X(K\otimes \Q _p )_1$ is that $r\leq d(g-1)$ and $X$ is not defined over any intermediate extension of $K|\Q $. In section \ref{sec:AS-C}, we show that this question has a negative answer, but it is not clear that the condition that we obtain is optimal.

\begin{remark}
In the case when $K$ is an imaginary quadratic field, Proposition \ref{QC_numberfield} (modified as in the above remark) implys that $X(K_v )_2 $ is finite whenever $\rk J(K)\leq 2g$ and $\rho (J)>1$, which has applications to the scope of the algorithms developed by Balakrishnan, Besser, Bianchi and M\"uller \cite{balakrishnan2020explicit}.
\end{remark}

\subsection{Notation and plan of the paper}\label{subsec:notation}
In section \ref{sec:AS-C}, we explain the relation between the application of the Ax--Schanuel theorem to questions on Chabauty's method over number fields. Although the main result is essentially a special case of results proved later in the paper, we give an independent exposition that involves none of the machinery from Kim's method. We hope this may be of independent interest, and provide an illustration of how unlikely intersection results imply finiteness of Chabauty--Kim sets. In section \ref{sec:CK}, we provide a brief re-cap on the Chabauty--Kim method over $\Q $ and over number fields. We also explain how the Chabauty--Kim method for $X/K$ is related to the Chabauty--Kim method for $\Res _{K|\Q }(X)/\Q $, where $\Res _{K|\Q }(X)$ denotes the Weil restriction of $X$ (following an analogous construction of Stix, in the context of the section conjecture, in \cite{stix2010trading}). In section \ref{universal_con}, we recall the explicit description of the $p$-adic iterated integrals (or more precisely, the unipotent connections) which arise in Kim's method given in \cite{kim:chabauty}. This enable us to reduce the unlikely intersection statement required to a statement about zeroes of iterated integrals. In section \ref{sec:the_beef}, we then prove the unlikely intersection result, following the strategy outlined above. In \ref{sec:complete}, we describe how to reduce the proof of finiteness of $\mathcal{X}(\mathcal{O}_K \otimes \mathbb{Z}_p )_{S,n}$ to specific inequalities for the dimension of (abelian) Galois cohomology groups. In section \ref{the_proof} we complete the proof of Theorem \ref{finiteness} by verifying these inequalities.

The following notation will be used throughout the paper. If $X$ is a rigid analytic space over $\Q _p $ with formal model $\mathcal{X}$, then, for any subscheme $Y$ of the special fibre $\mathcal{X}_{\mathbb{F}_p },$ we denote by $]Y[\subset X$ the \textit{tube} of $Y$ in the sense of Berthelot. This is a rigid analytic space whose $\Q _p $-points are exactly the $\Q _p$-points of $X$ which reduce to $Y(\mathbb{F}_p )$ modulo $p$. When $Y$ is an $\mathbb{F}_p $-point of $X$, we refer to $]Y[$ as a residue disk. If $X$ and $Y$ are rigid analytic spaces over $\Q _p$, and $X$ has a fixed formal model $\mathcal{X}$, a \textit{locally analytic} morphism $F:X(\Q _p )\to Y(\Q _p )$ will mean a morphism of sets which, for all residue disks $]b[$, $b\in \mathcal{X}(\mathbb{F}_p )$, has the property that $F|_{]b[ (\Q _p )}$ comes from a morphism of rigid analytic spaces $]b[\to Y$.We will often view such a morphism as a morphism of rigid analytic spaces
\[
F':\sqcup _{b\in \mathcal{X}(\mathbb{F}_p )}]b[ \to Y.
\]
In particular, when we refer to the \textit{graph} of $F$, we will mean the image of the graph of $F'$ under the map
\[
\sqcup _{b\in \mathcal{X}(\mathbb{F}_p )}]b[ \times Y \to X\times Y.
\]

When we talk about \textit{irreducible components} of a rigid analytic space, the rigid analytic space will always be a closed affinoid subspace of a polydisk, or a union of such spaces.

We denote the Galois group of a field $F$ by $G_F$. When $F$ is a number field, and $S$ is a set of primes of a subfield $L\subset F$, we denote by $G_{F,S}$ the maximal quotient of $G_F$ unramified outside above all primes above $S$.

Recall that, given a scheme $Z$ over $K$, we say that a $\Q $-scheme is the \textit{Weil restriction} of $X$, denoted $\Res _{K|\Q }(Z)$, if it represents the functor on $\Q $-algebras
\[
S\mapsto Z(S\times _{\Q }K).
\]
We will sometimes write $\Res _{K|\Q }(Z)$ simply as $\Res (Z)$.
The only statements about existence of Weil restrictions that we will need are that smooth projective curves and abelian varieties over fields, or over Dedekind domains, admit Weil restrictions \cite[7.6.4]{BLR}.

Given an algebraic (or pro-algebraic) group $U$ we denote by $C_i U$ the central series filtration
\[
C_1 U=U \supset C_2 U =[U,U]\supset \ldots
\]
and, unless otherwise indicated, we denote $C_i /C_{i+1}U$ by $\gr_i U $. We similarly define $C_i L$ and $\gr _i L$ for a Lie algebra $L$.

If $K$ is a number field, and $T\supset S\sqcup \{v|p \}$ are finite sets of primes of $K$,
and $U$ is a unipotent group over $\Q _p $ with a continuous action of $G_{K,T}$, we denote by $H^1 _{f,S}(G_{K,T},U)$ the set of isomorphism classes of $G_{K,T}$-equivariant $U$-torsors which are unramified at all primes not in $S\sqcup \{v|p\}$ and are crystalline at all primes above $p$ in the sense of \cite{kim:siegel}. In the case where $U$ is a vector space, this recovers the usual Bloch--Kato Selmer group \cite{blochkato}. For a continuous $G$-representation $W$, we define $\h ^i (G,W):=\dim _{\Q _p }H^i (G,W)$, and similarly define $\h ^i _{f,S}(G,W)$ etc.

If $L|K$ is a finite extension of fields, we will write $\Ind ^K _L $ and $\Res ^K _L$ to denote the functors $\Ind ^{G_K }_{G_L}$ and $\Res ^{G_K}_{G_L}$ on Galois representations.
\subsection*{Acknowledgements}
The author is very grateful to Nils Bruin, Minhyong Kim, Martin Orr and Samir Siksek for helpful discussions related to this paper. He is also indebted to the anonymous referee for their careful reading of an earlier version of this paper, including the identification of a serious error, as well as Jan Steffen M\"uller, Pavel Coupek, David Lilienfeldt, Luciena Xiao and Zijian Yao for comments and corrections.

\section{Ax--Schanuel and Chabauty's method over number fields}\label{sec:AS-C}
\subsection{Chabauty's method and $p$-adic unlikely intersections}
The Ax--Schanuel theorem for abelian varieties \cite{ax1972some} can be translated into the following statement (the translation is identical to the geometric statement of Ax--Schanuel for tori \cite{ax1971schanuel} as stated in \cite{tsimerman:unlikely}).

\begin{theorem}[Ax--Schanuel \cite{ax1972some}]\label{ASAV}
Let $A$ be an abelian variety over $\mathbb{C}$ of dimension $n$.
Let $\Delta _{\exp } \subset A\times \lie (A)$ be the graph of the exponential, and let $p:A\times \lie (A)\to A$ be the projection. Let $V$ be a subvariety of $A\times \lie (A) $. Let $W$ be an irreducible component of the complex analytic space $V^{\an }\cap \Delta _{\exp } $. Suppose $\codim _{V^{\an }} (W)< n$. Then $p(W)$ is contained in a translate of a proper abelian subvariety of $A$.
\end{theorem}
\begin{remark}
This theorem can be deduced from Ax's original theorem as follows. Theorem 1 of \cite{ax1972some} says that there exists a complex analytic sub-group $B$ of $A\times \lie (A)$, containing $V^{\an }$ and $\Delta _{\exp }$, such that
\[
\codim _B (\Delta _{\exp })\leq \codim _{W^{\Zar}}(W).
\]
where $W^{\Zar }\subset A\times \lie (A)$ denotes the Zariski closure of $W$.
Let $B'$ be the subgroup variety of $A\times \lie (A)$ generated by $W^{\Zar }$. By Chevalley's theorem, $B'$ is of the form $B_1 \times B_2 $, for $B_1$ an abelian subvariety of $A$ and $B_2$ a sub-vector space of $\lie (A)$. If $p(W)$ is not contained in a translate of a proper abelian subvariety of $A$, then $p(W^{\Zar })$ generates $A$, and hence $B_1 =A$, and $p(V)\cup \Delta _{\exp }$ generates $A\times \lie (A)$, so that  $\codim _B (\Delta _{\exp })=n$.
\end{remark}

Now let $A/\Q _p $ be an abelian variety with good reduction (the generalisations of these statements to the case of bad reduction are also well known, but we omit them as we don't use them, and haven't defined the notion of locally analytic in this setting). The $p$-adic logarithm defines a locally analytic group homomorphism
\[
\log _A :A(\Q _p )\to \lie (A).
\]
in the sense of section \ref{subsec:notation}.
Theorem \ref{ASAV} can be translated into a statement about $\log _A$ via the following Lemma.
\begin{lemma}\label{exp_vs_log}
Let $A/\Q _p $ be an abelian variety with good reduction. Let $\Delta _{\log }\subset A\times \lie (A)$ denote the graph of the $p$-adic logarithm.
Choose an embedding $\Q _p \hookrightarrow \mathbb{C}$. Let 
\[
\widehat{\Delta }_{\exp }\subset \widehat{A}_{\mathbb{C}}\times \widehat{\lie }(A)_{\mathbb{C}}
\] denote the formal completion of the graph of the exponential at $(0,0)$.
Then
\[
\widehat{\Delta }_{\exp }=\widehat{\Delta }_{\log ,\mathbb{C}}.
\]
\end{lemma}
\begin{proof}
$\widehat{\Delta } _{\log ,\mathbb{C}}$ and $\widehat{\Delta }_{\exp }$ are the graphs of the morphisms of formal groups
\begin{align*}
\widehat{\log }_A : & \widehat{A}_{\mathbb{C}} \to \widehat{\lie }(A)_{\mathbb{C}}, \\
\widehat{\exp }_A: & \widehat{\lie}(A)_{\mathbb{C}} \to \widehat{A}_{\mathbb{C}}. \\
\end{align*}
The morphisms $\widehat{\exp }_A$ and $\widehat{\log }_A$ are inverse,
hence their graphs agree.
 \end{proof}
Now let $X$ be a smooth projective geometrically irreducible curve of genus $g$ over a number field $K$ of degree $d$ over $\Q$, and let $J$ be the Jacobian of $X$. Let $p$ be a rational prime of good reduction for  $X$. Let $\Res (X):=\Res _{K|\Q }(X)$ denote the Weil restriction of $X$. Recall that this is a smooth projective variety of dimension $d=[K:\Q ]$ over $\Q $. The morphism $\Res (X) \to \Res (J)$ induces a morphism $\Alb (\Res (X))\to \Res (J)$, which can be seen to be an isomorphism by base changing to $\overline{\Q }$.
Let $Y$ denote the formal completion of $\Res J\times \lie (\Res (J))$ at $(0,0)$.

Let $\log _J$ denote the $p$-adic logarithm map
\[
\Res (J)(\Q _p )\to \lie (\Res (J))_{\Q _p }.
\]
Since $\log _J$ is locally analytic, the graph of $\log _J$ gives a rigid analytic space $\Delta _{\log }$, and the formal completion at the point $(0,0)$ defines a formal subscheme $\widehat{\Delta } _{\log }$ of $Y_{\Q _p }$. 
Let $\AJ:\Res (X)\to \Res (J)$ denote the Abel--Jacobi map with respect to the chosen basepoint $b\in  \Res (X)(\Q )$.
Let $\overline{J(K)} \subset \lie \Res (J)_{\Q _p }$ denote the $\Q _p $-vector space generated by the image of $\Res (J)(\Q )$ in $\lie \Res (J)_{\Q _p }$ under the map $\log $.
Let $X(K\otimes \Q _p )_1 \subset \Res (X)_{\Q _p }^{\an }$ denote the rigid analytic space $(\log \circ \AJ )^{-1}(\overline{J(K)})$.

\begin{corollary}\label{padic_AS}
Let $Z$ be a positive dimensional geometrically irreducible subvariety of $\Res (X)_{\Q _p }$. Let $L$ denote the image of $\lie (\Alb (Z))$ in $\lie (\Res (J))_{\Q _p }$. Let $W$ be an irreducible component of $Z\cap X(K\otimes \Q _p )_1 $. Suppose 
\[
\codim _{Z}(W)<\codim _{L}(L\cap \overline{J(K)}).
\]
Then the projection of $W$ to $\Res (X)_{\Q _p }$ is not Zariski dense in $Z$.
In particular, if $\rk (J(K)) \leq d(g-1)$, then $X(K\otimes \Q _p )_1 $ is not Zariski dense in $\prod _{v|p}X(K_v )$.
\end{corollary}
\begin{proof}
Let $B$ denote the image of $\Alb (Z)$ in $\Res (J)_{\Q _p }$.
Define $V\subset \Res (J)_{\Q  _p }\times \lie (\Res (J))_{\Q  _p }$ by 
\[
V=\AJ (Z)\times \overline{J(K)}.
\]
Then the map $(\AJ ,\log _J \circ \AJ )$ induces an isomorphism
\[
X(K\otimes \Q _p )_1 \cap Z \simeq \Delta _{\log }\cap V.
\]
Let $D$ be a residue disk of $J_{\Q  _p }\times \lie (J)_{\Q  _p }$ intersecting $\Delta _{\log }\cap V$ non-trivially. For any point $(x,\log (x))$ in $\Delta _{\log }(\Q  _p )$, the group law on $J_{\Q  _p }\times \lie (J)_{\Q  _p }$ induces an isomorphism
\begin{equation}\label{group_law}
J_{\Q  _p }\times \lie (J)_{\Q  _p }\stackrel{\simeq }{\longrightarrow }J_{\Q  _p }\times \lie (J)_{\Q _p }.
\end{equation}
sending $(x,\log (x))$ to $(0,0)$.
Let $P$ be a point of $W$.
Using \eqref{group_law}, we may assume $P=(0,0)$. Let $\widehat{W}$ denote the formal completion of $W$ at $P$. Choose an embedding $\overline{\Q } _p \hookrightarrow \mathbb{C}$. Then, by Lemma \ref{exp_vs_log}, $\widehat{W}_{\mathbb{C}}$ is an irreducible component of the formal completion of $V_{\mathbb{C}}\cap \Delta _{\exp }$ at the $(0,0)$. Then $\widehat{W}_{\mathbb{C}}$ is the formal completion at $(0,0)$ of an irreducible component $\widetilde{W}$ of $\Delta _{\exp }\cap V_{\mathbb{C}}$ satisfying
\[
\codim _{V_{\mathbb{C}}}(\widetilde{W})< \dim (B).
\]
Hence, by Theorem \ref{ASAV}, $p(\widetilde{W})$ is contained in a translate of a proper abelian subvariety of $B_{\mathbb{C}}$, hence the same holds for $\widehat{W}$ and hence for $W$. Let $B'\subset B$ be the translate of a proper abelian subvariety contianing $W$. This implies that the Zariski closure of $W$ is contained in the pre-image of $B'$. In particular, the Zariski closure does not equal $Z$, since it does not generate $\Alb (Z)$.
 \end{proof}
\subsection{Applications to Siksek's question}\label{subsec:siksek}
In \cite{siksek2013explicit}, Siksek examined the question of when $X(K\otimes \Q _p )_1$ can be proved to be finite for a curve $X$ of genus $g$. In particular, he asked whether a sufficient condition for finiteness of $X(K\otimes \Q _p )_1$ is that, for all intermediate extensions $K|L|\Q $ over which $X$ admits a model $X' /L$, the Chabauty--Coleman condition 
\begin{equation}\label{samir:eqn}
\rk (\Jac (X' )(L))\leq (g-1)[L:\Q ]
\end{equation} 
is satisfied.

This question has a negative answer, as one can construct counterexamples as follows. Let $X_0$ be a curve of genus $g_0$ defined over $\Q $, such that the $p$-adic closure of $\Jac (X_0 )(\Q )$ in $\Jac(X_0 )(\Q _p )$ has finite index. Let $K|\Q $ be a finite extension such that the rank of $\Jac( X_0 )(K)$ is $\leq [K:\Q ](g_0 -1)$. Let $X\to X_{0,K}$ be a cover such that $X$ is not defined over a proper subfield of $K$, and the Prym variety $P=\Ker (\Jac (X)\to\Jac (X_0 )_K )^0$ has the property that the $p$-adic closure of $P(K)$ in $\prod _{v|p}P(K_v )$ has finite index. Then $X$ satifies \eqref{samir:eqn}, but $X(K\otimes \Q _p )_1$ will contain the pre-image of $X_0 (\Q _p )\subset \prod _{v|p}X_0 (\Q _v)$ in $\prod _{v|p}X(\Q _v )$, and in particular will be infinite.

For example, we can take $X_0 $ to be the curve
\[
y^2 =(x^4-\frac{11}{27})(x^2-\frac{27}{11}),
\]
take $K=\Q (\sqrt{33})$, and take $X$ to be the degree two cover of $X_0 $ given by
\begin{align*}
& y^2 =x^8 + \frac{2916\cdot b + 484}{297}x^6 + \frac{-128304\cdot b + 168112}{8019}x^4 \\
& + \frac{214057728\cdot b - 35529472}{23181643}t^2 + \frac{-10784721024\cdot b + 8742087808}{64304361},
\end{align*}
where $b:=\sqrt{11/27}$. The Jacobian of $X$ is isogenous to $\Jac (X_0 )_K$ times the rank two elliptic curve
\begin{align*}
y^2 + \frac{2916\cdot b + 484}{297}xy + \frac{276156864\cdot b + 116895680}{2381643}y \\
=x^3 + \frac{384912\cdot b - 168112}{8019}x^2 +\frac{3594907008\cdot b - 4270950016}{21434787}x.\end{align*}
This also gives a counterexample to Conjecture 5.1 of \cite{hast-thesis}, which is a generalisation of Siksek's question to the setting of Kim's method. A generalisation of this construction has recently been considered in \cite{triantafillou2020restriction}, where it is referred to as a base change Prym obstruction.

The Ax--Schanuel Theorem implies the following weaker form of Siksek's question has a positive answer. Informally, we can phrase the result as follows: Siksek's question asks whether, for finiteness of $X(K\otimes \Q _p )_1$, it is sufficient for all subvarieties of $\Res _{K|\Q }(X)$ arising from the diagonal embedding of $\Res _{L|\Q }(X)$ to satisfy the usual Chabauty--Coleman condition, where $L\subset K$ is a subfield. Whilst this is not true in general, the corollary below says that  it is sufficient for \textit{all} irreducible subvarieties of $\Res _{K|\Q }(X)_{\overline{\Q }_p }$ to satisfy a Chabauty--Coleman-type condition.
\begin{corollary}\label{siksek:cor}
Let $X$ be a smooth projective curve of genus $g$ over $K$. Let $J\to B$ be a quotient of $J$ defined over $K$. Suppose that the cokernel of 
\[
\Res (B)(K)\otimes \overline{\Q } _p\to \lie (\Res (B))_{\overline{\Q }_p }
\]
has rank $\geq [K:\Q ]$.
If $X(K\otimes\Q _p )_1$ is infinite then there exists a positive dimensional subvariety $Z$ of $\Res _{K|\Q }(X)_{\overline{\Q }_p }$, such that the image $A$ of $\Alb (Z)$ in $\Res _{K|\Q }(B)_{\overline{\Q }_p }$ satisfies
\[
\codim _{\lie (A)}\lie (A)\cap \overline{B(K)}<\dim (Z),
\]
where $\overline{B(K)}$ denotes the $\overline{\Q }_p $-subspace of $\lie (\Res _{K|\Q }(B))_{\overline{\Q }_p }$ generated by $B(K)$.
\end{corollary}
\begin{proof} Let $Z\subset \Res (X)_{\overline{\Q }_p }$ be a positive dimensional irreducible component of the Zariski closure of $X(K\otimes \Q _p )_1$. Let $A$ denote the image of $\Alb (Z)$ in $\Res _{K|\Q }(B)_{\overline{\Q }_p }$. Since $X(K\otimes \Q _p )_1 \cap Z$ is Zariski dense in $Z$ by construction, Corollary \ref{padic_AS} implies that the codimension of $\lie (A)\cap J(K)\otimes \overline{\Q}_p $ in $\lie (A)$ is less than the dimension of $Z$.
 \end{proof}

\begin{proof}[Proof of Proposition \ref{QC_numberfield}, case (1)]
Recall that we suppose that 
\begin{equation}\label{eqn:hom_condition}
\Hom (\Jac (X)_{\overline{\Q },\sigma _1 },\Jac (X)_{\overline{\Q },\sigma _2 })=0
\end{equation}
whenever $\sigma _1 \neq \sigma _2 $.
Let $Z\subset \Res (X)_{\Q _p }$ be an irreducible component of the Zariski closure of $\Res (X)(\Q _p )_1 $. Note that, since $Z$ is an irreducible component of the Zariski closure of a set of $\Q _p $-points of $\Res (X)_{\Q _p }$, it is actually geometrically irreducible. Suppose, for contradiction, that the projection of $Z$ to every factor of $\Res(X)_{\Q _p }\simeq\prod _{v|p}X_{K_v }$ is dominant. Then, by assumption, $\Alb (Z)$ surjects onto every factor of $\Res (J)_{\Q _p }$. Then, by \eqref{eqn:hom_condition}, it must be the case that $\Alb (Z)$ surjects onto $\Res (J)_{\Q _p }$. This contradicts Corollary \ref{siksek:cor}.
\end{proof}

\begin{remark}
Note that the condition in Proposition \ref{QC_numberfield} (concerning the homomorphisms between $J_{\sigma ,\mathbb{C}}$ for different embeddings $\sigma $) certainly implies that $X$ does not descend to a subfield, however it is strictly stronger.
The explicit counterexample given above has the property that it can be explained by a quotient curve which does descend to $\Q $. It is natural to wonder if there exist `stronger' counterexamples not explained by a quotient curve which descends to a subfield. For example, does there exist a genus two curve $X$ defined over a quadratic field $K|\Q $, with simple Jacobian $J$, which gives a negative answer to Siksek's question?
\end{remark}

\section{The Chabauty--Kim method}\label{sec:CK}
\subsection{Selmer varieties}\label{subsec:selmer}
To describe the main obstacle to proving finiteness over general number fields, we first explain Kim's method over $\Q $, following \cite{kim:siegel} \cite{kim:chabauty}. Let $\mathcal{X}, S$ and $p$ be as in the introduction. Let $T_0 $ denote the union of the set of primes in $S$, the set of primes ramifying in $K|\Q $ and the set of primes of bad reduction for $X$, and let $T:=T_0 \cup \{v|p\}$. Suppose we have a $K$-rational point $b\in \mathcal{X}(\mathcal{O}_{K,S})$. For any field $L|K$, and any $L$-point $y\in X(L)$, we have a $\Gal (\overline{L}|L)$-equivariant $\pi _1 ^{\et}(X_{\overline{L}},b)$-torsor (where $\pi _1 ^{\et}(X_{\overline{L}},b)$ denotes the \'etale fundamental group of $X_{\overline{L}}$) given by the \'etale torsor of paths $\pi _1 ^{\et}(X_{\overline{L}};b,y)$. 
Hence we have a commutative diagram
\begin{equation}
\begin{tikzpicture}
\matrix (m) [matrix of math nodes, row sep=3em,
column sep=3em, text height=1.5ex, text depth=0.25ex]
{\mathcal{X}(\mathcal{O}_{K,S}) & & H^1 (G_K ,\pi _1 ^{\et }(X_{\overline{K}},b)) \\ {\displaystyle \prod _{v\in S}X(K_v )\times \prod_{v\in T\backslash S}\mathcal{X}(\mathcal{O}_{K_v})} & & \displaystyle{ \prod_{v\in T}H^1 (G_{K_v },\pi _1 ^{\et }(X_{\overline{K}_v },b)) } \\ };
\path[->]
(m-1-1) edge[auto] node[auto] {} (m-1-3)
edge[auto] node[auto] {  } (m-2-1)
(m-1-3) edge[auto] node[auto] { $ \loc $} (m-2-3)
(m-2-1) edge[auto] node[below] {} (m-2-3);
\end{tikzpicture}
\end{equation}
Hence a natural obstruction to $(x_v )\in  \prod _{v\in S}X(K_v )\times \prod _{v\in T\backslash S}\mathcal{X}(\mathcal{O}_{K_v})$ coming from $x\in \mathcal{X}(\mathcal{O}_{K,S})$ is that $[\pi _1 ^{\et}(X_{\overline{K_v }};b,x_v )]$ lies in the subspace $\loc H^1 (G_{K},\pi _1 ^{\et }(X_{\overline{K}},b))$. 

In practice, the set $H^1 (G_K , \pi _1 ^{\et }(X_{\overline{\Q }},b))$ is rather mysterious, and Kim's method starts by replacing it with a more tractable object. Namely, for any variety $Z$ over a field $K$ of characteristic zero, and $b\in Z(\overline{L})$, we define $\pi _1 ^{\et ,\Q _p }(X_{\overline{L}},b)$ to be the $\Q _p$-unipotent completion of $\pi _1 ^{\et}(Z_{\overline{L}},b)$ \cite[\S 10]{Del89}, and define
\[
U_n (Z)=U_n (Z)(b):=\pi _1 ^{\et  ,\Q _p }(Z_{\overline{L}},b)/C_{n+1}\pi _1 ^{\et  ,\Q _p }(Z_{\overline{L}},b).
\]

Returning to the setting of $\mathcal{X}/\mathcal{O}_{K,S}$ and $b\in \mathcal{X}(\mathcal{O}_{K,S})$ as above, we have a commutative diagram
\[
\begin{tikzpicture}
\matrix (m) [matrix of math nodes, row sep=3em,
column sep=3em, text height=1.5ex, text depth=0.25ex]
{\mathcal{X}(\mathcal{O}_{K,S}) & & H^1 (G_{K,T},U_n (X)(b)) \\ {\displaystyle \prod _{v\in S}X(K_v )\times \prod_{v\in T\backslash S}\mathcal{X}(\mathcal{O}_{K_v})} & & \displaystyle{ \prod_{v\in T}H^1 (G_{K_v },U_n (X)(b)), } \\ };
\path[->]
(m-1-1) edge[auto] node[auto] { $j_n$ } (m-1-3)
edge[auto] node[auto] {  } (m-2-1)
(m-1-3) edge[auto] node[auto] { $ \loc _n $} (m-2-3)
(m-2-1) edge[auto] node[below] {$ \prod\limits_{v\in T}j_{n,v}$} (m-2-3);
\end{tikzpicture}
\]
where the map $j_{n}$ may be defined as follows: by definition of the unipotent completion, we have a Galois equivariant map $\pi _1 ^{\et }(X_{\overline{\Q }},b)\to U_n (X)(b)$, and hence by functoriality a map
\[
H^1 (G _K ,\pi _1 ^{\et }(X_{\overline{K}},b))\to H^1 (G_K ,U_n (b)).
\]
By Grothendieck's specialization theorem \cite[\S X]{SGA1}, the image of $[\pi _1 ^{\et }(X_{\overline{K}};b,x)]$ in $H^1 (G_K ,U_n (b))$ will be unramified at all $v$ outside of $T$, and hence defines an element of the subspace $H^1 (G_{K,T},U_n (b))$.

Hence a natural obstruction to $(x_v )\in  \prod _{v\in S}X(K_v )\times \prod _{v\in T\backslash S}\mathcal{X}(\mathcal{O}_{K_v})$ coming from $x\in \mathcal{X}(\mathcal{O}_{K,S})$ is that $(j_{n,v}(x_v ))$ lies in the subspace $\loc _{n}H^1 (G_{K,T},U_n (b))$. 
For a rational prime $p$ we denote by $\mathcal{X}(\mathcal{O}_K \otimes \mathbb{Z}_p )_{S,n} $ the image of \\ $(\prod _{v \in T}j_{n,v})^{-1}\loc _n H^1 (G_{K,T},U_n )$ in 
$\prod _{v|p}\mathcal{X}(\mathcal{O}_v )$ under the projection 
\[
\prod _{v\in S}X(K_v )\times \prod _{v\in T\backslash S}\mathcal{X}(\mathcal{O}_{K_v})\to \prod _{v|p}\mathcal{X}(\mathcal{O}_v ).
\]
That is, $\mathcal{X}(\mathcal{O}_K \otimes \mathbb{Z}_p )_{S,n}$ is simply the set of all tuples $(x_v )$ in $\prod _{v|p}\mathcal{X}(\mathcal{O}_{K_v })$ which extend to a tuple $(x_v )$ in $\prod _{v\in S}X(K_v )\times \prod _{v\in T\backslash S}\mathcal{X}(\mathcal{O}_{K_v})$ for which $(j_{n,v}(x_v ))$ lies in the image of $H^1 (G_{K,T},U_n (X)(b))$. For $v$ a prime above $p$, we define $\mathcal{X}(\mathcal{O}_{K_v})_{S,n}$ to be the projection of $\mathcal{X}(\mathcal{O}_K \otimes \Z _p )_{S,n}$ to $X(K_v )$. Equivalently, it is the set of $\mathcal{O}_{K_v}$-points whose $H^1 (G_{K_v },U_n )$ class extends to a $\prod _{v'\in T\backslash S}H^1 (G_{K_{v'}},U_n )$ class in the image of $H^1(G_{K,T},U_n )$.

By a theorem of Kim and Tamagawa \cite[Corollary 0.2]{kim2008component}, for $v$ prime to $p$ and not in $S$, the map $j_{n,v}$ has finite image. Let $\alpha \in \prod _{v\in T_0 -S}j_{n,v}(\mathcal{X}(\mathcal{O}(K_v )))$. Let $\Sel (U_n )_{\alpha } \subset H^1 _{f,T}(G_{K,T},U_n (b))$ denote the fibre of $\alpha $ with respect to the localisation map (the Selmer scheme of $U_n $ with local conditions $\alpha $). Let $\mathcal{X}(\mathcal{O}_{K,S})_{\alpha } \subset \mathcal{X}(\mathcal{O}_{K,S})$ denote the subset of points mapping to $\Sel (U_n )_{\alpha }(\Q _p )$ under $j_n $, and let $\mathcal{X}(\mathcal{O}_K \otimes \mathbb{Z}_p )_{\alpha }$ denote the subset of $p$-adic points mapping to $\loc _p (\Sel (U_n )_{\alpha })$. There is a commutative diagram
\begin{equation}\label{eqn:the_square_2}
\begin{tikzpicture}
\matrix (m) [matrix of math nodes, row sep=3em,
column sep=3em, text height=1.5ex, text depth=0.25ex]
{\mathcal{X}(\mathcal{O}_{K,S} )_{\alpha } & \Sel (U_n )_{\alpha } \\ 
\prod _{v|p}\mathcal{X}(\mathcal{O}_v ) & \prod _{v|p} H^1 _f (G_{K_v },U_n (X) (b)).  \\ };
\path[->]
(m-1-1) edge[auto] node[auto] { $j_n $ } (m-1-2)
edge[auto] node[auto] {  } (m-2-1)
(m-1-2) edge[auto] node[auto] { $ \loc _{p}$} (m-2-2)
(m-2-1) edge[auto] node[auto] {$\prod _{v|p}j_{n,v}$} (m-2-2);
\end{tikzpicture}
\end{equation}
 We define 
\[
\Sel (U_n ):=\sqcup _{\alpha }\Sel (U_n )_{\alpha },
\]
where the disjoint union is over all $\alpha$ in the finite set  
\[
\prod _{v\in T_0 -S}j_{n,v}(\mathcal{X}(\mathcal{O}_{K_v })) \subset \prod _{v\in T_0 -S }H^1 (G_{K_v },U_n (X)(b)).
\]
Note that
\[
\mathcal{X}(\mathbb{Z}_p \otimes \mathcal{O}_K )_n =(\prod _{v|p}j_v )^{-1}(\Sel (U_n )).
\]

More generally for any $G_K$-stable quotient $U$ of $U_n$, we can define maps $j_v :X(K_v )\to H^1 (G_K ,U)$, Selmer schemes $\Sel (U)$, and global maps $j:\mathcal{X}(\mathcal{O}_{K,S})\to \Sel (U)$, and $\mathcal{X}(\mathbb{Z}_p \otimes \mathcal{O}_K )_{S,U}$. Whenever $U$ is a quotient of $U'$, we have an inclusion
\[
\mathcal{X}(\mathbb{Z}_p \otimes \mathcal{O}_K )_{S,U} \supset \mathcal{X}(\mathbb{Z}_p \otimes \mathcal{O}_K )_{S,U'}.
\]

We recall the interpretation of $\Sel (U)_{\alpha }$ in terms of twisting.
\begin{lemma}\label{twisty_lemma}
Let $\widetilde{\alpha }\in H^1 (G_{K ,T},U)$ be a cohomology class whose image in $\prod _{v\in S}H^1 (G_{K_v },U)$ is equal to $\alpha =(\alpha _v )$. Let $U^{\widetilde{\alpha }}$ denote the twist of $U$ by the torsor $\widetilde{\alpha }$. Then we have an inclusion
\[
\Sel (U )_{\alpha }\hookrightarrow H^1 _{f,S} (G_{K,T},U^{\widetilde{\alpha }})
\]
\end{lemma}
\begin{proof}
Recall from \cite[\S 1.5.2 Proposition 34]{serre-gc} that the twisting construction defines an isomorphism
\[
H^1 (G_{K,T},U ^{\widetilde{\alpha }}) \simeq H^1 (G_{K,T},U),
\]
which sends the trivial torsor to the class of $\widetilde{\alpha } $, and is functorial in both arguments. Hence, under this isomorphism, classes which are trivial at $v\neq p$ go to classes which are equal to $\alpha _v$ at $v$, and classes which are crystalline go to classes which are crystalline.
 \end{proof}
In \cite{kim:chabauty}, Kim proves two fundamental properties of the diagram \eqref{eqn:the_square_2}. First, the Galois cohomology sets $\Sel (U )$ and $H^1 _f (G_{\Q _p },U)$ are representable by $\Q _p $-schemes of finite type, in such a way that the morphism $\loc _p $ is algebraic. Second, the map $\loc _{\alpha }$ is algebraic, and that for all $v|p$, $j_{v}$ is locally analytic (in the sense defined in the introduction), and for all $z\in \mathcal{X}(\mathcal{O}_v )$, the map $j_{v}|_{]z[}$ has Zariski dense image.
If $K=\Q $ and $\loc _{\alpha }$ is not dominant, then the set $\mathcal{X}(\mathbb{Z}_p )_{\alpha }$ is thus finite, since on each residue disk it is given by a non-trivial power series.
\subsection{Trading degree for dimension in the Chabauty--Kim method}
Over number fields, the situation is slightly more complicated. Suppose that the image of $\loc _{\alpha }$ has codimension $d$ in $\prod _{v|p}H^1 _f (G_{K_v },U(b))$. Then, on each residue polydisk $]z[ \subset \prod _{v|p}\mathcal{X}(\mathcal{O}_v )$, we deduce that $]z[ \cap \mathcal{X}(\mathcal{O}_K \otimes \mathbb{Z}_p )$ is contained in the zeroes of $d$ power series. However, this does not imply that $]z[ \cap \mathcal{X}(\mathcal{O}_K \otimes \mathbb{Z}_p )$ is finite.

First, we replace the problem of finding $K$-rational points on $X$ with that of finding $\Q $-rational points on the Weil restriction $\Res _{K|\Q }(X)$. We recall some properties of the Weil restriction from \cite{stix2010trading}. Given topological groups $G,H,N$ with $H<G$ finite index and a continuous action of $H$ on $N$, define $\Ind ^G _H (N)$ to be group of continuous left $H$-equivariant maps
$
G\to N.
$
This has a natural continuous action of $G$ (see \cite[I.5.8]{serre-gc}).
\begin{proposition}[\cite{stix2010trading}, Proposition 8]\label{nonab_shap}
Let $G$ be a profinite group, and $H$ a finite index subgroup. Let $U$ be a topological group with a continuous action of $H$, and let 
$
\Ind ^G _H U \to U
$ 
denote the non-abelian induction, as defined in \cite[2.1.2]{stix2010trading}. Then the natural map
\[
H^1 (G,\Ind ^G _H U )\to H^1 (H,U).
\]
is an isomorphism.
\end{proposition}

\begin{lemma}[Stix,\cite{stix2010trading}]
There is an isomorphism
\[
\pi _1 ^{\et}(\Res (X),b)\simeq \Ind ^\Q _K \pi _1 ^{\et}(X,b),
\]
inducing a $G_{\Q }$-equivariant isomorphism
\[
\pi _1 ^{\et}(\Res (X)_{\overline{\Q }},b) \simeq \Ind ^\Q _K \pi _1 ^{\et}(X_{\overline{K}},b).
\]
\end{lemma}

By Proposition \ref{nonab_shap}, this induces an isomorphism 
\[
H^1 (G_K ,\pi _1 ^{\et}(X_{\overline{K}},b))\simeq H^1 (G_{\Q },\pi _1 ^{\et}(\Res (X)_{\overline{\Q }},b)),
\]
giving a commutative diagram whose vertical maps are bijections 
\[
\begin{tikzpicture}
\matrix (m) [matrix of math nodes, row sep=3em,
column sep=3em, text height=1.5ex, text depth=0.25ex]
{X(K) & H^1 (G_{K},\pi _1 ^{\et}(X_{\overline{K}},b)) & H^1 (G_{K },U_n (X)(b) ) \\ 
\Res (X)(\Q ) & H^1 (G_{\Q },\pi _1 ^{\et}(\Res (X)_{\overline{\Q }},b) & H^1 (G_{\Q },U_n (\Res (X)(b )).  \\ };
\path[->]
(m-1-1) edge[auto] node[auto] { } (m-1-2)
edge[auto] node[auto] {  } (m-2-1)
(m-1-2) edge[auto] node[auto] { } (m-2-2)
edge[auto] node[auto] {  } (m-1-3)
(m-1-3) edge[auto] node[auto] { } (m-2-3)
(m-2-1) edge[auto] node[auto] {} (m-2-2)
(m-2-2) edge[auto] node[auto] {} (m-2-3);
\end{tikzpicture}
\]

\begin{lemma}\label{nonab_mack}
Let $G$ be a topological group, $H$ a finite index subgroup, and $K<G$ a closed subgroup. Let $U$ be a topological group with a continuous action of $H$. Then we have a $K$=equivariant isomorphism of topological groups
\[
\Ind ^G _H (U)\simeq \prod _{x\in H\backslash G/K}\Ind ^ K_{K\cap xHx^{-1}} U^x .
\]
where $U^x$ denote the group $U$ with action twisted by conjugating by $x$.
\end{lemma}
\begin{proof}
Recall that $\Ind ^G _H (U)$ is, by definition, the set of continuous $H$-equivariant functions $G\to U$. Such a function is uniquely determined by where it sends $H \backslash G$. Hence, via the bijection $G/H \simeq \sqcup _{x\in H \backslash G /K}HxK$, we obtain a bijection 
\[
\Ind ^G _H (U)\to \prod _{x\in H \backslash G /K}\Ind ^K _{K\cap xHx^{-1} }U^x .
\]
which may be checked to be $K$-equivariant, as in the classical case of Mackey's restriction formula.
 \end{proof}
Lemma \ref{nonab_mack} and Proposition \ref{nonab_shap} together imply that, for all primes $l\neq p$, we have isomorphisms
\begin{equation}\label{local_shiz}
\prod _{v|l}H^1 (G_{K_v} ,U_n (X)(b))\simeq H^1 (G_{\Q _{\ell}} ,U_n (\Res (X)(b)).
\end{equation}

For the remainder of this paper, we denote $U_n (\Res (X))(b)$ by $U_n (b)$, or sometimes simply $U_n$. 
\begin{lemma}
At $p$, we have an isomorphism of unipotent groups with filtration over $\Q _p$
\[
U_n ^{\dR}(\Res _{K|\Q }(X)_{\Q _p })(b)\simeq \prod _{\sigma :K\hookrightarrow \Q _p }U_n ^{\dR}(X_{\Q _p ,\sigma })(b_{\sigma }),
\]
where the product is over all $\sigma $ in $\Hom (K,\Q _p )$.
\end{lemma}
\begin{proof}
This follows from the fact that, since $p$ splits completely in $K$, we have
\[
\Res _{K|\Q }(X)\times _{\Q }\Q _p \simeq \prod _{\sigma }X\times _{K,\sigma }\Q _p .
\]
 \end{proof}
Hence we obtain a commutative diagram whose vertical maps are bijective
\[
\begin{tikzpicture}
\matrix (m) [matrix of math nodes, row sep=3em,
column sep=3em, text height=1.5ex, text depth=0.25ex]
{ \prod _{v|p} X(K_v ) & \prod _{v|p}H^1 _f (G_{K_v },U_n (X)(b)) & \prod _{v|p}U_n ^{\dR}(X_{K_v })/F^0  \\ 
\Res (X)(\Q _p ) & H^1 _f (G_{\Q _p },U_n (\Res (X))(b)) &  U_n ^{\dR}(\Res (X)_{\Q _p })/F^0  \\ };
\path[->]
(m-1-1) edge[auto] node[auto] {  } (m-1-2)
edge[auto] node[auto] {  } (m-2-1)
(m-1-2) edge[auto] node[auto] { } (m-2-2)
edge[auto] node[auto] {} (m-1-3)
(m-2-1) edge[auto] node[auto] {} (m-2-2)
(m-1-3) edge[auto] node[auto] {} (m-2-3)
(m-2-2) edge[auto] node[auto] {} (m-2-3);
\end{tikzpicture}
\]
If $S\subset S'$, then by definition $\mathcal{X}(\mathbb{Z}_p \otimes \mathcal{O}_K )_{S,n}\subset \mathcal{X}(\mathbb{Z}_p \otimes \mathcal{O}_K )_{S',n}$, and in particular finiteness of the latter implies finiteness of the former. Hence, enlarging the set $S$ if necessary, we may assume that $S$ is of the form $\{v|l:l\in S_0 \}$, for a finite set $S_0$ of rational primes. Then, by \eqref{local_shiz} and the preceding commutative diagrams, the bijection $\mathcal{X}(\mathcal{O}_K \otimes \mathbb{Z}_p )\simeq \Res _{\mathcal{O}_{K,S} |\mathbb{Z}_{S_0 }}(\mathcal{X})(\mathbb{Z}_p )$ induces a bijection
\begin{equation}\label{reduce_2_weil}
\Res _{\mathcal{O}_{K,S}|\mathbb{Z}_{S_0 } }(\mathcal{X})(\mathbb{Z}_p )_{S_0 ,n}\simeq \mathcal{X}(\mathcal{O}_K \otimes \mathbb{Z}_p )_{S,n}.
\end{equation}
To ease notation, we will sometimes write $\Res _{\mathcal{O}_{K,S}|\mathbb{Z}_{S_0 } }(\mathcal{X})$ simply as $\Res _{K|\Q }(\mathcal{X})$, or $\Res (\mathcal{X})$. Hence \eqref{reduce_2_weil} reduces Theorem \ref{finiteness} to proving finiteness of 
$\Res _{K|\Q }(\mathcal{X})(\mathbb{Z}_p )_{S ,n}$.

We recall the following result from \cite{BDCKW}. Although the proof was given for curves, it also applies for a general smooth geometrically irreducible quasi-projective variety.
\begin{lemma}\cite[\S 2]{BDCKW}\label{indpt_basept}
Let $U(b)$  be a quotient of $U_n (b)$, and $b'$ another basepoint. Let $U(b')$ be the corresponding quotient of $U_n (b')$. Then 
\[
X(\Q _p )_{U(b)} = X(\Q _p )_{U(b')}.
\]
\end{lemma}
\subsection{Tangential localization}
In this subsection we recall Kim's description of the map on tangent spaces
\[
d\loc _p :T_c H^1 _f (G_{\Q ,T},U_n )\to T_{\loc _p (c)}H^1 _f (G_{\Q _p },U_n  )
\]
in terms of (abelian) Galois cohomology. For a variety $Z$, we let $L_n (Z)$ denote the Lie algebra of the group $U_n (Z)$, and similarly denote by $L_n ^{\dR} (Z)$ the Lie algebra of $U_n ^{\dR} (Z)$. When $Z=\Res (X)$, we write these simply as $L_n $ and $L_n ^{\dR}$.
\begin{proposition}[\cite{kim2012tangential}, Proposition 1.1, 1.2 and 1.3]\label{tangential_loc}
For any $c\in H^1 _f (G_{\Q ,T},U_n )$, there is a commutative diagram whose vertical maps are isomorphisms
\[
\begin{tikzpicture}
\matrix (m) [matrix of math nodes, row sep=3em,
column sep=3em, text height=1.5ex, text depth=0.25ex]
{T_c H^1 _{f,S} (G_{\Q ,T},U_n ) & T_{\loc _p (c)}H^1 _f (G_{\Q _p },U_n ) \\ 
H^1 _{f,S} (G_{\Q ,T},L_n ^c ) & H^1 _f (G_{\Q _p },L_n ^c ) \\ };
\path[->]
(m-1-1) edge[auto] node[auto] { $d\loc _p $ } (m-1-2)
edge[auto] node[auto] {  } (m-2-1)
(m-1-2) edge[auto] node[auto] { } (m-2-2)
(m-2-1) edge[auto] node[auto] {$\loc _p $} (m-2-2);
\end{tikzpicture}
\]
where $L_n ^c $ denotes the twist of $L_n $ by $c$ by the adjoint action of $U_n $ on $L_n$.
In particular, if the map $\loc _p :H^1 _{f,S} (G_{\Q ,T},L_n ^c )\to H^1 _f (G_{\Q _p} ,L_n ^c )$ is not surjective then  
 $\loc _p :H^1 _{f,S} (G_{\Q ,T},U_n )\to H^1 _f (G_{\Q _p} ,U_n )$ is not dominant in a neighbourhood of $c$.
\end{proposition}

\begin{lemma}
Let $U$ be a Galois stable quotient of $U_n $, with Lie algebra $L$. Let $U^{\dR }:=D_{\dR} (U)$.
Let $H$ be a $G_{\Q _p}$-stable subgroup of $U$. Let $W$ be an irreducible component of $\loc _p H^1 _{f,S} (G_{\Q ,T},U )\cap D_{\dR}(H)/F^0$. Then
\[
\dim W \leq \max _{c\in H^1 _{f,S} (G_{\Q ,T},U )}\dim \loc _p H^1 _{f,S} (G_{\Q ,T},L ^{c} ) \cap \lie (D_{\dR}(H ^{\loc _p c})) /F^0 ,
\]
where $\lie (H)^{\loc _p c} /F^0 $ is defined to be zero if $\loc _p c$ is not in the image in $H/F^0 $.
\end{lemma}
\begin{proof}
The dimension of $W$ is bounded by the generic dimension of its tangent space (i.e. the dimension of the tangent space at the generic point $\eta \in W$). Since the map
\[
T_c (\loc _p ^{-1}W)\to T_{\loc _p (c)}W
\]
is generically surjective, we have
\begin{align*}
\dim W& \leq \dim T_{\eta }(W)\\
& \leq \max _{c\in \loc_p^{-1}H/F^0 } \dim d\loc_p T_c H^1 _{f,S} (G_{\Q ,T},U )\cap T_{\loc_p(c)}(D_{\dR}(H)/F^0 ) \\
& = \max _{c\in H^1 _f (G_{\Q ,T},U  )}\loc _p H^1 _{f,S} (G_{\Q ,T},L ^{c} ) \cap \lie (D_{\dR}(H ^{\loc _p c})) /F^0 .
\end{align*}
 \end{proof}

By a \textit{virtual basepoint}, we shall mean a $\Q _p $-point $z\in \Res (\mathcal{X})(\mathbb{Z}_p )$ such that $j(z)$ lies in the image of $\Sel (U)$, together with a torsor $P\in \Sel (U)$ such that $\loc _p (P)=j(P)$. In particular, the notion of a virtual basepoint depends on a choice of quotient of $\pi _1 ^{\et ,\Q _p }(\Res (X)_{\overline{\Q }},b)$, but by Lemma \ref{indpt_basept} the property of $z\in \Res (\mathcal{X})(\mathbb{Z}_p )$ extending to a virtual basepoint is independent of the choice of basepoint. Given a virtual basepoint $(z,P)$, we obtain a $\Q _p $-unipotent group $U^{P}$ with an action of $G_{\Q ,T}$, and morphisms
\begin{align*}
j_{z,v} : & \Res (X)(\Q _v )\to H^1 (G_{\Q _v } ,U^{P}) \\
j_{z,p}: & \Res (X)(\Q _p )\to H^1 _f (G_{\Q _p },U^{P}) \\
j_z : & \mathcal{X}(\mathcal{O}_{K,S})\to \Sel (U^{P}).
\end{align*}
\begin{lemma}\label{virtual_lemma}
For every collection of local conditions such that $\mathcal{X}(\mathcal{O}_K \otimes \mathbb{Z}_p )_{\alpha }$ is non-empty, there is a virtual basepoint $(b,P) \in X(\Q _p )$ such that
\begin{equation}\label{virtual_waffle}
\mathcal{X}(\mathcal{O}_K \otimes \mathbb{Z}_p )_{\alpha }\subset j_b ^{-1}(\loc _p H^1 _{f,S} (G_{\Q ,T},U^{P})).
\end{equation}
\end{lemma}
\begin{proof}
If $\mathcal{X}(\mathcal{O}_K \otimes \mathbb{Z}_p )_{\alpha }$ is non-empty, then by definition there exists $P\in H^1 (G_{K,T},U_n (X))$ such that $\loc _v (P))=\alpha _v $ for all $v\in T_0 -S$ , and $\loc _p (P)=j(b)$ for some $b\in \Res (\mathcal{X})(\Q _p )$. Then, taking $(b,P)$ as a virtual basepoint, we deduce \eqref{virtual_waffle} from Lemma \ref{twisty_lemma}.
 \end{proof}
Often,  when we work with virtual basepoints $(b,P)$, we will simply write $U$ rather than $U^P$, and write the virtual basepoint simply as $b$.
\subsection{The unipotent Albanese morphism}\label{sec:unalb}
We now recall some properties of the morphism $\mathcal{X}(\mathcal{O}_v )\to H^1 _f (G_{K_v} ,U_n )$ when $v|p$ from \cite{kim:chabauty}. We refer to \cite{kim:chabauty} and the references therein for the background material regarding $p$-adic Hodge theory.
As we always take $p$ to be a prime which splits completely in $K$, we henceforth fix an isomorphism $K_v \simeq \Q _p $ and work over $\Q _p $.
Let $\mathcal{X}$ be a smooth curve over 
$\mathbb{Z}_p $. Fontaine's functor
$
D_{\cris }
$
sends continuous $\Q _p$ representations of $G_{\Q _p }$ to filtered $\phi $-modules over $\Q _p$. Recall that a filtered $\phi $-module over $\Q _p$ is a finite dimensional vector space $W$ over $\Q _p$, equipped with a $\Q _p $-linear Frobenius automorphism $\phi $, and a decreasing filtration $F^\bullet $ on $W$.

As explained in \cite{kim:chabauty}, Fontaine's functor induces functors $D_{\cris }$ and $D_{\dR}$ on unipotent groups over $\Q _p $ with a continuous action of $G_{\Q _p}$. The target of $D_{\cris }$ is the category of unipotent groups $U$ over $\Q _p$, together with an automorphism $\phi $ of $U$, and a filtration $\Fil $ by subgroups on $U$ (we shall refer to an object of this category as a filtered $\phi $-group). The target of $D_{\dR}$ is the category of unipotent groups over $\Q _p $ equipped with a filtration by subgroups. We say a $G_{\Q _p }$-equivariant $U$-torsor $P$ is \textit{crystalline} if it admits a $G_{\Q _p }$-equivariant trivialisation when base changed to $B_{\cris }$ (if $\mathcal{O}(U)$ is ind-crystalline, this is the same as saying that $\mathcal{O}(P)$ is ind-crystalline). The functor $D_{\cris}$ induces an equivalence of categories between crystalline $U$-torsors and filtered $\phi$-torsors over $D_{\cris }(U)$ (where a filtered $\phi $-torsor for a filtered $\phi $-group is simply a torsor with compatible filtration and $\phi $-action). By \cite[Proposition 1]{kim:chabauty}, if $U$ is crystalline and $D_{\cris }(U)^{\phi =1}=1$, this induces an isomorphism of $\Q _p $-schemes (a `non-abelian Bloch--Kato logarithm')
\[
H^1 _f (G_{\Q _p} ,U )\simeq D_{\dR}(U)/F^0 .
\]
Let $U^{\cris }(b)$ denote the Tannakian fundamental group of the category of unipotent isocrystals on $\mathcal{X}_{\mathbb{F}_p }$. This is a pro-unipotent group over $\Q _p$. Let $U^{\dR }(b)$ denote the Tannakian fundamental group of the category of unipotent flat connections on $X_{\Q _p}$. This is a pro-unipotent group over $\Q _p$. By Chiarellotto--Le Stum \cite[Proposition 2.4.1]{CLS99}, we have an isomorphism of $\Q _p $-group schemes
\[
U^{\cris }(b)\simeq U^{\dR} (b).
\]
Let $U_n ^{\cris }(b)$ and $U_n ^{\dR} (b)$ denote the respective maximal $n$-unipotent quotients. Then $U_n ^{\cris }(b)$ has the structure of a filtered $\phi $ group over $\Q _p$.
By Olsson's non-abelian comparison theorem \cite[Theorem 1.11]{olsson2011towards}, $U_n (b)$ is crystalline, and we have an isomorphism of filtered $\phi $ groups over $\Q _p $
\[
D_{\cris }(U_n (b))\simeq U_n ^{\cris }(b).
\]
Putting all this together, we obtain a locally analytic morphism
\[
j_{n,p}:\mathcal{X}(\mathbb{Z} _p)\to U_n ^{\dR}(b) /F^0 .
\]

\subsection{The universal connection}\label{universal_con}
Let $b\in \mathcal{X}(\mathbb{Z}_p )$ be a virtual base-point. The goal of this subsection is to describe the map $j_{n,p}$ in a formal neighbourhood of $b$, following Kim \cite{kim:chabauty}. Let $\mathcal{C}^{\dR}(X)$ denote the category of unipotent flat connections on $X$. A \textit{pointed flat connection} will be a flat connection $\mathcal{V}$ on $X$, together with an element $v\in b^* \mathcal{V}$.

\begin{definition}
(The depth $n$ universal connection, \cite{kim:chabauty} section 1,\cite{hadian2011motivic} section 2 ). The depth $n$ universal connection $\mathcal{E}_n $ on a pointed geometrically integral variety $(Z,b)$ is a pointed flat connection $(\mathcal{E}_n ,e_n )$ that is $n$-unipotent, such that for all $n$-unipotent flat connections $\mathcal{V}$, and $v\in b^* \mathcal{V}$, there exists a unique morphisms of connections $f: \mathcal{E}_n \to \mathcal{V}$ such that $b^* (f)(e_n )=v$. When we want to emphasise the dependence on $Z$, we write it as $\mathcal{E}_n (Z)$.
\end{definition}

\begin{lemma}\label{lemma:uce}
For all $n$, a universal $n$-unipotent pointed flat connection exists, and there is a canonical isomorphism
\[
\varprojlim b^* \mathcal{E}_n \simeq \mathcal{U}(\lie (\pi _1 ^{\dR}(Z,b))),
\]
where $\mathcal{U}(\lie (\pi _1 ^{\dR}(Z,b)))$ denotes the universal enveloping algebra of the Lie algebra of $\pi _1 ^{\dR}(Z,b)$.
\end{lemma}
\begin{proof}
More generally, we can define a depth $n$ universal object in any Tannakian category $(\mathcal{C},\omega )$ with as a pair $(\mathcal{V}_n ,v_n )$, where $\mathcal{V}_n $ is an $n$-unipotent object of $\mathcal{C}$, $v_n $ is an element of $\omega (\mathcal{V}_n )$, and for all $n$-unipotent objects $\mathcal{W}$, and $w\in \omega (\mathcal{W})$, there is a unique morphism $\mathcal{V}_n \to \mathcal{W}$ sending $v_n $ to $w$.

If $F:(\mathcal{C},\omega )\to (\mathcal{C}',\omega ')$ is an equivalence of Tannakian categories, then it sends universal $n$-unipotent objects to universal $n$-unipotent objects. In particular, if $\mathcal{C}$ is the category $\mathcal{C}^{\dR}(Z)$ of unipotent flat connections on $Z$, with fibre functor $b^*$, then $\mathcal{C}$ has a universal $n$-unipotent object if and only if the category of representations of $\pi _1 ^{\dR}(\mathcal{C}^{\dR}(Z),b )$, with fibre functor given by the forgetful functor, does.

Using the equivalence between $\pi _1 ^{\dR}(\mathcal{C}^{\dR}(Z),b )$-representations and $\mathcal{U}(\lie (\pi _1 ^{\dR}(Z,b)))$-modules, we see that $\mathcal{U}(\lie (\pi _1 ^{\dR}(Z,b)))/I^{n+1}$ is a universal $n$-unipotent object, where $I$ denotes the augmentation ideal.
 \end{proof}

The following Lemma, describes the morphisms of connections of the univeral unipotent connection corresponding to certain natual morphisms of $\pi _1 ^{\dR}(Z,b)$ modules.
\begin{lemma} \label{lemma:broken_link}
Let $\mathcal{E}_n $ be as above.
\begin{enumerate}
\item Let $x\in b^* \mathcal{E}_n$, and let $m(.,x):\mathcal{E}_n \to \mathcal{E}_n $ denote the unique morphism of connections which, in the fibre at $b$, sends $1$ to $x$. Then, for all $y\in b^* \mathcal{E}_n $,
\[
y\cdot x =b^* (m(,x))(y).
\]
More generally, for any unipotent flat connection $\mathcal{V}$, and any $v\in b^* \mathcal{V},x\in \varprojlim b^* \mathcal{E}_n $, the action of $x$ on $v$ is given by $b^* (f)(x)$, where $f:\varprojlim \mathcal{E}_n \to \mathcal{V}$ is the unique (pro-)morphism of connections sending $1$ to $v$.
\item Let 
\begin{equation}
\Delta :\varprojlim \mathcal{E}_n \to \varprojlim (\mathcal{E}_n \otimes \mathcal{E}_n )
\end{equation}
denote the unique morphism of pro-connections which, in the fibre at $b$, sends $1$ to $1\widehat{\otimes }1$. Then $b^* \Delta $ is equal to the co-multiplication on $\varprojlim b^* \mathcal{E}_n $.
\end{enumerate}
\end{lemma}
\begin{proof}
As in the proof of Lemma \ref{lemma:uce}, it is enough to check this with $\mathcal{C}^{\dR}(Z)$ replaced by the category of $\pi _1 ^{\dR}(Z,b)$-representations. Hence part (1) is immediate. For part (2), the co-multiplication is a morphism of (pro-)representations of $\pi _1 ^{\dR}(Z,b)$. Hence, by universal properties, it is uniquely determined by the fact that it sends $1$ to $1\widehat{\otimes }1$.
 \end{proof}

Let $Z'$ be an affine open of $b$, and let $\widehat{Z}$ denote the formal completion of $Z$ at $b$. The bundle $\mathcal{E}_n $ is unipotent, and hence admits a trivialisation 
\[
T:\mathcal{E}_n |_{Z'}\stackrel{\simeq }{\longrightarrow }\mathcal{O}_{Z'} \otimes b^* \mathcal{E}_n.
\]
With respect to this trivialisation, the connection $\nabla _n $ is given by $d-\Lambda $, for some $\Lambda \in \End (b^* \mathcal{E}_n ^{\dR} (b))\otimes _K \Omega _{Z|K}$.
The universal connection $(\mathcal{E}_n )$ carries a filtration by sub-bundles $(\varprojlim F^\bullet \mathcal{E}_n )$ satisfying the Griffiths transversality condition
\[
\nabla _n (F^i \mathcal{E}_n )\subset F^{i-1}\mathcal{E}_n \otimes \Omega ^1 _{Z|K}
\]
(see \cite[\S 3]{hadian2011motivic}). Replacing $Z'$ be a smaller affine open neighbourhood if necessary, we may choose a group-like section
$(F_n ) \in \varprojlim H^0 (Z',F^0 \mathcal{E}_n )$ - i.e. a section which satisfies 
\[
\Delta ((F_n )_n )=(F_n \widehat{\otimes }F_n )_n .
\]

The connection $\mathcal{E}_n $ admits a trivialisation on $\widehat{Z}$, i.e. an isomorphism of connections
\[
G:(\mathcal{E}_n |_{\widehat{Z}},\nabla )\simeq (b^* \mathcal{E}_n ^{\dR} \otimes _K \mathcal{O}(\widehat{Z}),d) 
\]
Via the trivialisation $G$, we may view $F_n$ as a function $Z'\to b^* \mathcal{E}_n $.
Via the trivialisation $T$, we obtain an endomorphism $(T|_{\widehat{Z}})\circ G^{-1}$ in $\End (b^* \mathcal{E}_n )\otimes \mathcal{O}(\widehat{Z}).$ This endomorphism sends $1$ to something in $1+I$ (recall that $I$ denotes the augmentation ideal of $\varprojlim b^* \mathcal{E}_n $), hence we have a well-defined element
\[
\boldsymbol{J}_{n,T}:=\log ((T|_{\widehat{Z}})\circ G^{-1}(1))\in Ib^* \mathcal{E} (Z)\widehat{\otimes }\mathcal{O}(\widehat{Z}),
\]
In words, $\exp (\boldsymbol{J}_{n,T})$ is a horizontal section of $(\mathcal{E}_n ,\nabla _n )$ on a formal neighbourhood of $b$, described with respect to the affine chart $T$.
In particular, by definition $\boldsymbol{J}_{n,T} (G)$ satisfies
\begin{equation}\label{expJ}
d\exp (\boldsymbol{J}_{n,T} (G))=\Lambda (\exp \boldsymbol{J}_{n,T}(G)).
\end{equation}
\begin{lemma}
The section $\exp(\boldsymbol{J}_{T}):= (\exp(\boldsymbol{J}_{n,T}))_n $ is group-like, i.e. it satisfies
\[
\Delta (\exp (\boldsymbol{J}_T ) )=(\exp (\boldsymbol{J}_T ))\widehat{\otimes }(\exp (\boldsymbol{J}_T )).
\]
\end{lemma}
\begin{proof}
Since they are both horizontal sections of $\nabla _{\mathcal{E}\widehat{\otimes }\mathcal{E}}$, the left hand side and right hand side are equal if and only if they are equal at one fibre, and at $b$ they both equal $1\widehat{\otimes }1$.
 \end{proof}
We deduce that $\boldsymbol{J}_{n,T}$ lies in $L_n ^{\dR }(Z)\widehat{\otimes }\mathcal{O}(\widehat{Z})$ (recall that $L_n ^{\dR} (Z)$ is defined to be the Lie algebra of $U_n ^{\dR} (Z)$). We also sometimes think of $\boldsymbol{J}_{n,T}$ as a morphism
\[
\boldsymbol{J}_{n,T}:\widehat{Z}\to L_n ^{\dR}(Z)
\]

\begin{remark}
At this point, there is no obvious reason for working with $\boldsymbol{J}_{n,T}$ rather than $\exp (\boldsymbol{J}_{n,T} )$. The reason for working with $\boldsymbol{J}_{n,T}$ is that, (as a consequence of \eqref{expJ}), $\boldsymbol{J}_{n,T}$ satisfies a particularly simple differential equation (see \eqref{rather_technical} and Lemma \ref{derivative_theta}) which is used in the proof of the unlikely intersection result needed for the proof of Theorem \ref{finiteness}.
\end{remark}

\begin{lemma}
Let $X$ be a smooth geometrically irreducible curve over $\Q _p$, and $b$ a $\Q _p $ point of $X$. Let $T$ be a trivialisation of $\varprojlim \mathcal{E}_n (T)$ as above. Then the unipotent Albanese morphism
\[
X(\Q _p )\to U_n ^{\dR}(X_{\Q _p }) /F^0
\]
is given, on a formal neighbourhood of $b$, by $\exp(\boldsymbol{J}_{n,T})\cdot F_n ^{-1}$.
\end{lemma}
\begin{proof}
This follows from Kim's explicit description of the map $j_n$ in \cite[\S 1]{kim:chabauty}. The map $j_n $ can be defined by sending $z$ to the class of $u$ in $U_n ^{\dR}/F^0 $, where $u\in U_n ^{\dR}$ is defined as follows: let $p_ \phi $ be the unique element of $P_n ^{\dR} (b,z)^{\phi =1}$, and choose $p_H \in F^0 P_n ^{\dR} (b,z)$, then define $u$ by $p_\phi =u\cdot p_H$. Although $u$ depends on the choice of $p_H$, it class in $U_n ^{\dR}/F^0 $ does not.

This is related to the parallel transport map as follows. The element $p^\phi (b,z)$ is an element of $\Hom (b^* \mathcal{E}_n ,z^* \mathcal{E}_n )$. When $b$ and $z$ are on the same residue disk, this homomorphism is given by (rigid analytic) parallel transport (see \cite{kim:chabauty}, above Lemma 4). Hence, in a formal neighbourhood of $b$, it is given by $\exp (\boldsymbol{J}_{n,T})$, and similarly $p_H$ is given by $F_n$, hence $u=\exp (\boldsymbol{J}_{n,T})\cdot F_n ^{-1}$.
 \end{proof}

\subsection{Higher Albanese manifolds}
Recall that, given a quasi-projective variety $Z$ over $\mathbb{C}$, we can define a higher Albanese manifold as follows \cite{hain1987unipotent} \cite{hain:higher_alb}. Let $U_n ^{\dR }(Z)$ denote the $n$-unipotent de Rham fundamental group of $Z$. Let $U_n ^{\Be }(Z)$ denote the $\Q $-unipotent Betti fundamental group of $Z$ at $b$, i.e. the maximal $n$-unipotent quotient of the $\Q $-unipotent completion of $\pi _1 (Z(\mathbb{C}),b)$. Abusing notation, we will denote by $U_n ^{\Be }(Z)(\mathbb{Z})$ the image of $\pi _1 (Z(\mathbb{C}),b)$ in $U_n ^{\Be }(Z)(\mathbb{Q})$. By the Riemann-Hilbert correspondence, we have an isomorphism of unipotent groups over $\mathbb{C}$
\[
U_n ^{\Be }(Z)_{\mathbb{C}}\simeq U_n ^{\dR}(Z).
\]
The $n$th higher Albanese manifold of $Z$ is the double quotient 
\[
\mathcal{H}_n :=U_n ^{\Be }(Z)(\mathbb{Z})\backslash U_n ^{\dR}(Z)(\mathbb{C})/F^0 U_n ^{\dR}(Z)(\mathbb{C}).
\]
There is an $n$th higher Albanese map (see \cite[\S 5]{hain1987unipotent})
\[
j_n ^{\Be} :Z(\mathbb{C})\to \mathcal{H}_n ,
\]
which is a map of complex manifolds. 
The formal completion of $\mathcal{H}_n $ at the identity is isomorphic to the formal completion of $U_n ^{\dR }(Z)(\mathbb{C})/F^0 U_n ^{\dR}(Z)(\mathbb{C})$ at the identity, hence we have a morphism of formal schemes
\[
\widehat{j}_n ^{\Be }:\widehat{Z}_{\mathbb{C}}\to \widehat{U_n ^{\dR} /F^0 }_{\mathbb{C}}.
\]
\begin{lemma}
The formal completion of $j_n ^{\Be }$ is given by
is given by 
\[
\widehat{j}^{\Be }_n =\exp (\boldsymbol{J}_{n,T})\cdot F_n ^{-1},
\] where $\boldsymbol{J}_{n,T}$ and $F_n$ are as in section \ref{universal_con}.
In particular, for any isomorphism $\overline{\Q }_p \simeq \mathbb{C}$ we obtain an identification of the formal completions of $j_{n,p}$ at $b$ (base changed to $\mathbb{C}$ and the formal completion of $j_n ^{\Be }$ at $b$,
\end{lemma}
\begin{proof}
This follows from the description of the unipotent Albanese morphism given in \cite[\S 5]{hain1987unipotent} and \cite[Proposition 3.2, (3.3)]{hain:higher_alb}.
 \end{proof}
We deduce the following Lemma.
\begin{lemma}\label{formal_nbd_ok}
Let $i:Z\hookrightarrow \Res (X)_{\overline{\Q }_p }$ be an irreducible subvariety, and $b\in Z(\overline{\Q }_p )\cap \Res (X)(\Q _p )$. Let $\widehat{Z}$ and $\widehat{\Res }(X)_{\overline{\Q }_p }$ denote the formal completions at $b$. Then the restriction of $j_{n,p}$ to $\widehat{Z}$ lands in the image of $\widehat{U}_n ^{\dR} (Z)/F^0 $ in $\widehat{U}_n ^{\dR}(\Res (X))/F^0 $, and 
the diagram
\[
\begin{tikzpicture}
\matrix (m) [matrix of math nodes, row sep=3em,
column sep=3em, text height=1.5ex, text depth=0.25ex]
{\widehat{Z}_b & \widehat{U}_n ^{\dR}(Z)/F^0  \\ 
\widehat{\Res }(X)_b & \widehat{U}_n ^{\dR}(\Res (X))/F^0.  \\ };
\path[->]
(m-1-1) edge[auto] node[auto] { $\exp (\boldsymbol{J}_{n,T})\cdot F_n ^{-1}$ } (m-1-2)
edge[auto] node[auto] {} (m-2-1)
(m-2-1) edge[below] node[auto] {$\widehat{j}_n $  } (m-2-2)
(m-1-2) edge[auto] node[auto] {} (m-2-2);
\end{tikzpicture}
\]
commutes. In particular, the pre-image of the graph of $j_n $ in $\widehat{Z}\times \widehat{U}_n ^{\dR}/F^0 $ under the map
\begin{align*}
\widehat{Z}\times L_n ^{\dR}(Z) & \to \widehat{\Res }(X)\times \widehat{U}_n ^{\dR}/F^0 \\
(z,x) & \mapsto (i(z),\exp (x) \cdot  F_n ^{-1}(z))
\end{align*}
is equal to the graph of $\boldsymbol{J}_{n,T}$.
\end{lemma}
\begin{proof}
This follows from the previous lemma, together with the fact that Hain's higher Albanese morphism is functorial in complex manifolds \cite[\S 3]{hain1987unipotent}.
 \end{proof}

\section{Unlikely intersections among zeroes of iterated integrals}\label{sec:the_beef}
The aim of this section is to prove the following proposition, which is the unlikely intersection result needed to establish Theorem \ref{finiteness}. Given a scheme $V$, and a formal subscheme $W$ of the formal completion of $V$ at a point $x$, we say that $W$ is Zariski dense if there is no proper closed subscheme of $V$ whose formal completion at $x$ contains $W$.

\begin{proposition}\label{unlikely_intersection}
Let $Z$ be a smooth irreducible subvariety of $\Res (X)_{\overline{\Q }_p }$.
Let $V$ be an irreducible subvariety of $L_n ^{\dR} (Z) \times Z$, containing $(0,b)\in L_n ^{\dR} (Z)\times Z$, such that the projection of $V$ to $Z$ is dominant. Let $\widehat{V}$ and $\widehat{L}_n ^{\dR} (Z)$ denote the formal completions of $V$ and $L_n ^{\dR}(Z)$  at $b$ and $0$ respectively. Let $\Delta \subset \widehat{L}_n ^{\dR} (Z)\times \widehat{Z}$ denote the graph of $\boldsymbol{J}_{n,T}$.
Let $W=\Delta \cap \widehat{V}$. Suppose $W$ is irreducible and 
\[
\codim _{\Delta } W <\codim _{L_n ^{\dR} (Z) \times Z}V .
\]
Then $W$ is not Zariski dense in $V$.
\end{proposition}

For example, we can apply this result when $V$ is the Zariski closure of $\Delta$, and $W=\Delta $. Then $W$ is Zariski dense in $V$ by definition, and hence we deduce
\begin{equation}\label{graph_is_Zariski_dense}
0=\codim _{\Delta }(W)=\codim _{L_n ^{\dR}(Z)  \times Z}V.
\end{equation}
(i.e. the graph of $\boldsymbol{J}_{n,T}$ is Zariski dense in $L_n ^{\dR}(Z) \times Z$). This could also more elegantly be proved following the topological argument in \cite{faltings2007mathematics}.

Proposition \ref{unlikely_intersection} can be informally thought of as saying that algebraic relations between $n$-unipotent iterated integrals must have a geometric explanation.
In the $1-$unipotent (or abelian) case, this is due to Ax \cite{ax1972some}. The idea of the general proof is inspired by Ax's approach: we inductively show that non-trivial algebraic relations between $n$-unipotent iterated integrals come from geometry by differentiating them to produce relations between $(n-1)$-unipotent iterated integrals (the difficult part being to show that these relations must also be non-trivial).

We also note that this Proposition translates into a criterion for finiteness of $\Res (\mathcal{X})(\mathbb{Z}_p )_n $ which will be used repeatedly in subsequent sections. To state the criterion, we introduce the following notation: for $(Z,z)\to (Y,y)$ a morphism of varieties over a field $L$ sending $z$ to $y$, we define $U_n (Z/Y) \subset U_n (Y)=U_n (Y)(y)$ to be the image of $\pi _1 ^{\et ,\Q _p }(Z_{\overline{L}},z)$ in $U_n (Y)$. For $i\leq n$, we define
\[
\gr _i (U_n (Z/Y)):=\Ker (U_i (Z/Y)\to U_{i-1}(Z/Y)),
\]
and similarly for $U_i ^{\dR}(Z/Y)$. Note that, in general, $\gr _i (U_i (Z/Y))$ is not the same as $C_i (U_i (Z/Y))/C_{i+1}(U_i (Z/Y))$. Since the homomorphism
\[
\pi _1 ^{\et ,\Q _p }(Z_{\overline{L}},z)\to \pi _1 ^{\et ,\Q _p }(Y_{\overline{L}},y)
\]
respects the central series filtrations, we have an induced morphism
\[
U_n (Z,z)\to U_n (Y,y),
\]
however this morphism need not be strict with respect to the central series filtration. For example, this non-strictness occurs when $Z=E-\{x_1 ,x_2 \}$, for $E$ an elliptic curve with points $x_1 $ and $x_2 $, $Y=(E-\{ x_1 \} )\times (E-\{ x_2 \} )$, and the map is the diagonal embedding $Z\to Y$. However if $X$ is projective morphisms will be strict, since the weight filtration on $U_n $ will agree with the central series filtration.
\begin{proposition}\label{prop:useful_criterion}
Suppose that, for every irreducible subvariety $Z\subset \Res (X)_{\Q  _p }$, we have
\begin{equation}\label{big_codim}
\codim _{U_n ^{\dR}(Z/\Res (X))/F^0 }(\loc _p (\Sel (U_n ))\cap U_n ^{\dR}(Z/\Res (X))/F^0  )\geq \dim Z.
\end{equation}
Then $\Res (X)(\Z _p )_n$ is finite.
\end{proposition}
\begin{proof}
Let $\widetilde{Z}\subset \Res (X)_{\Q_p }$ denote the Zariski closure of $\Res (\mathcal{X})(\mathbb{Z}_p )_n \cap ]b[$. Let $\widetilde{Z}=Z_0 \sqcup Z_1 \sqcup \ldots \sqcup Z_N$ be a stratification of $\widetilde{Z}_{\Q _p }$ into smooth irreducible subvarieties. Since $\widetilde{Z}$ is the Zariski closure of a set of $\Q _p $-points of $\Res (X)_{\Q _p }$, the irreducible components are geometrically irreducible.
Let $Z=Z_i$ be the element of the stratification containing the point $b$. To prove finiteness of $\Res (\mathcal{X})(\mathbb{Z}_p )_n$, it will be enough to prove that $Z=\{b \}$. 

Let $W_0 $ denote the formal completion of $\Res (\mathcal{X})(\mathbb{Z}_p )_n \cap Z$ at $b$.
Recall that, by Lemma \ref{formal_nbd_ok}, the pre-image of the graph of $j_n $ in $\widehat{Z}_b \times \widehat{L}_n ^{\dR}(Z)$ is given by the graph of $\boldsymbol{J}_{n,T}$.

We apply Proposition \ref{unlikely_intersection} with $V$ being the product of $Z$ with the pre-image of $\loc _p (\Sel (U_n ))\cap U_n ^{\dR}(Z/Y)/F^0 $ in $L_n ^{\dR}(Z)$. Since we are free to choose $b$ outside a closed analytic subvariety of smaller dimension, we may assume that $W:=V\cap \Delta $ is an irreducible formal scheme (recall $\Delta $ is the graph of $\boldsymbol{J}_{n,T}$) and hence apply Proposition \ref{unlikely_intersection} to deduce
\begin{equation}\label{usefulprop_ineq2}
\codim _{\Delta }(W)\geq \codim _{L_n ^{\dR}(Z)}V.
\end{equation}
On the other hand by definition of $V$ we have
\begin{equation}\label{usefulprop_ineq1}
\codim _{Z\times L_n ^{\dR}(Z)}(V)\geq \codim _{U_n ^{\dR}(Z/Y)/F^0 }(\loc _p (\Sel (U_n ))\cap U_n ^{\dR}(Z/Y)/F^0  ).
\end{equation}

Since $j(W_0 ) \subset \loc _p (\Sel (U_n ))$, we have $W_0 \subset W$, and 
\eqref{big_codim}, \eqref{usefulprop_ineq1} and \eqref{usefulprop_ineq2} together imply the Lemma.
 \end{proof}

\subsection{Universal connections and reduced form}
In the proof of Proposition \ref{unlikely_intersection}, it will be useful to have a fairly explicit description of the map $\boldsymbol{J}_{n,T}$, and hence of the connection $(\mathcal{E}_n ,\nabla _n )$. For this, we introduce the notion of the \textit{reduced form} of a connection, which is inspired by Kim's description of the universal connection \cite{kim:chabauty}. As above, $Z$ is a smooth irreducible affine subvariety of $\Res (X)_{\overline{\Q }_p }$. Let $d_n$ denote the rank of $\mathcal{E}_n $, and let $r_n$ denote the dimension of $U_n ^{\dR}(Z)$.

\begin{definition}[Reduced form]
Let $S$ be a complement of $d\mathcal{O}(Z)$ in $H^0 (Z,\Omega )$.

A $n$-nilpotent matrix $M\in \End (\oplus _{i=0}^n V_i )\otimes H^0 (Z,\Omega )$ (for $V_0 ,\ldots ,V_n$ vector spaces over $K$) is in reduced form (relative to the complement $S$) if it is block $n$-nilpotent and all of its entries lie in $S$.

An $n$-unipotent connection $(\mathcal{V},\nabla )$ in reduced form is an isomorphism $\mathcal{V}\simeq \oplus _{i=0}^n\mathcal{O}_Z \otimes V_i $ with respect to which $\nabla =d+\Lambda $, where $\Lambda \in \End (\oplus V_i )\otimes H^0 (Z,\Omega )$ is an $n$-nilpotent matrix in reduced form.
\end{definition}

\begin{lemma}\label{lemma:reduced}
Let $\mathcal{V}$ be an $n$-unipotent connection on a smooth affine variety $Z$ over a field $K$ of characteristic zero. Then $\mathcal{V}$ can be written in reduced form. All morphisms $f$ between connections $(\oplus V_i \otimes \mathcal{O}_Z,\nabla )$, $(\oplus W_i \otimes \mathcal{O}_Z ,\nabla ')$ in reduced form are $K$-linear, i.e. $f\in \Hom (\oplus V_i ,\oplus W_j )\otimes \mathcal{O}(Z)$ actually has entries in $\Mat _{n,m}(K)$.
\end{lemma}
\begin{proof}
This may be proved in a similar fashion to \cite[Lemma 2]{kim:chabauty}. We argue by induction on $n$, the case $n=0$ being immediate. If $\mathcal{V}$ is an extension of an $(n-1)$-unipotent connection $\mathcal{V}'$ by a trivial connection $\mathcal{V}'' =(V'' \otimes \mathcal{O},d)$, then we can assume $\mathcal{V}'$ can be written in reduced form, say 
\[
\mathcal{V}'\simeq \oplus _{i=0}^{n-1}\mathcal{V}_i \simeq \oplus _{i=0}^{n-1}V_i \otimes \mathcal{O}
\] 
so that $\nabla _{\mathcal{V}'}=d+\Lambda $, where $\Lambda =\sum _{0\leq i<j<n} \Lambda _{ij}$ and $\Lambda _{ij}\in \Hom (V_i ,V_j )\otimes S$.
Since $Z$ is affine, we can choose a vector bundle splitting of the short exact sequence
\begin{equation}\label{induction_exact_seq}
0\to \mathcal{V}'' \to \mathcal{V}\to \mathcal{V}'\to 0,
\end{equation}
so that we can write $\mathcal{V}$ as $\mathcal{V}' \oplus \mathcal{V}''$, with connection given by $d+\Lambda  +\Lambda ' $, where $\Lambda '\in \Hom (\mathcal{V}',\mathcal{V}'')\otimes H^0 (Z,\Omega )$. We want to show that, changing the basis by an element of $1+\Hom (\mathcal{V}',\mathcal{V}'')$, we can make $\Lambda '$ of the form $\sum \Lambda ' _i $, where 
\begin{equation}\label{s_i}
\Lambda ' _i \in \Hom (V_i ,V'')\otimes S.
\end{equation}
Write $\Lambda '$ as $\sum \Lambda ' _i $, with $\Lambda ' _i $ in $\Hom (V_i ,V'')\otimes H^0 (Z,\Omega )$. If we change the splitting of \eqref{induction_exact_seq} by $1+\sum M_i $, where $M_i \in \Hom (\mathcal{V}_i ,\mathcal{V}'')$, this will change $\Lambda '$ by
\begin{equation}\label{the_obligatory_gauge_transformation}
\Lambda _i ' \mapsto \Lambda _i ' +dM_i -\sum _{j>i}M_j \circ \Lambda _{ij}.
\end{equation}
We show that we can change our basis so that the $\Lambda '_i$ satisfy \eqref{s_i} by descending induction on $i$. When $i=n-1$, \eqref{s_i} implies that there is an 
$M\in \Hom (\mathcal{V}_{n-1},\mathcal{V}'')$ such that
\[
\Lambda _{n-1}'-dM\in \Hom (V_{n-1} ,V'')\otimes S .
\]
Changing the splitting of \eqref{induction_exact_seq} by $M$, we obtaining a splitting for which $\Lambda _{n-1}'$ satisfies \eqref{s_i}. Now suppose we have a splitting of \eqref{induction_exact_seq} such that $\Lambda _{n-1}',\ldots ,\Lambda _{n-i+1}'$ satisfy \eqref{s_i}. Then \eqref{s_i} implies there is an $M\in \Hom (\mathcal{V}_{n-i},\mathcal{V}'')$ such that
\[
\Lambda _{n-i}'-dM\in \Hom (V_{n-i} ,V'')\otimes S .
\]
If we change the splitting of \eqref{induction_exact_seq} by $M$, then by \eqref{the_obligatory_gauge_transformation}, $\Lambda _{n-1}',\ldots,\Lambda _{n-i+1}$ will be unchanged, and $\Lambda _{n-i}$ will now lie in $\Hom (V_{n-i} ,V'')\otimes S$. Hence $\mathcal{V}$ is in reduced form.
 
For the second part of the lemma, write the homomorphism $f$ as $\sum f_{ij}$, where $f_{ij}\in \Hom (\oplus V_i ,\oplus W_j )\otimes \mathcal{O}(Z)$. Write $\nabla =d+\sum _{i<j}\Lambda _{ij}$ and $\nabla '=d+\sum _{i<j}\Lambda _{ij} '$ with the notation as above. Then the identity
\[
df_{ij}+\sum _{k<j}\Lambda _{kj} ' f_{ik}=\sum _{i<k}f_{kj}\Lambda _{ik},
\]
and the fact that $S\cap d\mathcal{O}(Z)=0$, shows $df_{ij}=0$ for all $i$ and $j$ by a double induction on $i$ and $j$.
\end{proof}
\begin{definition}
We write $\mathcal{E}_n $ in reduced form as $\oplus _{i=0}^{d_n }\mathcal{O}_Z \cdot e_i $ (i.e. we write it in reduced form, and then pick a basis for each vector space $V_i $), and $\nabla _n =d+\Lambda _n $. The matrix $\Lambda _n $ is $n$-nilpotent, and its upper left $k$-nilpotent submatrix  is equal to $\mathcal{E}_k $, for all $k<n$ (i.e. the quotient $\mathcal{E}_n \to \mathcal{E}_k $ is just given by projecting onto the first $(k+1)$-blocks).
If $T$ is a trivialisation of $\mathcal{E}_n $ coming from its reduced form, we denote $\boldsymbol{J}_{n,T}$ simply by $\boldsymbol{J}_n$.
\end{definition}
\subsection{An explicit description of the universal connection}
The main result of this section is a differential equation satisfied by $\boldsymbol{J}_{n}$, which is used in the inductive step of the proof of Proposition \ref{unlikely_intersection}.
Let $\mathcal{E}_n \simeq \mathcal{O}_Z ^{\oplus d_n}$ be a bundle trivialisation of $\mathcal{E}_n $ putting it in reduced form. Let $\Lambda _n $ be the connection matrix for its reduced form. Recall that by Lemma \ref{lemma:uce}, $\varprojlim b^* \mathcal{E}_n $ is isomorphic via Tannaka duality to the universal enveloping algebra of $\pi _1 ^{\dR}(Z,b)$, viewed as a pro-representation of $\pi _1 ^{\dR} (Z,b)$.
Recall that the Lie algebra $L_\infty ^{\dR} (Z) \subset \varprojlim b^* \mathcal{E}_n $ is defined to be the subspace satisfying $\epsilon (x)=0$ and $\Delta (x)=x\widehat{\otimes }1+1\widehat{\otimes }x$, and $L_n ^{\dR} (Z)\subset b^* \mathcal{E}_n $ is the image of $\lie (\pi _1 ^{\dR} (Z,b))$ in $b^* \mathcal{E}_n $. By making a $K$-linear change of basis (preserving the unipotent filtration) if necessary, we may (and do) henceforth assume that there is a sub-basis $e_{l_i }$ of  $e_i$ forming a basis of $L_n ^{\dR} (Z)$.
Define $b_{ijk}\in K,1\leq i,j,k\leq r_n $, by
\begin{equation}\label{bijk}
[e_{l_i} ,e_{l_j} ]=\sum b_{ijk}e_{l_k} .
\end{equation}

\begin{lemma}\label{omega}
There is an element $\boldsymbol{\omega }=\sum _{i=1}^{r_n } e_{l_i} \otimes \omega _i $ of $L_n ^{\dR} (Z)\otimes H^0 (Z,\Omega )$ such that the connection $\nabla _n $ on $\mathcal{E}_n $ is given by 
\[
v \mapsto \boldsymbol{\omega}\cdot v+dv .
\]
\end{lemma}
\begin{proof}
First, we show that there is an element $\boldsymbol{\omega }\in H^0 (Z,\mathcal{E}_n \otimes \Omega ^1 )$ such that for all $v$ in $b^* \mathcal{E}_n $, $\nabla (v \otimes 1 )=\boldsymbol{\omega }\cdot v$. Define $\boldsymbol{\omega }:=\nabla (1_n )$. Let $f $ be the unique morphism $\mathcal{E}_n \to \mathcal{E}_n $ such that $b^* (f)(1_n )=v$. Then, since $f$ is $K$-linear with respect to the trivialisation, we have
\[
\nabla (v\otimes 1 )=f (\boldsymbol{\omega }).
\]
By Lemma \ref{lemma:broken_link}, $f (\boldsymbol{\omega })=\boldsymbol{\omega }\cdot v $.

We now show that $\boldsymbol{\omega }$ is in $L_n ^{\dR} (Z) \otimes H^0 (Z,\Omega )$. It is convenient to instead show this in the limit over $n$: i.e. that $\boldsymbol{\omega }\in \varprojlim H^0 (Z,\mathcal{E}_n \otimes \Omega ^1 )$ lies in $\varprojlim H^0 (Z,L_n ^{\dR} (Z)\otimes \Omega ^1 )$. Since, in $\mathcal{E}_0$, we have $\nabla (1_0 )=0$, the first condition $\boldsymbol{\omega }\in \Ker(\epsilon \otimes 1)$ is satisfied. The second condition is that $\Delta (\boldsymbol{\omega })=\boldsymbol{\omega }\widehat{\otimes }1_n +1_n \widehat{\otimes }\boldsymbol{\omega }$. By Lemma \ref{lemma:broken_link}, $\Delta $ is the unique morphism of pro-connections
\[
\varprojlim \mathcal{E}_n \to \varprojlim (\mathcal{E}_n \otimes \mathcal{E}_n )
\]
sending $(1_n )$ to $(1_n \otimes 1_n )$. Hence $\Delta (\boldsymbol{\omega })=\nabla (1_n \widehat{\otimes }1_n )=\boldsymbol{\omega }\widehat{\otimes }1_n +1_n \widehat{\otimes }\boldsymbol{\omega }$, by definition of the tensor product of two connections.
 \end{proof}
We write $\boldsymbol{J}=\sum J_i \otimes e_{l_i} \in \mathcal{O}(\widehat{Z})\otimes L_n ^{\dR} (Z)$. Recall $\boldsymbol{J}(b)=0$ and
\begin{equation}
d\exp (\boldsymbol{J})=\boldsymbol{\omega }\cdot \exp (\boldsymbol{J})
\end{equation}
in $H^0 (\widehat{Z},\Omega )$.
Let $t_i \in \mathcal{O}(\widehat{L}_n ^{\dR} (Z))$ be the $i$th coordinate function with respect to the basis $(e_{l_i} )$. Define 
\[
\boldsymbol{t}=\sum t_i \otimes e_{l_i}\in H^0 (L_n ^{\dR}(Z),\mathcal{O})\otimes L_n ^{dR} (Z).
\]
Recall $\Delta \subset \widehat{Z}\times \widehat{L}_n ^{\dR}(Z)$ denotes the graph of $\boldsymbol{J}$.
Then $\Delta $ is the zero set of the functions $t_i -J_i $ (here we view $t_i $ and $J_i$ as functions on $\widehat{Z}\times \widehat{L}^{\dR} (Z)$ via pulling back to ease notation we suppress the pull-back from the notation).

To describe the derivative of $\boldsymbol{J}$, we make use of Poincar\'e's Lemma on the derivative of the exponential function. Given an $N$-nilpotent Lie algebra $L$, and a formal power series $F=\sum _{i\geq 0}a_i t^i $, and $x\in L$, we write $F(\ad _x )$ to mean the operator $\sum _{i=0}^N a_i (\ad _x )^i $ on $L$. Finally, if $F,G$ are nonzero power series in $K[\! [t]\! ]$ such that $H:=F/G$ lies in $K[\! [t]\! ]$, we further abuse notation by writing $F(\ad _x )/G(\ad _x )$ to mean the operator $H(\ad _x )$.
\begin{lemma}\label{poincare}
Let $\exp :L_n ^{\dR} (Z)\to b^* \mathcal{E}_n (Z) $ denote the exponential function. Then
\[
d\exp (x)=\frac{e^{\ad _x }-1}{\ad _x }(dx)\cdot e^{x}.
\]
\end{lemma}
\begin{proof}
It will be enough to prove this identity in the universal enveloping algebra of $L_n ^{\dR}(Z)$, as $b^* \mathcal{E}_n (Z)$ is a quotient of the universal enveloping algebra of $L_n ^{\dR}(Z)$, compatible with the exponential and logarithm maps. Hence this follows from the usual version of Poincar\'e's formula (see e.g. \cite[II.5.11]{CRS95}).
 \end{proof}

We define $\widetilde{\theta }_i \in H^0 (L_n ^{\dR} (Z) \times Z,\Omega )$ by
\[
\sum \widetilde{\theta } _i \otimes e_{l_i} =\widetilde{\boldsymbol{\theta }}:=\frac{e^{\ad _{\boldsymbol{t}}}-1}{\ad _{\boldsymbol{t}}}d\boldsymbol{t}-\boldsymbol{\omega }.
\]
We define $\theta _i $ to be the image of $\widetilde{\theta }_i $ in $H^0 (V,\Omega )$ and $\overline{\theta }_i $ to be the image of $\theta _i $ in $H^0 (V,\Omega )\otimes \mathcal{O}(W)$. 

\begin{remark}\label{rk:BSCFN}
The $\theta _i $ have the following interpretation in terms of foliations on principal bundles. The frame bundle on $\mathcal{E}_n ^{\dR}$ descends to a $U_n ^{\dR}$-bundle $P$. The choice of trivialisation on $\mathcal{E}_n ^{\dR}$ on $Z$ defines a trivialisation of $P$, i.e. an isomorphism $P\simeq U_n ^{\dR} \times Z$. With respect to this isomorphism, the connection form $\Omega $ on $P$ can be viewed as an element of $H^0 (Z\times U_n ^{\dR},\Omega )\otimes L_n ^{\dR}$. As explained in \cite[2.7]{BSCFN}, the connection form is given by $\boldsymbol{t}^{-1}d\boldsymbol{t}-\boldsymbol{t}^{-1}\boldsymbol{\omega }\boldsymbol{t}$. Via the exponential map, this may be viewed as an element $\Omega '$ of $H^0 (Z \times L_n ^{\dR},\Omega )\otimes L_n ^{\dR}$. Then we see that 
\[
\Omega ' =\boldsymbol{t}^{-1}\widetilde{\boldsymbol{\theta }}\boldsymbol{t}.
\]
Hence finding linear relations between the $\theta _i$ is equivalent to finding linear relations between the coefficients of $\Omega $, as in the proof of Theorem 3.6 of loc. cit.
\end{remark}
We have an exact sequence \cite[20.7.20]{egaiv}
\begin{equation}\label{ega}
I/I^2 \to \widehat{\Omega }_{\widehat{V}|\overline{\Q }_p }\otimes \mathcal{O}(W)\to \widehat{\Omega }_{W|\overline{\Q }_p }\to 0.
\end{equation}
\begin{lemma}\label{generatorsI/I2}
Let $I\subset \mathcal{O}(\widehat{V})$ denote the ideal of functions vanishing on $W$. Then the image of $I/I^2 $ in $\widehat{\Omega } _{V|K}\otimes \mathcal{O}(W)$ under
\[
I/I^2 \to \widehat{\Omega } _{\widehat{V}|K}\otimes \mathcal{O}(W)
\] is spanned by $\overline{\theta } _1 ,\ldots ,\overline{\theta } _{r_n }$.
\end{lemma}
\begin{proof}
The image of $I/I^2$ is spanned by the functions $dt_1 -dJ_1,\ldots ,dt_{r_n }-dJ_{r_n }$, so it will be enough to show that the submodule spanned by these differentials is equal to the submodule spanned by $\theta _1 ,\ldots ,\theta _{r_n }$.
By \eqref{expJ} and Lemma \ref{omega}, $\exp (\boldsymbol{J})$ satisfies
\[
d(\exp (\boldsymbol{J}))=\boldsymbol{\omega } \cdot \exp (\boldsymbol{J}).
\]
Hence, by Lemma \ref{poincare}, the function $\boldsymbol{J}$ satisfies the differential equation
\begin{equation}\label{rather_technical}
d\boldsymbol{J}=\frac{\ad _{\boldsymbol{J}}}{e^{\ad _{\boldsymbol{J}}}-1}(\boldsymbol{\omega }).
\end{equation}
Hence the image of $I/I^2$ is equal to the submodule spanned by
\begin{equation}\label{silly}
d\boldsymbol{t}-\frac{\ad _{\boldsymbol{t}}}{e^{\ad _{\boldsymbol{t}}}-1}(\boldsymbol{\omega }).
\end{equation}
since $t_i =J_i $ on $W$. Finally, the map
\[
\boldsymbol{v}\mapsto \frac{e^{\ad _{\boldsymbol{t}}}-1}{\ad _{\boldsymbol{t}}}(\boldsymbol{v})
\]
is an $\mathcal{O}(W)$-linear automorphism of $L_n ^{\dR}(Z)  \otimes \mathcal{O}(W)$, hence the submodule spanned by the coordinates of $\boldsymbol{\theta }$ is equal to the submodule spanned by the coordinates of
$
\frac{\ad _{\boldsymbol{t}}}{e^{\ad _{\boldsymbol{t}}}-1}\cdot \boldsymbol{\theta } ,
$
which equals \eqref{silly}.
 \end{proof}

\subsection{Proof of Proposition \ref{unlikely_intersection}}
Since the singular locus of $V$ is not Zariski dense in $V$, we may restrict to the case where $b$ is a smooth point of $V$. We suppose that the codimension of $\Delta $ in $W$ is less than the codimension of $V$ in $L_n ^{\dR}(Z)\times Z$. Let $\overline{\Q }_p (V)$ denote the function field of $V$, and $\overline{\Q }_p (W)$ the function field of $W$. By the exact sequence \eqref{ega}
\[
I/I^2 \otimes \overline{\Q }_p (W) \to \widehat{\Omega }_{\widehat{V}_{\overline{\Q} _p }|\overline{\Q }_p }\otimes \overline{\Q }_p (W)\to \widehat{\Omega } _{\overline{\Q }_p (W)|\overline{\Q }_p } \to 0
\]
and the inequality $\dim (W)\leq \rk \widehat{\Omega }_{W|\overline{\Q }_p }$, we deduce that $\overline{\theta }_1 ,\ldots ,\overline{\theta }_{r_n }$ are linearly dependent in $\widehat{\Omega }_{\widehat{V}|\overline{\Q }_p }\otimes \overline{\Q }_p (W) $.
We henceforth assume that $W$ is Zariski dense in $V$, and assume, as hypothesis for contradiction, that the $\overline{\theta } _i $ are $\overline{\Q }_p (W)$-linearly dependent. By our assumption that $W$ is Zariski dense in $V$, this implies that the $\theta _i $ are $\overline{\Q }_p (V)$-linearly dependent. In this subsection, we aim to show that such a dependence contradicts the Zariski density of $W$.
The following elementary Lemma gives a couple of ways to prove that a formal sub-scheme is not Zariski dense.
\begin{lemma}\label{no_diffl}
Let $V$ be an integral variety over $\overline{\Q }_p $, $W$ an integral closed formal subscheme of the formal completion of $V$ at a $\overline{\Q }_p $-point $b$ at which $V$ is smooth.
\begin{enumerate}
\item Suppose there are $h_1 ,h_2$ in $\overline{\Q }_p (V)$ such that $h_1 dh_2$ is zero in $\Omega _{\overline{\Q }_p (W)|\overline{\Q }_p }$, but non-zero in $\Omega _{\overline{\Q }_p (V)|\overline{\Q })_p }$. Then $W$ is not Zariski dense in $V$.
\item Let $M$ be an $\mathcal{O}_V$-submodule of $\mathcal{O}_V ^{\oplus r}$. Suppose the $\overline{\Q }_p (W)$-rank of the image of $M\otimes \overline{\Q }_p (W)$ in $\overline{\Q }_p (W)^{\oplus r}$ is less than the $\overline{\Q }_p (V)$-rank of $M\otimes \overline{\Q }_p (V)$. Then $W$ is not Zariski dense in $V$.
\end{enumerate}
\end{lemma}
\begin{proof}
\begin{enumerate}
\item If $h_1 dh_2 =0$ then either $h_1 =0$ or $dh_2 =0$. Write $h_i =f_i /g_i $, with $f_i ,g_i \in H^0 (V,\mathcal{O})$. Then in the first case $W$ is contained in the zero locus of $f_1 $. In the second case $h_2 $ is constant on $W$, say equal to $\lambda $, and hence $W$ is contained in the zero locus of $f_2 -\lambda g_2 $.
\item Suppose the generic rank of $M$ is $s$, and $m_1 ,\ldots ,m_s $ are generically independent elements of $M$, say $m_i =\sum f_{ij}e_j$. Then $W$ is contained in the zero set of the determinants of the $(s,s)$-minors of $(f_{ij})$.
\end{enumerate}
 \end{proof}
Combining this with Lemma \ref{generatorsI/I2}, we deduce that to prove Proposition \ref{unlikely_intersection}, it will be enough to prove that, 
if $\theta _1 ,\ldots ,\theta _{r_n }$ are not $ \overline{\Q }_p (V)$-linearly independent, then there exists $h_1 ,h_2 \in \overline{\Q }_p (V)$ such that $h_1 dh_2 \neq 0$, and $h_1 dh_2 $ is in the $ \overline{\Q }_p (V)$-span of $\theta _1 ,\ldots ,\theta _{r_n }$. We prove this by induction on $n$.

The case $n=1$ is elementary: in this case $\theta _i $ is of the form $dt_i -\omega _i $, where $\omega _1 ,\ldots ,\omega _{r_1 }$ are closed $1$-forms forming a basis of $H^1 _{\dR} (Z/\overline{\Q }_p )$. Suppose 
\begin{equation}\label{relation_level1}
\sum_{i=1}^{r_1} a_i \theta _i =0
\end{equation}
in $\Omega _{\overline{\Q }_p (V)|\overline{\Q }_p }$. Suppose \eqref{relation_level1} is minimal among all non-trivial relations, in the sense that it has a minimal number of nonzero $a_i$ among all relations \eqref{relation_level1} for which the $a_i$ are not identically zero. Without loss of generality $a_1$ is non-zero, and re-scaling if necessary, we may assume $a_1 =1$. Since $\omega _i $ are closed for $i\leq r_1 $, we have
\[
da_i \wedge \theta _i =d(\sum a_i \wedge \theta _i ) =0
\]
in $\Omega ^2 _{\overline{\Q }_p (V)|\overline{\Q }_p }$. If all of the $da_i$ are zero in $\Omega _{\overline{\Q }_p (V)|\overline{\Q }_p }$, then we have
\[
d(\sum a_i t_i )=\sum a_i \omega _i ,
\]
hence the map $H^1 _{\dR}(Z_{\overline{\Q }_p }/\overline{\Q }_p) \to H^1 _{\dR}(V/\overline{\Q }_p)$ has a non-trivial kernel. Dually, this implies that the map on Albanese varieties is not dominant, and hence that $V\to Z$ is not dominant, giving a contradiction.

If the $da_i$ are not all zero, but are in the $\overline{\Q }_p (V)$-span of the $\theta _i$, then the Proposition is proved, by Lemma \ref{no_diffl}. Hence we reduce to the case that
the $da_i $ are not all zero, and not all in the $\overline{\Q }_p (V)$-span of the $\theta _i $ (say $da_2 \notin \sum \overline{\Q }_p (V)\cdot \theta _i $). Then there is a derivation
\[
D:\Omega _{\overline{\Q }_p (V)|\overline{\Q }_p }\to \overline{\Q }_p (V)
\]
for which $D(\theta _i )=0$ for all $i$ but $D(da_2 )\neq 0$. Then
\[
\sum D(da_i )\cdot \theta _i =0
\]
is a non-trivial relation, which has fewer non-zero terms than \eqref{relation_level1} since $da_1 =0$, contradicting our assumption that \eqref{relation_level1} was minimal.

Now suppose $n>1$, and let $M_k$ denote the submodule of $\Omega _{\overline{\Q }_p (V)|\overline{\Q }_p }$ spanned by $\theta _1 ,\ldots ,\theta _{r_k }$. We suppose that
\[
\{h_1 dh_2 :h_1 ,h_2 \in \overline{\Q }_p (V)\} \cap M_n =0.
\]
Hence, by induction $\rk M_{n-1}=r_{n-1}$. Suppose as hypothesis for contradiction that $\rk M_n < r_n $, and hence there are non-zero $a_i \in \overline{\Q }_p (V)$ such that
\begin{equation}\label{relation}
\sum _{i=1}^{r_n } a_i \theta _i =0.
\end{equation}

To deduce the inductive step, we use the following differential equation satisfied by the the $\theta _i $ (in the notation of remark \ref{rk:BSCFN}, this identity could also be deduced from the Cartan structure equation for the connection form $\Omega $).
\begin{lemma}\label{derivative_theta}
$\widetilde{\boldsymbol{\theta }}$ satisfies
\[
d\widetilde{\boldsymbol{\theta }}=\frac{1}{2}[\widetilde{\boldsymbol{\theta }},\widetilde{\boldsymbol{\theta }}]+[\widetilde{\boldsymbol{\theta }},\boldsymbol{\omega }].
\]
(here the Lie bracket may be thought of as the Lie bracket on the differential graded Lie algebra $L_n ^{\dR} (Z) \otimes H^0 (L_n \times Z,\Omega ^{\bullet })$). Equivalently, we have
\[
d\widetilde{\theta } _k = \sum _{i,j}b_{ijk}\widetilde{\theta  }_i \wedge (\frac{1}{2}\widetilde{\theta } _j +\omega _j ).
\]
where $b_{ijk}$ are as in \eqref{bijk}.
\end{lemma}
\begin{proof}
We prove this in $\varprojlim L_n ^{\dR} (Z) \otimes H^0 (L_n \times Z,\Omega ^{\bullet })$. Then it is enough to prove 
\[
[\boldsymbol{t},d\widetilde{\boldsymbol{\theta }}]=\frac{1}{2}[\boldsymbol{t},[\widetilde{\boldsymbol{\theta }},\widetilde{\boldsymbol{\theta }}]]+[\boldsymbol{t},[\widetilde{\boldsymbol{\theta }},\boldsymbol{\omega }]].
\]
since $\ad _{\boldsymbol{t}}$ is injective on $\varprojlim L_n^{\dR} (Z)\otimes H^0 (U_n \times Z,\Omega ^{\bullet })$.
From Lemma \ref{poincare} we derive
\begin{align*}
d (e^{\ad _{\boldsymbol{t}}} \cdot d\boldsymbol{t}) & =d(e^{\boldsymbol{t}}\cdot d\boldsymbol{t} \cdot e^{-\boldsymbol{t}}) \\
& = [\frac{e^{\ad _{\boldsymbol{t}}}-1}{\ad _{\boldsymbol{t}}}d\boldsymbol{t},e^{\ad _{\boldsymbol{t}}}d\boldsymbol{t}].
\end{align*}
Since the connection $\nabla $ is flat we have
$
d\boldsymbol{\omega }=\frac{1}{2}[\boldsymbol{\omega },\boldsymbol{\omega }].
$
Hence
\begin{align*}
[\boldsymbol{t},d\widetilde{\boldsymbol{\theta }}] & = [\boldsymbol{t},d\widetilde{\boldsymbol{\theta }}+\boldsymbol{\omega }]-[\boldsymbol{t},d\boldsymbol{\omega }] \\
& = d[\boldsymbol{t},\widetilde{\boldsymbol{\theta }}+\boldsymbol{\omega }] -[d\boldsymbol{t},\widetilde{\boldsymbol{\theta}}+\boldsymbol{\omega } ]-[\boldsymbol{t},d\boldsymbol{\omega }] \\
& = [\widetilde{\boldsymbol{\theta}}+\boldsymbol{\omega },e^{\ad _{\boldsymbol{t}}}d\boldsymbol{t}]-
\widetilde{\boldsymbol{\theta}}+\boldsymbol{\omega },d\boldsymbol{t}]-\frac{1}{2}[\boldsymbol{t},[\boldsymbol{\omega },\boldsymbol{\omega }]] \\
& = [\widetilde{\boldsymbol{\theta}}+\boldsymbol{\omega },[\boldsymbol{t},\widetilde{\boldsymbol{\theta}}+\boldsymbol{\omega }]]-
\widetilde{\boldsymbol{\theta}}+\boldsymbol{\omega },d\boldsymbol{t}]-\frac{1}{2}[\boldsymbol{t},[\boldsymbol{\omega },\boldsymbol{\omega }]] \\
& = \frac{1}{2}[\boldsymbol{t},[\widetilde{\boldsymbol{\theta}}+\boldsymbol{\omega },\widetilde{\boldsymbol{\theta}}+\boldsymbol{\omega }]]-\frac{1}{2}[\boldsymbol{t},[\boldsymbol{\omega },\boldsymbol{\omega }]].
\end{align*}
 \end{proof}

Note that, since $\theta _1 ,\ldots ,\theta _{r_{n-1}}$ are linearly indendent, there is $r_{n-1}<i\leq r_n$ such that $a_i \neq 0$. We choose the $a_i$ minimally in the sense that the size of the set $\{ i\in \{r_{n-1}+1, \ldots ,r_n \}:a_i \neq 0\}$ is minimal among all non-trivial relations \eqref{relation} (where \textit{non-trivial} simply means the $a_i$ are not all zero).

By Lemma \ref{derivative_theta} we have
\begin{equation}\label{Dminimal_reln}
d(\sum a_i \theta _i ) =\sum \theta _i \wedge (-da_i +\sum b_{ijk}a_k (\frac{1}{2}\theta _j +\omega _j ))= 0.
\end{equation}
Suppose that for $r_{n-1}<i\leq r_n$, the $a_i $ are not all constant. Pick $j_0$ between $r_{n-1}$ and $r_n$ such that 
$a_{j_0 }$ is non-zero. Rescaling if necessary, we may assume $a_{j_0 } =1$. We claim that, after this rescaling, the $a_i$ are all constant for $i>r_{n-1}$. Suppose $a_{j_1 }$ is non-constant. Since, by assumption, $da_{j_1 }$ is not in the span of the $\theta _i $, there is a derivation
\[
D:\Omega _{\overline{\Q }_p (V)|\overline{\Q }_p }\to \overline{\Q }_p (V)
\]
such that $D(\theta _i )=0$ for all $i$, and $D(da_{j_1 } )\neq 0$. 
Write $c_i :=D(-da_i +\sum b_{ijk}a_k (\frac{1}{2}\theta _j +\omega _j ))$.
Then
\begin{equation}\label{minimal_contradiction}
D(d(\sum _{i=1}^{r_{n}}a_i \theta _i ))=\sum _{i=1}^{r_{n-1}} c_i \theta _i =0
\end{equation}
Since $b_{ijk}=0$ whenever $i$ or $j$ are greater than $r_{k-1}$, we have, for all $i>r_{n-1}$,
\[
c_i =D(-da_i ).
\]
In particular $c_{j_1 }\neq 0$ and $c_{j_0 }=0$. Then \eqref{minimal_contradiction} is a non-trivial relation with a smaller number of non-zero terms between $r_{n-1}$ and $r_n$, contradicting our assumption of minimality of \eqref{relation}.

Hence we may assume that for $r_{n-1}<i\leq r_n$, the $a_i$ are in $\overline{\Q }_p$. Define 
\[
\alpha _i =da_i -\sum b_{ijk}a_k \omega _j .
\] 
Then \eqref{Dminimal_reln} can be rewritten
\[
\sum _{i=1}^{r_{n-1}}\theta _i \wedge \alpha _i =\frac{1}{2}\sum _k a_k\sum _{i,j \leq r_{n-1}}b_{ijk} \theta _i \wedge \theta _j .
\]
Hence, by our assumption that $\theta _1 ,\ldots ,\theta _{r_{n-1}}$ are $\overline{\Q }_p (V)$-linearly independent, each $\alpha _i$ is in the $ \overline{\Q }_p (V)$-span of $\theta _1 ,\ldots ,\theta _{r_{n-1}}$, and in fact can be written as
\begin{equation}\label{alph}
\alpha _i =\sum \lambda _{ij}\theta _j
\end{equation}
where 
\begin{equation}\label{lambda}
\lambda _{ij}-\lambda _{ji}=\sum _k b_{ijk}a_k .
\end{equation}

A solution to \eqref{alph} has the following interpretation. We can define a connection on $L_n ^{\dR} (Z) \otimes \mathcal{O}_Z$ is given by
$
x\mapsto [\boldsymbol{\omega } ,x],
$
or with respect to our chosen basis by
\[
e_{l_j} \mapsto \sum b_{ijk}\omega _{l_i } e_{l_k} .
\]

Let $\mathcal{L}_n ^{\dR}(Z)$ denote the unipotent flat connection corresponding, by Tannaka duality, to $L_n ^{\dR}(Z)$ with the adjoint action of $\pi _1 ^{\dR} (Z,b)$. By Lemma \ref{lemma:broken_link} and Lemma \ref{omega}, we have $\mathcal{L}_n ^{\dR} (Z)\simeq L_n ^{\dR} (Z)\otimes \mathcal{O}_Z $ with connection given by
\[
v\otimes 1 \mapsto [\boldsymbol{\omega },v].
\]

The dual connection $\nabla _{L_n ^{\dR *} }$ on $\mathcal{L}_n ^{\dR }(Z)^*$ is given by
\[
e_{l_j } ^* \mapsto \sum -b_{ikj}\omega _i e_{l_k} ^* .
\]
Hence the $\alpha _i$ can also be interpreted as the coefficients, with respect to the basis $e_{l_i}^* $, of $\nabla _{L_n ^{\dR *}}(\boldsymbol{a})$, where $\boldsymbol{a}=\sum a_i e_{l_i }^* $. As the following lemma explains, the condition that $\alpha _i \in M_{n-1}$ for all $i$ is equivalent to the existence of a morphism $\mathcal{E}_{n-1}\to \mathcal{L}_n ^{\dR}(Z)^* $. Let $\pi :V\to Z$ denote the projection.

\begin{lemma}
\begin{enumerate}
\item $M_n$ is equal to the submodule spanned by the coefficients of $\pi ^* \nabla _n (\exp (\boldsymbol{t}))$ with respect to the basis $e_i $, where $\pi ^* \nabla _n $ denotes the connection $\mathcal{E}_n ^{\dR}(Z)$ pulled back to $V$.
\item Suppose $\{0 \}=\{ h_1 dh_2 :h_1 ,h_2 \in  \overline{\Q }_p (V)\} \cap M_{n-1}$, and $\rk M_{n-1}=r_{n-1}$. Let $\mathcal{V}=(\mathcal{O}_{Z}^{\oplus N},\nabla _{\mathcal{V}})$ be an $(n-1)$-unipotent flat connection on $Z$ in reduced form. Let $x\in V(\overline{\Q }_p )$, and let $\mathcal{O}_{V,x}$ denote the local ring of $V$ at $x$. Then, for each $\overline{v}\in x^* \pi ^* \mathcal{V}$, a lift of $\overline{v}$ to $v\in \pi ^* \mathcal{V}(\mathcal{O}_{V,x})$ such that $\nabla _{\pi ^* \mathcal{V}}(v)$ lies in the subspace $\mathcal{O}_V ^{\oplus N}\otimes M_{n-1}$ is unique if it exists.
\item Given any morphism of flat connections
\[
P:\mathcal{E}_{n-1}\to \mathcal{V}
\]
on $Z$,
we have a solution to
\begin{equation}\label{eta_lift}
\nabla _{\pi ^* \mathcal{V}}(v)\in \mathcal{\pi ^* V}\otimes M_{n-1}
\end{equation}
given by $v=P\exp (\boldsymbol{t})$. In particular, by the universal property of $\mathcal{E}_m $, for all $m> 0$, and any $\overline{v}$ in $b^* \mathcal{V}$, there exists a unique lift of $\overline{v}$ to $v\in \pi ^* \mathcal{V}(V)$ satisfying $\nabla _{\pi ^* \mathcal{V}}(v)\in M_m$.
\end{enumerate}
\end{lemma}
\begin{proof}
\begin{enumerate}
\item As in the proof of Lemma \ref{derivative_theta}, this follows from Lemma \ref{poincare}, which gives
\begin{align*}
\nabla _n (\exp (\boldsymbol{t})) & = d\exp (\boldsymbol{t})-\boldsymbol{\omega }\cdot \exp(\boldsymbol{t}) \\
& = \boldsymbol{\theta }\cdot \exp (\boldsymbol{t}).
\end{align*}
\item We prove this by induction on the minimal degree of unipotence of $\mathcal{V}$. For $0$-unipotent connections, $\nabla (\sum a_i \otimes e_i )=da_i \otimes e_i$, so the result is immediate. Given the result for $k$-unipotent connections, let $\mathcal{V}$ be $(k+1)$-unipotent, and let 
\[
0\to \mathcal{V}' \to \mathcal{V}\stackrel{\tau }{\longrightarrow }\mathcal{V}''\to 0
\]
be a short exact sequence where $\mathcal{V}'$ is $0$-unipotent and $\mathcal{V}''$ is $k$-unipotent. Let $v$ and $w$ be two lifts of $\overline{v}$ to sections in $\pi ^* \mathcal{V}(\mathcal{O}_{V,x})$ such that $\nabla _{\pi ^* \mathcal{V}}(v)$ and $\nabla _{\pi ^* \mathcal{V}}(w)$ are in $\pi ^* \mathcal{V}\otimes M_{n-1}$. Then $\tau (v)$ and $\tau (w)$ are equal at $x$, and satisfy $\nabla _{\pi ^* \mathcal{V}''}(\tau (v)),\nabla _{\pi ^* \mathcal{V}''}(\tau (w))$ $\in \pi ^* \mathcal{V}'' \otimes M_{n-1}$. Since $\mathcal{V}''$ is $k$-unipotent, we have $\tau (v)=\tau (w)$, hence $v-w$ is a section of the trivial connection $\mathcal{V}'$ satisfying $\nabla _{\pi ^* \mathcal{V}'}(v-w)\in \pi ^* \mathcal{V}' \otimes M_{n-1}$, hence is zero.
\item By Lemma \ref{lemma:reduced}, the morphism is $K$-linear, since both connections are in reduced form. Hence it satisfies
\[
P\nabla _{n-1}=\nabla _{\mathcal{V}}P,
\]
hence $P\exp (\boldsymbol{t})$ satisfies \eqref{eta_lift}.
\end{enumerate}
 \end{proof}
Now suppose we have a solution to 
\begin{equation}\label{spurious_soln}
\alpha _i =\sum \lambda _{ij}\theta _j
\end{equation}
with $\lambda _{ij}$ in $\overline{\Q }_p (V)$. Let $x$ be a point at which none of the $\lambda _{ij}$ have a pole, then by Lemma \ref{eta_lift}, a solution to \eqref{spurious_soln} with $\lambda _{ij}\in \mathcal{O}(Z)_x $ is unique given their value at $x$, and for every choice of $(c_{ij})$ with $c_{ij}$ in $\overline{\Q }_p $, there is a unique lift to $\lambda _{ij}\in \mathcal{O}(Z)_x$ satisfying $\lambda _{ij}(x)=c_{ij}$, and coming from a morphism of flat connections
$
\mathcal{E}_{n-1}\to \mathcal{L}_n .
$
In particular the $\lambda _{ij}$ are actually in $\mathcal{O}(Z)$. We claim that for all $i \in \{r_{n-2}+1,\ldots ,r_{n-1}\},j\in \{1,\ldots ,r_1 \}$, if \eqref{lambda} holds then, for all $i\geq r_{n-2}+1$,
$
\lambda _{ji}=0.
$

For $i\in \{ r_{n-2}+1,\ldots r_{n-1}\}$, we have
\[
\alpha _i =da_i -\sum _{k=r_{n-1}+1}^{r_n } b_{ijk}a_k \omega _j .
\]
since $b_{ijk}=0$ if $i>r_{n-2}$ and $k\leq r_{n-1}$. Since $a_k $ is constant for $k>r_{n-1}$, we deduce
\[
a_i =\sum b_{ijk}a_k t_j
\]
is a solution, hence for $i$ in $\{r_{n-1}+1,\ldots ,r_n \}$, we have
\[
\lambda _{ij}=\sum _{k=r_{n-1}+1}^{r_n } b_{ijk}a_k .
\]
This completes the proof of the claim.

This means that the map
\[
\mathcal{E}_{n-1}\to \mathcal{L}_n ^*
\]
factors through $\mathcal{E}_{n-2}$, or equivalently that the action of $b^* \mathcal{E}_{n-1}$ on $\sum a_k e_{l_k} ^* $ factors through $b^* \mathcal{E}_{n-2}$. Since we assume $a_k$ is non-zero for some $k\in \{r_{n-1}+1,\ldots ,r_n \}$, the $b^* \mathcal{E}_{n-1}$-module generated by $\sum a_k e_{l_k }^* $ is not $(n-2)$-unipotent. Hence we obtain a contradiction, completing the proof of Proposition \ref{unlikely_intersection}.

\section{The intersection of $j_n (Z)$ with the Selmer variety}\label{sec:complete}
In this section we prove certain techniques for proving finiteness of $\mathcal{X}(\mathcal{O}_{K}\otimes \Z _p )_{S,n}$ and $\mathcal{X}(\mathcal{O}_{K_v })_{S,n}$. First, we prove a general result which allows to prove finiteness after passing to a finite extension. The only wrinkle this introduces is that, as we have chosen to work with a prime which splits completely in $K$, passing to a finite extension will compel us to work with a prime splitting completely in the extension field.

\begin{lemma}\label{lemma:fin_extn}
Let $L|K$ be a finite extension of $\Q $, and $X$ a curve over $K$ as in the introduction. Let $p$ be a rational prime which splits completely in $L$.
\begin{enumerate}
\item If $\mathcal{X}(\mathcal{O}_L \otimes \Z _p )_{S,n}$ is finite, then $\mathcal{X}(\mathcal{O}_K \otimes \Z _p )_{S,n}$ is finite.
\item Let $w$ be a prime of $K$ lying above $p$, and $v$ a prime of $L$ lying above $w$. If $\mathcal{X}(\mathcal{O}_v )_{S,n}$ is finite, then $\mathcal{X}(\mathcal{O}_w )_{S,n}$ is finite.
\end{enumerate}
\end{lemma}
\begin{proof}
We have diagonal embeddings of $\Res _{K|\Q }(X)$ into $\Res _{L|\Q }(X_L )$ which, upon base change to $\Q _p$, induce embeddings of $X_{K_w}$ into $\prod _{v' |w}X_{L_{v'}}$, where the product is over primes of $L$ lying above $w$. On the other hand, we also have inclusions of $H^1 _{f,S}(G_{K,S},U_n (X))$ into $H^1 _{f,S}(G_{L,S},U_n (X))$. We obtain a commutative diagram
\begin{tikzcd}
\Res _{K|\Q }(X)(\Q ) \arrow[r] \arrow[d] & H^1 _{f,S} (G_{\Q ,S} ,\Ind ^\Q  _K U_n (X)) \arrow[d] \\
\Res _{L|\Q }(X_L )(\Q ) \arrow[r]           & H^1 _{f,S}(G_{\Q ,S} ,\Ind ^\Q  _L U_n (X)),         
\end{tikzcd}

and similarly for local fields, giving inclusions
\[
\mathcal{X}(\mathcal{O}_K \otimes \Z _p )_{S,n}\hookrightarrow \mathcal{X}(\mathcal{O}_L \otimes \Z _p )_{S,n}
\]
and
\[
\mathcal{X}(\mathcal{O}_{K_w})_{S,n}\hookrightarrow \mathcal{X}(\mathcal{O}_{L_v})_{S,n}
\]
\end{proof}
\begin{lemma}\label{lemma:twists}
Let $U$ be a Galois-stable finite-dimensional quotient of the $\Q _p$-unipotent fundamental group of $X$. Suppose $\Hom (J_{\sigma _1 ,\mathbb{C}},J_{\sigma _2 ,\mathbb{C}})=0$ for all distinct embeddings $\sigma _i :K\hookrightarrow \mathbb{C}$. Then, if 
\[
\sum _{i=1}^n \h^1 _f (G_{K},\gr _i (U)) \leq [K:\Q ]+\sum _{i=1}^n \sum _{v|p}\dim D_{\dR}(U)/F^0 ,
\]
then there is a prime $v|p$ of $K$ such that $X(K_v )_U $ is finite.
\end{lemma}
\begin{proof}
By Proposition \ref{tangential_loc}, we have
\[
\dim \Sel (U )\leq \sum _{i=1}^n \h^1 _f (G_{K},\gr _i (U)).
\]
On the other hand 
\[
\sum_{v|p}\dim U^{\dR}(X_v)/F^0 = \dim \sum _{i=1}^n \sum _{v|p}\dim D_{\dR}(\gr _i (U) )/F^0 .
\]
Hence, if the inequality in the statement of the Lemma holds, then by Proposition \ref{prop:useful_criterion} the Zariski closure of $X(K\otimes \Q _p )_U $ is a proper subvariety, all of whose positive dimensional irreducible components $Z$ satisfy
\begin{equation}\label{eqn:Zsmall}
\codim _{U ^{\dR}(Z/\Res (X))/F^0 }(\loc _p (\Sel (U ))\cap U ^{\dR}(X/\Res (X))/F^0 ) <\dim Z.
\end{equation}
Since there are no nonzero homomorphisms between the different factors of $\Res (J)_{\Q _p }$, the image of $\Alb (Z)$ in $\Res (J)_{\Q _p }$ is the product of the images in the different factors. Hence if $Z$ dominates each factor of $\Res (X)_{\Q _p }$, then its geometric unipotent fundamental group surjects onto the geometric unipotent fundamental group of $\Res (X)_{\Q _p }$, contradicting \eqref{eqn:Zsmall}.
\end{proof}
This straightforwardly implies case (2) of Proposition \ref{QC_numberfield}.
\begin{proof}[Proof of \ref{QC_numberfield}, case (2)]
Let $U$ be, as in \cite[Proposition 2.2]{QC2}, the quotient of $U_2 $ which is an extension of $V_p J$ by $\Ker (\NS (J_{\overline{K}})\to \NS (X_{\overline{K}}))\otimes \Q _p (1)$. By \cite[Lemma 2.3]{QC2}, we have
\[
h^1 _f (G_{K,S},\gr _2 U)=\sum _{v\in P_{\mathbb{R}}}(\dim \NS (J_{\overline{K}})^{c_v =1} -\dim \NS (J)).
\]

Hence case (2) of Proposition \ref{QC_numberfield} follows from Lemma \ref{lemma:twists}.
\end{proof}

\subsection{Semi-simplicity properties of graded pieces of fundamental groups}
Let $n>0$, and let $X$ be a smooth projective curve of genus $g>1$ over $K$ with $n-1$ marked points $x_i \in X(K)$. Let $Y=X-\{x_1 ,\ldots ,x_{n-1}\}$. Let $\mathcal{M}_{g,n,K}$ denote the moduli stack of $n$-pointed curves of genus $g$ over $K$. For $x_n \in Y(K)$, we have a short exact sequence
\[
1\to \pi _1 (Y_{\overline{K}},x)\to \pi _1 (\mathcal{M}_{g,n,K},[(X,(x_i )_{i=1}^n )]) \to \pi _1 (\mathcal{M}_{g,n-1,K},[(X,(x_i )_{i=1}^{n-1})]) \to 1,
\]
(see e.g. Nakamura--Tsunogai \cite{nakamura1993some}). This induces an outer action of \\ 
$\pi _1 (\mathcal{M}_{g,n-1,K},[(X,(x_i ))])$ on $\pi _1 ^{\et } (Y_{\overline{K}},x_n )$, which induces an outer action on the Malcev completion of $\pi _1 ^{\et }(Y_{\overline{K}},x_n )$, and hence an action on the graded pieces $\gr _i (U_i (Y)(x)$. When $n=1$ or $2$, the action of $\pi _1 (\mathcal{M}_{g,n-1,K},[(X,(x_i ))])$ on $U_1 (Y)=H_1 ^{\et } (X_{\overline{K}},\Q _p )$ has Zariski dense image in $\mathrm{GSp} (U_1 (X))$. The action of $G_K$ on $\gr _i (U_i (Y))$ factors through the action of $\pi _1 (\mathcal{M}_{g,0,K},[X])$ via the morphism
\begin{equation}\label{moduli_stack}
\Gal (\overline{K}|K)\to \pi _1 (\mathcal{M}_{g,n-1,K},[(X,(x_i ))])
\end{equation} 
induced by the morphism $\spec (K)\to \mathcal{M}_{g,n-1,K}$ induced by $[(X,(x_i ))]$.
\begin{lemma}\label{lemma_summands}
Let $X/K$ be either projective or a projective curve minus a point.
\begin{enumerate}
\item The commutator homomorphism
\[
U_1 (X)^{\otimes i}\to \gr _i (U_i (X))
\]
admits a $G_K$-equivariant section.
\item Let $U_i =U_i (\Res (X))$, and let $W\subset V=U_1$ denote the image of the $\Q _p $-Tate module of an abelian subvariety of $\Res (J)_L$, for some $L|\Q$. Then the image of $W^{\otimes i}$ in $\gr _i (U_i )$ is a $G_{L,T}$-stable direct summand of $\gr _i (U_i )$. \end{enumerate}
\end{lemma}
\begin{proof}
Let $n=1$ or $2$ depending on whether $X$ is projective or projective minus a point $x\in \overline{X}(K)$, and let $S=\O $ or $\{x\}$.
By the homomorphism \eqref{moduli_stack}, to prove the first claim it is enough to prove them with $G_K$ replaced by $\pi _1 (\mathcal{M}_{g,n-1,K},[(\overline{X},S)])$.  Since $\pi _1 (\mathcal{M}_{g,n-1,K},[(\overline{X},S)])$ has Zariski dense image in $\mathrm{GSp} (U_1 (X))$, it is enough to prove them with $\pi _1 (\mathcal{M}_{g,n-1,K},[(\overline{X},S)])$ replaced by $\mathrm{GSp}(U_1 (X))$, which proves the Lemma.
The second claim follows from part (1), together with the fact that the image of $W^{\otimes i}$ in $V^{\otimes i}$ is a direct summand, since $W\subset V$ is a direct summand.
 \end{proof}
\subsection{Further reductions}
Given Proposition \ref{prop:useful_criterion}, to prove finiteness of $\mathcal{X}(\mathcal{O}_K \otimes \mathbb{Z}_p )_n$, it is enough to prove that, for any positive-dimensional, geometrically irreducible smooth quasi-projective subvariety $Z\subset \Res (X)_{\Q _p }$, and any virtual basepoint $b\in Z(\Q _p )\cap \Res (X)(\Q _p )$, there is a Galois-stable quotient $U$ of $U_n (\Res (X))$, such that
\begin{equation}\label{codim_inequality}
\codim _{U (Z/\Res (X))/F^0} (\loc _p (\Sel (U )_{\alpha })\cap U (Z/\Res (X))/F^0 )>\dim (Z).
\end{equation}

First, we reduce proving \eqref{codim_inequality} to proving an inequality involving abelian Galois cohomology. Recall from Lemma \ref{virtual_lemma} that, by changing $b$ to a `virtual basepoint' we may assume $\alpha $ is the trivial collection of local conditions, and hence $\Sel (U )_{\alpha }\subset H^1 _{f,S} (G_{\Q ,T},U )$.
Recall that, by Proposition \ref{tangential_loc}, we have
\begin{align*}
& \dim H^1 _{f,S} (G_{\Q ,T},U )\cap U ^{\dR} (Z/\Res (X))/F^0 \\ 
\leq & \max _{c\in H^1 _{f,S} (G_{\Q ,T},U )}\dim \loc _p H^1 _{f,S} (G_{\Q ,T},L ^c )\cap L ^{\dR}  (Z/\Res (X))^c /F^0 ,
\end{align*}
where $L ^{\dR } (Z/\Res (X))^c /F^0$ is defined to be $0$ if $c$ is not in $U ^{\dR } (Z/\Res (X)) /F^0 $. Hence we can estimate the dimension of $H^1 _{f,S} (G_{\Q ,T},U )\cap U ^{\dR} (Z/\Res (X))/F^0$ using the following Lemma.
\begin{lemma}\label{W_bound_2}
Let $U$ be a Galois stable quotient of $U_N$, with Lie algebra $L$, and $c\in H^1 _{f,S}(G_{\Q ,T},U_n )$. Let $Z\subset \Res (X)_{\Q _p }$ be an irreducible subvariety, and let $L^{\dR}(Z/\Res (X))$ denote the image of $\varprojlim L_i ^{\dR}(Z)$ in $D_{\dR}(L)$.
We have
\begin{align*}
& \codim _{L ^{\dR} (Z/\Res (X))^c /F^0 }\loc _p H^1 _{f,S} (G_{\Q ,T},L^c )\cap L ^{\dR} (Z/\Res (X))^c /F^0 ) \\
 \geq & \sum _{i=1}^N \codim _{\gr _i (L^{\dR}(Z/\Res (X)))/F^0 }H^1 _{f,S} (G_{\Q ,T},\gr _i (L))\cap (\gr _i (L^{\dR}(Z/\Res (X)))/F^0 ) \\
& -\sum _{i=1}^{N-1}\dim \Ker (H^1 _{f,S} (G_{\Q ,T},\gr _i (L))\to H^1 _f (G_{\Q _p },\gr _i (L))).
\end{align*}
\end{lemma}
\begin{proof}
We have
\[
\dim L^{\dR }(Z/\Res (X))^c /F^0 =\sum _{i=1}^N \dim \gr _i (L^{\dR}(Z/\Res (X)))/F^0 .
\]
Hence it is enough to prove that
\begin{align*}
& \dim \loc _p (H^1 _{f,S} (G_{\Q ,T},L^c ))\cap L^{\dR} (Z/\Res (X))^c /F^0 \\
 \leq &   \sum _{i=1}^N \dim \loc _p H^1 _{f,S} (G_{\Q ,T},\gr _i (L))\cap (\gr _i (L^{\dR}(Z/\Res (X)))/F^0 ) \\
& +\sum _{i=1}^{N-1}\dim \Ker (H^1 _{f,S} (G_{\Q ,T},\gr _i (L))\to H^1 _f (G_{\Q _p },\gr _i (L))).\\
\end{align*}
Note that $U_n$ acts unipotent on $L$, hence we may replace $\gr _i (L)$ with $\gr _i (L^c )$ in the above.

This is then just linear algebra: more generally suppose $A=A_{\bullet },B=B_{\bullet }$ are finite dimensional vector space with separated exhaustive decreasing filtrations, such that $A$ is a strict filtered subpsace of $B$, and $C_{i,j}$ $(i<j)$ are finite dimensional vector spaces such that, for all $i<j<k$, we have a commutative diagram with exact rows
\[
\begin{tikzpicture}
\matrix (m) [matrix of math nodes, row sep=3em,
column sep=3em, text height=1.5ex, text depth=0.25ex]
{0 & C_{j,k} & C_{i,k} & C_{i,j} & \\
0 & B_j /B_k & B_i /B_k & B_i /B_j & 0. \\ };
\path[->]
(m-1-1) edge[auto] node[auto] {} (m-1-2)
(m-1-2) edge[auto] node[auto] {} (m-1-3)
edge[auto] node[auto] {$\phi _{j,k} $  } (m-2-2)
(m-1-3) edge[auto] node[auto] {$\phi _{i,k} $} (m-2-3)
edge[auto] node[auto] {  } (m-1-4)
(m-1-4) edge[auto] node[auto] {$\phi _{j,k} $} (m-2-4)
(m-2-1) edge[auto] node[auto] {  } (m-2-2)
(m-2-2) edge[auto] node[auto] {  } (m-2-3)
(m-2-3) edge[auto] node[auto] {  } (m-2-4)
(m-2-4) edge[auto] node[auto] {  } (m-2-5);
\end{tikzpicture}
\]
Then 
\begin{align*}
& \dim \Ker (\phi _{i,j}(C_{i,j})\cap (A_i /A_j ) \to \phi _{i,i+1}(C_{i,i+1})\cap (A_i /A_{i+1} )) \\
\leq & \dim \Ker (\phi _{i,i+1})+\dim \phi _{i+1 ,j}(C_{i+1,j})\cap (A_{i+1}/A_j ). \\
\end{align*}
Hence, for all $i<j$,
\begin{align*}
\dim \phi _{i,j}(C_{i,j})\cap (A_i /A_j ) \leq &  \sum _{i\leq k<j}\dim  \phi _{k,k+1}(C_{k,k+1})\cap (A_k /A_{k+1} ) \\ 
& +\sum _{i<k<j}\dim \Ker (\phi _{k-1,k}).
\end{align*}
Applying this when $B_i=C_i L^{\dR}$, $A_i=C_i L^{\dR}(Z/\Res (X))$ and $A_{i,j}=H^1 _{f,S} (G_{\Q ,T},C_i L/C_j L)$ completes the proof of Lemma.
 \end{proof}

One subtlety in estimating the dimension of the intersection of $\gr _i U^{\dR}(Z/\Res (X))/F^0 $ with $\loc _p H^1 _{f,S}(G_{\Q ,T},\gr _i (U))_{\overline{\Q }_p }$ is that we do not assume that $Z$ is defined over $\Q $, or even over a number field. However in spite of this, $\gr _i U^{\dR}(Z/\Res (X))/F^0 $ behaves as if it was defined over a number field, in the following sense.
\begin{lemma}\label{nearly_arithmetic}
Let $A=\Alb (\Res (X))$, where $X$ is either a projective curve or a projective curve minus a point.
Let $f:Z\hookrightarrow \Res (X)_{\Q _p }$ be a geometrically irreducible subvariety of $\Res (X)_{\overline{\Q }_p }$. There exists a finite extension $L|\Q $ such that $f_* H_1 ^{\et} (\Alb (Z)_{\overline{\Q }_p },\Q _p ) \subset H_1 ^{\et} (A_{\overline{\Q }_p },\Q _p )$ is stable under the action of $\Gal (L|\Q )$.
\end{lemma}
\begin{proof}
First, note that since $X$ is projective, or projective minus a point, its Albanese variety is abelian. Hence it is enough to show that, for any injective morphism of abelian varieties $
g:B\to A_{\overline{\Q }_p }
$,
defined over $\overline{\Q }_p $, there is a finite extension $L|\Q $ such that $g_* H_1 ^{\et }(B,\Q _p )$ is $G_L$-stable.
There is an endomorphism $\phi \in \End (A_{\overline{\Q }_p })$ such that
\[
g_* H_1 ^{\et} (B,\Q _p )=\phi _* (H_1 ^{\et} (\Alb(\Res (X))_{\overline{\Q }},\Q _p )).
\]
Indeed, we may take the endomorphism to be the composite
\[
A_{\overline{\Q }_p }\to A^* _{\overline{\Q }_p } \to B^* \to B \to A _{\overline{\Q }_p },
\]
for some choice of polarisation on $B$
(since the map $A^* _{\overline{\Q }_p }\to B^* $ is surjective, the image of $H_1 ^{\et }(A_{\overline{\Q }_p },\Q _p )$ under this endomorphism is exactly $g_* H_1 ^{\et }(B,\Q _p )$).
Hence it is enough to show that there exists a finite extension $L|\Q $ such that $\End (A_L )\simeq \End (A_{\overline{\Q }_p} )$, which follows from the classification of endomorphisms of abelian varieties \cite[Corollary IV.1]{mumford1974abelian}.
 \end{proof}

\begin{lemma}\label{galcoh:descent_intersection}
Let $K|\Q $ be a finite extension, and $n>1$. Let $W$ be a subspace of $V:=\Ind ^\Q  _K \Q _p (n)$, stable under $G_{\Q _p }$. Then
\[
\dim \loc _p H^1 (G_{\Q ,T},V)\cap H^1 (G_{\Q _p },W) \leq \dim W^{c=-1},
\]
where $c\in G_{\Q }$ is complex conjugation with respect to an inclusion $\overline{\Q }\hookrightarrow \mathbb{C}$.
\end{lemma}
\begin{proof}
Let $L|\Q $ denote a totally imaginary Galois extension containing $K|\Q $, with Galois group $G$, and $H:=\Gal (L|K)$. Let $M:=\Ind ^\Q  _L \Q _p (n)$.
We have an isomorphism $M\simeq \Q _p [G](n)$, and an inclusion
\[
V\hookrightarrow W.
\]
Hence it is enough to prove that for any $W\subset \Q _p [G]$, we have
\[
\dim \loc _p H^1 (G_{\Q ,T},M)\cap H^1 (G_{\Q _p },W(n)) \leq \dim W^{c=(-1)^{n+1}}.
\]
By Shapiro's lemma, we have a commutative diagram whose vertical arrows are isomorphisms
\[
\begin{tikzcd}
H^1 (G_{\Q ,T},M) \arrow[d] \arrow[r] & H^1 (G_{\Q _p },M) \arrow[d] \\
H^1 (G_{L,T},\Q _p (n)) \arrow[r]           & \oplus _{w|p}H^1 (G_{L_w },\Q _p (n))          
\end{tikzcd}
\]
hence it is enough to prove the claim for the dimension of the image of the bottom horizontal map. Both $H^1  (G_{L,T},\Q _p (n))$ and $\oplus _{w|p}H^1 (G_{L_w },\Q _p (n))$ have the structure of $\Gal (L|\Q )$-modules, with respect to which the localisation map is $\Gal (L|\Q )$-equivariant.

We claim that we have an isomorphism of Galois modules 
\begin{equation}\label{galeq_K}
H^1 (G_{L,T},\Q _p (n))\simeq \Ind ^{\Gal (L|\Q )}_{\langle \overline{c}\rangle }\chi ^{n+1},
\end{equation}
where $\chi $ is the unique nontrivial character of $\langle \overline{c} \rangle $, and $\overline{c}$ now denotes the image of $c$ in $\Gal (L|\Q )$,
and an isomorphism
\begin{equation}\label{galeqp_K}
H^1 (G_{\Q _p },\Q _p [G](n)) \simeq \Q _p [G].
\end{equation}
induced by $H^1 (G_{\Q _p },\Q _p (n))\simeq \Q _p $.

For the first claim, note that Borel's theorem \cite{borel} proves a $\Gal (L|\Q )$-equivariant isomorphism $K_{2n+1}(L)\otimes \mathbb{R}\simeq \Ind ^{\Gal (L|\Q )}_{\langle \overline{c} \rangle }\chi ^{n+1}\otimes \mathbb{R}$. This implies a Galois-equivariant isomorphism $K_{2n+1}(L)\otimes \Q \simeq \Ind ^{\Gal (L|\Q )}_{\langle \overline{c} \rangle }\chi ^{n+1}$. Hence, by Soul\'e's theorem \cite{soule}, we obtain the isomorphism \eqref{galeq_K}.

For the isomorphism \eqref{galeqp_K}, we may use the fact that the Bloch--Kato logarithm is Galois-equivariant by construction, and hence we have $\Gal (L_w |\Q _p )$-equivariant isomorphisms
\[
H^1 (G_{L_w },\Q _p (n))\simeq H^1 _f (G_{L_f },\Q _p (n))\simeq L_f .
\]
Having proved these isomorphisms, we deduce that the image of $H^1 (G_{\Q ,T},V)$ in $H^1 (G_{\Q _p },V)$ must be contained in $\Ker (1+(-1)^n c)$. This implies the Lemma.
\end{proof}

\subsection{Metabelian quotients of fundamental groups}
We say a group $G$ is \textit{metabelian} if $[[G,G],[G,G]]$ is zero, and similarly a Lie algebra $L$ is metabelian if $[[L,L],[L,L]]=0$. The free metabelian Lie algebra on a vector space $W$ is simply the quotient of the free pro-nilpotent Lie algebra $L$ on $W$ by the double commutator $[[L,L],[L,L]]$. Given a vector space $W$, we denote by $\widehat{\Sym }^\bullet (W)$ the completion of the symmetric algebra on $W$ with respect to the ideal generated by $W$. If $L$ is a metabelian Lie algebra, the adjoint action of $L$ on $[L,L]$ factors through $L^{\ab }$, and gives $[L,L]$ the structure of a module over $\widehat{\Sym }^{\bullet }(L^{\ab })$.
\begin{lemma}\label{metabelian_module}

\begin{enumerate}
Let $L^{\ma}$ be the free metabelian $\Q _p $-Lie algebra on generators $x_1 ,\ldots ,x_n $.
\item
Let $M$ denote the $\Q _p [x_1 ,\ldots ,x_n ]$ module
\[
M=\{ (v_1 ,\ldots ,v_n )\in \Sym ^\bullet (L^{\ma ,\ab })^{\oplus n}:\sum v_i x_i =0 \}.
\]
Then we have an isomorphism of $\Q _p [\! [x_1 ,\ldots ,x_n ]\! ]$ modules
\[
\widehat{M}\simeq [L^{\ma },L^{\ma }].
\]
via the identification $\widehat{\Sym }^\bullet (L^{\ma ,\ab })\simeq \Q _p [\! [x_1 ,\ldots ,x_n ]\! ]$.
\item Let $\overline{L}^{\ma }$ be the maximal metabelian quotient on the Lie algebra of the $\Q _p $-unipotent fundamental group of a smooth projective irreducible curve $X$ over an algebraically closed field $F$ of characteristic zero. Let $x_1 ,\ldots x_{2g}$ be a symplectic basis of $H:=H_1 ^{\et }(X,\Q _p )$.
Define
\[
M=\{ (v_1 ,\ldots ,v_n )\in \Sym ^\bullet (H)^{\oplus n}:\sum v_i x_i =0 \},
\]
and define $m:=(v_{g+1 },\ldots ,v_{2g},-v_1 ,\ldots ,-v_g )\in M$.
 Then we have an isomorphism of $\widehat{\Sym }^{\bullet }(H)$-modules.
\[
[\overline{L}^{\ma },\overline{L}^{\ma }]\simeq \widehat{M}/\widehat{\Sym } ^{\bullet }(H)\cdot m . 
\]
\end{enumerate}
\end{lemma}
Given an element $x$ of $[L^{\ma },L^{\ma }]$, we refer to the coefficients $v_i$ of the corresponding element of $\widehat{M}$ as the \textit{Fox differentials} of $x$ (motivated by Ihara's construction, used in the proof of Lemma \ref{metabelian_module}).
\begin{proof}
\begin{enumerate}
\item Let $G$ be the free pro-$p$ group on generators $\gamma _1 ,\ldots ,\gamma _n$. Let $\overline{G}$ be the maximal metabelian quotient
\[
\overline{G}:=G/[[G,G],[G,G]].
\]
We claim that $L^{\ma }$ is isomorphic to the Lie algebra of the $\Q _p $-Malcev completion of $\overline{G}$. This follows from universal properties: $L^{\ma }$ is the Lie algebra of the metabelian quotient $U$  of the free pro-unipotent $\Q _p $-group on generators $x_i $, which we will denote by $\widetilde{U}$. The category of continuous $\Q _p $-representations of $\overline{G}$ is equivalent to the category of metabelian representations of $\widetilde{U}$, which is equivalent to the category of representations of $U$.

Let $G_i$ and $U_i $ denote the maximal $i$-unipotent quotients of $\overline{G}$ and $U$ respectively. 
We have an isomorphism of $\widehat{\Sym }^{\bullet }(V)$-modules
\[
[\overline{G} ,\overline{G} ]\otimes _{\mathbb{Z}_p }\Q _p \simeq [U_i ,U_i ].
\]
Hence it is enough to compute the action of the Lie algebra of $\overline{G}$ on $[G,G]$. 
Let $\mathcal{L}$ denote the $\Z _p $-Lie algebra of $\overline{G}$: this is the $\mathbb{Z}_p $-module
$\varprojlim \oplus _{i\leq n}\gr _i (\overline{G})$, with Lie bracket induced by the commutator on $\overline{G}$. Then, by \cite[2.3.2]{nakamura-takao}, we have an isomorphism of $\Q _p$-Lie algebras
\begin{equation}\label{lie1}
\mathcal{L}\widehat{\otimes }_{\mathbb{Z}_p }\Q _p \simeq L^{\ma }.
\end{equation}
The conjugation action of $G^{\ab }$ on $[\overline{G},\overline{G}]$ gives it the structure of a $\mathbb{Z}_p [\! [G^{\ab }]\! ]$-module. On the other hand, the Lie bracket on $[\mathcal{L},\mathcal{L}]$ gives it the structure of a module over $\mathbb{Z}_p [\! [G^{\ab }]\! ]$. From the definitions, we obtain an isomorphism of $\mathbb{Z}_p [\! [G^{\ab }]\! ]$-modules 
\begin{equation}\label{lie2}
[\mathcal{L},\mathcal{L}]\simeq \gr ^\bullet [\overline{G},\overline{G}].
\end{equation}
In particular, we obtain a non-canonical isomorphism of $\Q _p [\! [x_1 ,\ldots ,x_n ]\! ]$-modules
\[
[L^{\ma },L^{\ma }]\simeq [\overline{G},\overline{G}]\widehat{\otimes }_{\mathbb{Z}_p [\! [G^{\ab }]\! ]}\Q _p [\! [x_1 ,\ldots ,x_n ]\! ]
\]
where $x_i$ acts on $[\overline{G},\overline{G}]$ by $\gamma _i -1$.
Hence it will be enough to prove that we have an isomorphism of $\mathbb{Z}_p [\! [x_1 ,\ldots ,x_n ]\! ]$-modules
\[
[\overline{G},\overline{G}]\simeq \{ (v_1 ,\ldots ,v_n )\in \mathbb{Z}_p [\! [G^{\ab }]\! ]:\sum v_i x_i =0 \}.
\]
This is a special case of a theorem of Ihara \cite[Theorem 2.2]{ihara}.
\item By \cite{hain2011rational}, we know that $\gr ^\bullet \overline{L}^{\ma }$ is isomorphic to a free pro-nilpotent $\Q _p $-Lie algebra on generators $x_1 ,\ldots ,x_{2g}$, modulo the Lie ideal generated by $\sum _{i=1}^g [x_i ,x_{g+i}]$. Hence part (2) follows from part (1).
\end{enumerate}
 \end{proof}

\section{Proof of Theorem \ref{finiteness}}\label{the_proof}
We first recall the following result, which is a corollary of Euler characteristic formulae / Poitou--Tate duality for finite Galois representations, and roughly says that for global Galois cohomology, showing that $H^1 $ is small, showing that $H^2$ is small, and showing that $H^1 $ of the Tate dual is small are equivalent problems.
\begin{lemma} \label{PTD}

\begin{enumerate}
\item \cite[Lemma 2]{jannsen1989} For any finite dimensional $\Q _p $-representation $W$ of $G_{K ,T}$, we have
\[
\h ^1 (G_{K ,T},W)=\h ^2 (G_{K ,T},W)+\h ^0 (G_{K ,T},W)+\sum _{v\in P_{\mathbb{R}}}\dim W^{c_v =-1}+(\# P_{\mathbb{C}})\cdot \dim W,
\]
where $P_{\mathbb{R}}$ and $P_{\mathbb{C}}$ denote the set of real and complex places of $K$ respectively, and the decomposition group at $v\in P_{\mathbb{R}}$ is generated by $c_v $.
\item \cite[Remark 1.2.4]{FPR94}:If $K=\Q $, and $H^0 (G_{\Q ,T},W), H^0 (G_{\Q ,T},W^* (1))$ and $D_{\cris }(W)^{\phi =1}$ are all zero, then
\begin{align*}
& \h ^1 _f (G_{\Q _p },W)-\h ^1 _f (G_{\Q ,T},W) \\
= & \h ^0 (G_{\mathbb{R}},W)- \dim H^1 _f (G_{\Q ,T},W^* (1)).
\end{align*}
\item \cite[1.2.2]{FPR94}For any number field $K$ and $\Q _p $-Galois representation $W$,
\begin{align*}
& \dim (\Ker (H^1 (G_{K,T},W)\to \oplus _{v\in T}H^1 (G_{K_v} ,W)))  \\
= & \dim \Ker (H^2 (G_{K,T},W^* (1))\to \oplus _{v\in T}H^2 (G_{K_v} ,W^* (1))).
\end{align*}
\end{enumerate}
\end{lemma}

\subsection{Theorem \ref{finiteness}, case (1)}
In this subsection we prove Theorem \ref{finiteness} in the case $X=\mathbb{P}^1 _K -D$. First, we make use of the following Lemma to reduce to the case where $S$ is empty.
\begin{lemma}\label{local_ignore}
For all $i>1$, any Galois stable quotient $W$ of $\gr _i (U)$, and any $v\neq p$, we have $H^1 (G_{\Q _v },W)=0$.
\end{lemma}
\begin{proof}
Let $I_v <G_{\Q _v }$ denote the inertia subgroup, and $\phi _v $ a generator of $G_{\Q _v }/I_v $.
Tate duality gives an exact sequence (see e.g. \cite[3.3.9]{FPR94})
\[
0\to H^1 (G_{\mathbb{F}_v },W^{I_v })\to H^1 (G_{\Q _v },W)\to H^1 (G_{\Q _v },W^* (1))^* \to 0.
\]
We have
\[
H^1 (G_{\mathbb{F}_v },W^{I_v })\simeq W^{I_v }/(\phi -1)W^{I_v }.
\]
Since $W$ and $W^* (1)$ have weight $-2i$ and $2-2i$ respectively, we deduce $H^1 (G_{\Q _v },W)=0$.
 \end{proof}
Using the Euler characteristic formula above, we can reduce the computation of dimensions of Galois cohomology groups of Artin--Tate representations (i.e. Tate twists of Artin representations) to a theorem of Soul\'e \cite[Theorem 5]{soule}.
\begin{theorem}\label{souleth}[Soul\'e]
For any number field $K$, and any $n>1$, 
\[
\h ^2 (G_{K,T},\Q _p (n))=0,
\]
\end{theorem}
We deduce case (1) of Theorem \ref{finiteness} as follows. Let $Z$ be an irreducible $\Q _p $-subvariety of $\Res (X)_{\Q _p }$, and let $U_n (Z/X)$ be as in Proposition \ref{prop:useful_criterion}. By Lemma \ref{W_bound_2}, and Theorem \ref{souleth}, it is enough to prove that for infinitely many $n $, $\gr _n (U_n ^{\dR} (Z/X))/F^0 =\gr _n (U_n ^{\dR}(Z/X))$ is not contained in $\loc _p H^1 _f (G_{\Q },\gr _n (U_n ))$. By Lemma \ref{galcoh:descent_intersection}, it is enough to prove that, for infinitely many $n$, $\gr _n (U_n (Z/X))$ is not contained in $U_n ^{c=-1}$. This follows from the fact that $\oplus _n \gr _n (U_n (Z/X))$ is a sub-Lie algebra of the graded Lie algebra $\oplus _n \gr _n U_n $.

\subsection{Theorem \ref{finiteness}, case (2)}
In this subsection we prove Theorem \ref{finiteness} in the case $X$ is a smooth projective curve of genus $g>0$, and we assume either the Bloch--Kato conjectures or Jannsen's conjecture, which we now recall.

\begin{conjecture}[Bloch--Kato, \cite{blochkato}, Conjecture 5.3]\label{BK}
Let $Z$ be a smooth projective variety over $\Q$.
For any $n>0$ and $2r-1\neq n$, the map
\[
ch _{n,r}:K_{2r-1-n}(Z)\otimes \Q _p \to H^1 _g (G_{\Q },H^n (Z_{\overline{\Q }},\Q _p (r)))
\]
is an isomorphism.
\end{conjecture}

\begin{conjecture}[Jannsen, \cite{jannsen1989}, Conjecture 1]\label{janns_conj}
Let $Z$ be a smooth projective variety over $K$ with good reduction outside $T$. Then
\[
H^2 (G_{K,T},H^i (Z_{\overline{K}},\Q _p (n)))=0
\]
whenever $i+1<n$ or $i>2n$.
\end{conjecture}

In particular, when $2r-1-n<0$, since negative $K$-groups are zero, Conjecture \ref{BK} implies
$
H^1 _f (G_{\Q },H^n _{\et} (Z_{\overline{\Q }},\Q _p (r)))=0.
$

\subsubsection{Finiteness assuming Conjecture \ref{BK}}
Part (2) of Lemma \ref{PTD} implies the following corollary of Conjecture \ref{BK}.
\begin{lemma}
Let $X$ be a smooth projective geometrically irreducible curve of genus $g>1$. Suppose Conjecture \ref{BK} holds for $H^n _{\et} (X^n _{\overline{\Q }} ,\Q _p (n))$. Then, for any Galois stable direct summand $W$ of $H^n _{\et} (X^n _{\overline{\Q}},\Q _p (n))$, we have
\[
\h ^1 _f (G_{\Q _p },W)-\h ^1 _f (G_{\Q ,T},W)=\h ^0 (G_{\mathbb{R}},W)
\]
when $n>2$.
\end{lemma}

This means that, for all but finitely many $i$, 
\[
\Ker (H^1 _f (G_{\Q ,T},\gr _i (U_i ))\to H^1 _f (G_{\Q _p },\gr _i (U_i )))=0.
\]

Similarly, by part (3) of Lemma \ref{PTD}, Jannsen's conjecture implies that the localisation map
\begin{equation}\label{jannsen_inj}
H^1 _f (G_{K,T},H^i _{\et } (X_{\overline{K}},\Q _p (i)))\hookrightarrow  \oplus _{v|p}H^1 _f (G_{K_v },H^i _{\et }(X_{\overline{K}},\Q _p (i)))
\end{equation}
is injective for $i>1$. This implies that 
\begin{align}\label{eqn:jannsen_eqn}
&\sum _{i=1}^n \dim \Ker(\loc _p : H^1 _f (G_{\Q ,T},\gr _i (U_i ))\to H^1 _f (G_{\Q _p },\gr _i (U_i ))) \nonumber \\
= & \dim \Ker (H^1 _f (G_{\Q ,T},V)\to H^1 _f (G_{\Q _p },V)).
\end{align}

To prove finiteness of $X(K\otimes \Q _p )_n$, we now show that $U_n (Z/\Res (X))$ contains a large Artin--Tate part, and use this to apply Lemma \ref{galcoh:descent_intersection}. This is done by showing that the unipotent fundamental group of a projective curve contains many Tate motives.
\begin{lemma}[Hain]
Let $X$ be projective. Then 
\[
\gr _6 (U_6 (X))^{\Sp (U_1 (X))}\neq 0.
\]
\end{lemma}
\begin{proof}
This is proved in Hain \cite[\S 9]{hain2011rational}. For the sake of completeness we briefly recall the argument. Let $\mathfrak{p}:= \varinjlim \oplus \gr _i (U_i (X))$, viewed as a graded Lie algebra, so that $U_1 (X)$ has weight $-1$. By computing the Lie algebra cohomology of $\mathfrak{p}$, Hain shows that, for all $i<-2$, the complex
\[
\gr ^W _i (\wedge ^\bullet \mathfrak{p}): \ldots \to \gr ^W _i (\wedge ^2 \mathfrak{p}) \to \gr ^W _i \mathfrak{p}\to 0
\]
is exact. Hence given the irreducible representations arising in $\gr _1 (U_i (X)),\ldots $, $\gr _{i-1}(U_i (X))$, one may compute $\gr _i (U_i (X))$ as an $\spg (U_1 (X))$-representation using the complex $\gr ^W _{-i}(\wedge ^\bullet \mathfrak{p})$. We obtain a non-zero morphism of $\spg (U_1 (X))$-representations $\Q _p \to \gr _6 (U_6 (X))$. 
\end{proof}
\begin{lemma}\label{lemma:artin-tate}
Let $X$ be a smooth projective curve over a number field $K$. Let $\gr ^\bullet (L)$ be the associated graded of the Lie algebra of the $\Q _p $ pro-unipotent completion of the \'etale fundamental group of $X_{\overline{K}}$. Then $\gr ^\bullet (L)$ contains a free Lie algebra generated by $\Q _p (3)$ and $\Q _p (5)$.
\end{lemma}
\begin{proof}
By the previous lemma, $\Q _p (3)$ is a direct summand of $\gr ^6 (L)$. This follows from the fact that it is enough to prove this as a statement about representations of $\GSp (V)$, for $V$ a $2g$-dimensional vector space with a nondegenerate symplectic form. Namely one wants to prove the corresponding statement for $L(V)$, which we define to be the free Lie algebra on $V$ modulo the Lie ideal generated by $\Q _p (1) \subset \wedge ^2 V$. Then it is enough to replace $\GSp (V)$ with $\Sp (V)$ and prove that $L(V)$, as a representation of $\Sp (V)$, contains an invariant vector in $\gr ^6 L(V)$. 

Similarly, to prove the Lemma it is enough to prove that $L(V)$ contains a free Lie algebra with trivial $\Sp (V)$ action, and generators lying in $\gr ^6 (L(V))$ and $\gr ^{10}(L(V))$. By a theorem of Labute \cite[Theorem 1]{labute:free}, if $L$ is a free Lie algebra on a vector space $V$ of dimension $>2$, modulo a Lie ideal generated by one element in $\wedge ^2 V$, then $[L,L]$ is a free Lie algebra (more precisely Labute's theorem says it is a sub-Lie algebra of a free Lie algebra, and thus is free by Shirsov's theorem). Hence, to prove that $L(V)$ contains the free Lie algebra mentioned above, it is enough to prove that $\gr ^6 (L(V))$ and $\gr ^{10}(L(V))$ both contain copies of the trivial representation. 

Hain proves in loc. cit. that $\gr ^6 (L(V))$ contains a copy of the trivial representation. Now let $I$ be the Lie ideal in $[L(V),L(V)]$ generated by the copy of $\Q _p (3)$ in the 6th graded piece. By another theorem of Labute \cite[Theorem 1]{labute:algebre}, $I/[I,I]$ is a free rank one $\mathcal{U}(L/I)$ module, where $\mathcal{U}(L/I)$ is the universal enveloping algebra of $L/I$. By the Poincar\'e--Birkhoff--Witt theorem, $\mathcal{U}(L/I)$ is isomorphic, as a representation on $\Sp (V)$, to the symmetric algebra on $\mathcal{U}(L/I)$. In particular, we see that $I/[I,I]$ contains a copy of $\Sym ^2 \gr ^2 (L)$, which contains a copy of the trivial representation. This can be seen by direct computation, or from the fact that, since $\gr ^2 (L)$ is self-dual as a representation of $\Sp (V)$, either $\Sym ^2 \gr ^2 (L)$ or $\wedge ^2 \gr ^2 (L)$ must contain a copy of the trivial representation, and the latter does not.
\end{proof}

We now deduce case (2) of Theorem \ref{finiteness}, assuming Jannsen's conjecture, as follows. Let $L^t$ be the maximal Artin--Tate subspace of the graded Lie algebra $\oplus \gr ^i (\lie (U_i (\Res (X)) ))$, i.e. the maximal subrepresentation each of whose summands becomes a Tate twist after a finite extension. By Lemma \ref{lemma:artin-tate}, this is infinite dimensional, and contains a free Lie algebra. Since $\gr ^\bullet (U_n (Z/\Res (X)))$ is Galois stable over a finite extension of $\Q $, and $U_n (Z/\Res (X))$ surjects onto a factor of $U_n (\Res (X))$, the intersection of $\gr ^\bullet (U_n (Z/\Res (X))$ with $L^t $ is infinite-dimensional and contains a free Lie algebra. Hence the dimension of the image of $\oplus \gr _i U_i (Z/\Res (X))\cap L^t$ in $L^t /(L^t)^{c=-1}$ is infinite dimensional, and hence by Lemma \ref{galcoh:descent_intersection}, together with \eqref{eqn:jannsen_eqn}, $X(K\otimes \Q _p )_n $ is finite.

\subsection{Proof of Theorem \ref{finiteness}, case (3)}
We follow the argument in \cite{kim2010p}. Let $\overline{U}_n$ denote the maximal metabelian quotient of $U_n (X)$. By Lemma \ref{metabelian_module}, $\varprojlim [\overline{U}_n ,\overline{U}_n ]$ is a free module of rank one over the completed symmetric algebra of $V:=V_p E$.
The Tate module has a decomposition $T_p E \simeq T_{\pi }E \oplus T_{\overline{\pi }}E$ over $L$ where $L|K$ is the field over which the CM is defined. Let $a$ and $b$ be generators of $\lie (\overline{U}_n )$ whose images in $V$ generate $V_\pi E:=\Q _p \otimes T_\pi E$ and $V_{\overline{\pi }} E:=\Q _p \otimes T_{\overline{\pi }} E$ respectively. Then $\lie (\overline{U}_n )$ has a $\Q _p $ basis $a,b$ and $\ad (a)^i \ad (b)^j [a,b]$ for $i+j\leq n-2$. Let $L_{\geq i,\geq j}$ denote the subpsace generated by $\ad (a)^k \ad (b)^l [a,b]$ for $k\geq i,l\geq j$. Then $L_{\geq i ,\geq j}$ is a Lie ideal, and $L_{\geq i ,\geq j}+L_{\geq j,\geq i }$ is a Galois stable Lie ideal. Let $\overline{L}_n $ denote the quotient of $\lie (\overline{U}_n )$ by $L_{\geq 1,\geq 1}$. We have isomorphisms
\[
\gr _i \overline{L}_n \simeq (\Ind ^K _L (V_{\pi }(E)^{\otimes (n-3)}))(1)
\]
for all $3\leq i\leq n$.

By Lemma \ref{lemma:fin_extn} we may enlarge the field $K$ and hence we may assume that $K|\Q $ is Galois, and that all isogenies between $E_{\sigma ,\overline{K}}$ for different embeddings $\sigma $ are in fact defined over $K$. Hence, if $H<\Gal (K|\Q )$ is the subgroup generated by all $\sigma $ such that $E$ is isogenous (over $K$ or equivalently $\overline{K}$) to $E_{\sigma }$, then $V$ descends to a Galois representation $V_0$ of $K_0 :=K^H $.

Let $L(V)$ be the associated graded Lie algebra of $\varprojlim \overline{L}_n $. Equivalently, we may define $L(V)$ to be the free metabelian Lie algebra on $V$ modulo the ideal generated by $\Q _p (2)$. We may similarly define $L(V_0 )$. Recall that we have an isomorphism of representations of $\Gal (\overline{\Q }_p |\Q _p )$
\[
\Res ^\Q _{\Q _p }\Ind ^{\Q }_{K_0 }L(V_0 )\simeq \oplus _{w|p}\Res ^{K_0 }_{K_{0,w}} L(V_0 ).
\]
and that in this way we think of $H^1 _f (G_{K_{0,w}},\gr _n L(V_0 ))$ as a direct summand of $H^1 _f (G_{\Q _p },\Ind ^\Q _{K_0 }\gr _n L(V_0 ))$.
\begin{lemma}\label{lemma:CMhell}
For any infinite set $\mathcal{S}$ of positive integers, there is a prime $w$ of $K_0$ lying above $p$ with the following property. For a positive proportion of $n$ in $\mathcal{S}$, the intersection of $\loc _p H^1 _f (G_{\Q },\Ind ^{\Q }_{K_0 }\gr _n(L(V_0 )))$ with $H^1 _f (G_{K_0 ,w} ,\gr _n (L(V_0 )))$ has dimension at most one.
\end{lemma}
\begin{proof}
Since $E$ has potentially good reduction at all primes, for all $v$ in $S$ there is a finite Galois extension $L_w |K_v$ such that the action of $G_{L_w }$ on $V_p (E)$ is unramified of weight $-1$. Hence for all $i>2$, and any $G_{K_v }$-stable quotient $W$ of $\gr _i (U)$, arguing as in Lemma \ref{local_ignore}, we have $H^1 (G_{L_w },W)=0$. By the Hochschild--Serre spectral sequence, and the fact that $H^1 (\Gal (L_w |K_v ),W')=0$ for all representations $W'$ (or that $H^0 (G_{L_w },W)=0$), we deduce that $H^1 (G_{K_{0,w} },W)=0$ for all $i>2$. Hence for all $i>2$, we have
\[
H^1 _f (G_{K_0 ,T},\gr _i \overline{L}_i )=H^1 _{f,S}(G_{K_0 ,T},\gr _i \overline{L}_i ).
\]

It will be enough to prove that for all but finitely many $i$, $h^1 (G_{K_0 ,T},\gr _i L(V_0 ) )=[K_0 :\Q ]$. By Lemma \ref{PTD}, it will hence be enough to prove that $h^2 (G_{K_0 ,T},\gr _i L(V_0 ))=0$, or equivalently that $h^1 _f (G_{K_0 ,T},\gr _i L(V_0 )^* (1))=0$. It will hence be enough to prove that 
\[
H^1 (G_{L,T},V_\pi (E) ^{\otimes n} )\to \oplus _{v\in T}H^1 (G_{L_v },V_\pi (E) ^{\otimes n} )
\]
is injective for all but finitely $n$ (and similarly for $V_{\overline{\pi }}(E)$.

Let $X_{\infty }$ denote the Galois group of the maximal unramified $\mathbb{Z}_p$-extension $L_{\infty }$ of $L $, where $\Lambda =\mathbb{Z}_p [\! [\Gal (L_{\infty }|L )]\! ]$. From the Hochschild--Serre spectral sequence we have an exact sequence
\[
H^1 (\Gal (L_{\infty }|L ),V_\pi (E) ^{\otimes i})\to H^1 (G_{L,T},V_\pi (E) ^{\otimes i})\to \Hom _{\Lambda }(X_{\infty },V_\pi (E) ^{\otimes i} ),
\]
using the isomorphism $\Hom _{\Lambda }(X_{\infty },V_\pi (E) ^{\otimes i})\simeq H^0 (\Gal (L_{\infty }|L),H^1 (\Gal (\overline{\Q }|L_{\infty }),V_\pi (E) ^{\otimes i})).$
Since $H^1 (\Gal (L_{\infty }|L),V_\pi (E) ^{\otimes i})=0$ for all $i\neq 0$, it is enough to bound the dimension of $\Hom _{\Lambda }(X_{\infty },V_\pi (E) ^{\otimes i})$.
By Iwasawa's theorem \cite[Theorem 5]{iwasawa}, $X_{\infty }$ is a torsion $\Lambda $-module, hence
$
\Hom _{\Lambda }(X_{\infty },V_\pi (E) ^{\otimes i})=0
$
for all but finitely many $i$, and similarly for $V_{\overline{\pi }}(E)$.
\end{proof}

Now we take $\mathcal{S}$ to be $2+e\cdot \mathbb{Z}_{>0}$, where $e:=\# \mu (F)$. Let $w$ be a prime as in Lemma \ref{lemma:CMhell} for the set $\mathcal{S}$. Let $Z\subset (\Res _{K|\Q }E)_{\Q _p }$ be an irreducible subvariety dominating a factor above $w$. We will henceforth identify $\Ind^\Q  _K L(V_0)$ with the associated graded Lie algebra of the corresponding quotient of $\varprojlim U_n (\Res (X))$. Let $\overline{L}_i (Z/\Res (X))$ denote the image of the $\Q _p$-nilpotent Lie algebra of $Z_{\overline{\Q }_p} $ in $\Ind ^\Q  _K \overline{L}_i $. To complete the proof of case (3) of Theorem \ref{finiteness}, we need to show that, for infinitely many $i>0$, $H^1 _f (G_{\Q _p },\gr _i (\overline{L}_i (Z/\Res (X))))$ is not contained in $\loc _p H^1 _f (G_{\Q ,S},\Ind ^\Q  _K \gr _i (\overline{L}_i ))$. Since there are no isogenies between $E_{\overline{K},v_1 }$ and $E_{\overline{K},v_2 }$ if $v_1$ lies above $w$ and $v_2 $ does not, the image of $V_p \Alb (Z)$ in 
$
V_p (\Res (J))\simeq \Ind ^\Q  _K V_p E
$ 
is a direct sum of the image in $\oplus _{v|w}V_p E_v$ and the image in $\oplus _{v\nmid w}V_p E_w$. It follows that $\overline{U}_n (Z/\Res (X))$ is a direct product of its image in $\prod _{v|w}\overline{U}_n (X_v )$ and its image in $\prod _{v\nmid w}\overline{U}_n (X_v )$. Hence, to show that for infinitely many $i$, $H^1 _f (G_{\Q _p },\gr _i (\overline{U}_i (Z/\Res (X))))$ is not contained in $\loc _p H^1 _{f,S}(G_{\Q ,S},\Ind ^\Q  _K \gr _i (\overline{U}_i (X)))$, it will be enough to show that for infinitely many $i$, the image of $H^1 _f (G_{\Q _p },\gr _i (\overline{U}_i (Z/\Res (X))))$ in $\oplus _{v|w}H^1 _f(G_{K_v },\gr _i (\overline{U}_i ))$ is not contained in 
$M_i:= \loc _p H^1 _f(G_{\Q ,S},\Ind ^\Q  _K \gr _i (\overline{U}_i ))\cap \oplus _{v|w}H^1 _f(G_{K_v },\gr _i (\overline{U}_i))$.

Via the isomorphism $V\simeq \Res ^{K_0} _{K}V_0$, we have an $H$-action on $H^1 _f (G_{\Q },\Ind ^\Q  _K \gr _i (U_i ))$, compatible with the $H$-action on $\oplus _{v|w}H^1 _f (G_{K_v },\gr _i(U_i ))$. In particular, $M_i $ is an $H$-submodule of $\oplus _{v|w}H^1 _f (G_{K_v },\gr _i (U_i ))$. Hence we deduce that it is enough to show that for infinitely many $i$, the $H$-module $N_i$ generated by the image of $H^1 _f (G_{\Q _p },\gr _i (\overline{U}_i (Z/\Res (X))))$ in $\oplus _{v|w}H^1 _f (G_{K_v },\gr _i (\overline{U}_i ))$ is not contained in $M_i$. In fact, we show that $N_i ^H$ is not contained $M_i ^H$. Note that $M_i ^H$ is equal to the intersection of $\loc _p H^1 _f (G_{\Q },\gr _i (L(V_0 )))$ with $H^1 (G_{K_{0,w}},\gr _i (L(V_0 )))$.

Note that, via the isomorphism $\oplus _{v|w}H^1 _f (G_{K_v },\gr _i (\overline{L}_i ))\simeq \Q _p [H]\otimes H^1 _f (G_{K_{0,w}},\gr _i (\overline{L}_i ))$, the subspace $M_i ^H$ can be identified with the image of $H^1 _f (G_{\Q _p },\gr _i (\overline{L}_i (Z/\prod _{v|w}X_v )))$ in $H^1 _f (G_{K_{0,w} },\gr _i (\overline{L}_i ))$ with respect to the norm map
\[
\Nm :\Q _p [H]\otimes H^1 _f (G_{K_{0,w}},\gr _i (\overline{L}_i ))\to H^1 _f (G_{K_{0,w}},\gr _i (\overline{L}_i ))
\]
induced by the co-unit map $\Q _p [H]\to \Q _p .$
\begin{proposition}
Let $Z\subset (\Res _{K|\Q }E)_{\Q _p }$ be an irreducible subvariety. Suppose $Z$ dominates a factor above $w$. Let $W\subset \Ind ^{K_0 } _K V_0 $ denote the image of $V_p \Alb (Z)$. Let $L(W)\subset \Ind ^{K_0 } _K L(V_0 )$ denote the sub-Lie algebra of $L(V)$ generated by $W$. Then for infinitely many $i>0$, the norm map, restricted to $\gr _i (L(W))$, surjects onto $\gr _i (L(V_0 ))_w$.
\end{proposition}
\begin{proof}
In fact we will show this for a positive proportion of $i>0$.
By embedding $\Q _p $ into $\mathbb{C}$, we can descend $W$ to a sub-$\Q $-vector space $W_{\Q }$ of $H_1 (\Res _{K|\Q }(E)_{\mathbb{C} },\Q )$. It follows that $W_{\Q }$ is an $F:=\End (E_{\overline{\Q }})\otimes \Q $-subspace, hence we may assume it is of the form $(\lambda _1 ,\ldots ,\lambda _m )\cdot H_1 (E_{\mathbb{C}},\Q )$ for some $\lambda _i \in F$ not all zero. Fix an embedding of $F$ into $\Q _p$. Then, choosing a basis $e,\overline{e}$ of $V_0$ such $F$ acts on $e$ by the embedding and on $\overline{e}$ by its conjugate, it follows that the $i$th graded piece of $L(W)$ is spanned by elements of the form
\begin{align*}
(\lambda _j ^{i-1}\overline{\lambda }_j \cdot e^{i-2}[e,\overline{e}])_{j=1}^m  .
\end{align*}
We see that it will be enough to prove that for a positive proportion of $i\geq 0$,
\begin{equation}\label{bigzero}
\sum _j \lambda _j ^{ei+1}\overline{\lambda }_j \neq 0
\end{equation}
(technically, this just shows that the norm is nonzero, but conjugating \eqref{bigzero} shows that the map is surjective).

Without loss of generality, we may assume all the $\lambda _i$ are nonzero. Re-ordering, we may assume that $\lambda _1 ^e ,\ldots ,\lambda _k ^e$ are pairwise distinct and $\lambda _i ^e \in \{\lambda _1 ^e ,\ldots ,\lambda _k ^e \}$ for all $i$. If $k=1$, then it is enough to prove that 
\[
\sum \lambda _j \overline{\lambda }_j \neq 0,
\]
which follows from the fact that $\lambda _j \overline{\lambda }_j >0$. If $k=2$, then by the same argument as for $k=1$, \eqref{bigzero} can be rewritten as 
\[
\sum _{j=1}^2 a_j \lambda _j ^{ei} \neq 0.
\]
for some rational constants $a_i >0$. Since $\lambda _1 ^e \neq \lambda _2 ^e $, we have that $\lambda _1 ^{ie}\neq \lambda _2 ^{ie}$ for all $i>0$, since distinct $e$th powers in $F^\times $ cannot differ by a root of unity. Hence \eqref{bigzero} holds for all but at most one $i$.

Now suppose $k\geq 3$. By Szemer\'edi's theorem \cite{szemeredi}, it will be enough to prove that the set of $i$ such that \eqref{bigzero} does not hold does not contain an arithmetic progression of length $k$. Let $i_1 ,\ldots ,i_k $ be an arithmetic progression in $e\mathbb{Z}_{>0}$. Let $d=i_2 -i_1$. Then to prove that \eqref{bigzero} holds for at least one of the $i_j$, it is enough to prove that for any positive constants $a_i$ in $\Q $, we have
\[
\sum _{j=1}^k a_j \lambda _j ^{i_1 }v_j \neq 0
\] 
where $v_j \in F^k$ is the vector $(1,\lambda _j ^{de} ,\ldots, \lambda _j ^{de(k-1)})$. Taking determinants, this follows from the fact that $\lambda _i ^{ed} \neq \lambda _j ^{ed} $ for $i\neq j $ between $1$ and $k$.
\end{proof}

\subsection{Proof of Theorem \ref{finiteness}, case (4)}
As in \cite{coates2010selmer} and \cite{ellenberg2017rational}, the key estimate in the proof of case (4) of Theorem \ref{finiteness} is the following theorem of Coates and Kim, coming from Greenberg's generalisation of Iwasawa's theorem to the case of $\mathbb{Z}_p ^d $ extensions (\cite[Theorem 1]{greenberg})..
\begin{theorem}[Coates, Kim \cite{coates2010selmer}, Theorem 0.1]\label{CK_thm}
Let $Y/L$ be a curve of genus $g$ whose Jacobian has complex multiplication. Let $\overline{U}_n (Y)$ be the maximal metabelian quotient of $U_n (Y)$. Then, for any finite extension $L'|L$, we have
\[
\sum _{j=1}^n \h ^2 (G_{L',T},\gr _i (\overline{U}_n (Y) ))\leq Bn^{2g-1},
\]
for some constant $B$ depending on $L'$ and $Y$.
\end{theorem}

As in case (3), by Lemma \ref{lemma:fin_extn} we may enlarge $K$ if necessary to assume that $K|\Q $ is Galois and that all isogenies between $J_{\sigma ,\mathbb{C}}$ for different embeddings $\sigma $ are in fact defined over $K$. Hence, if $V:=V_p J$ and $H<\Gal (K|\Q )$ is the subgroup generated by all $\sigma $ such that $\Jac(X)$ is isogenous (over $K$ or equivalently $\overline{K}$) to $\Jac (X_{\sigma })$, then $V$ descends to a Galois representation $V_0$ of $K_0 :=K^H $.

Let $F:=\End (J_{\overline{K}})\otimes \Q $ be the CM field by which $J$ has complex multiplication. Let $e:= \# \mu (F')$, where $F'$ is the Galois closure of $F$. Over a finite extension of $\Q $, $V_0$ decomposes as a direct sum of characters $\chi _i $ for $1\leq i\leq 2g$. Over $K_0 $, the Galois action permutes the $\chi_i$. In particular, if $S\subset \{1 ,\ldots ,2g \}^n$ is stable under the action of $S_{2g}$, then the subspace 
\[
\oplus _{(i_j )\in S}\otimes \chi _{i_j }
\]
of $V_0 ^{\otimes n}$ is stable under the action of $G_K$. Let $L$ denote the free metabelian Lie algebra generated by $V_0 $. For $n$ divisible by $e$, let $V_0 [n]\subset \gr _n (L)$ denote the subspace generated by the image of $\prod _{i=1}^{2g}\Sym ^{a_i }(\chi _i )\otimes \chi _j \otimes \chi _k $ in $L$ under the nested commutator map, where $a_i$ are such that $a_i$ is in $e\cdot \mathbb{Z}_{\geq 0}$ if $i\neq j,k$, and $a_i +1\in e\cdot \mathbb{Z}_{>0}$ otherwise.

\begin{lemma}\label{lemma:positive-const}
There is a positive constant $c_0 <1$ with the following property: there is a prime $w$ of $K_0$ lying above $p$ such that, for a positive proportion of $n$ dividing $e$, the intersection of $\loc _p H^1 _f (G_{\Q },\Ind ^\Q _{K_0 }V_0 [n])$ with $H^1 _f (G_{K_{0,w} },V_0 [n])$ is at most $c_0$ times the dimension of $H^1 _f (G_{K_{0,w} },V_0 [n])$.
\end{lemma}
\begin{proof}
It will be enough to show that the dimension of $\loc _p H^1 _f (G_{\Q },\Ind ^{\Q }_{K_0 }V_0 [n])$ is at most $c_0 <1$ times that of 
\[
H^1 _f (G_{\Q _p },\Ind ^{\Q }_{K_0 }V_0 [n])=\oplus _{w|p}H^1 _f (G_{K_{0,w} },V_0 [n]).
\]

By Lemma \ref{PTD} and Theorem \ref{CK_thm} is enough to show that $\frac{\dim (\Ind ^\Q _{K_0 }V_0 [n])^{c=1}}{\dim (\Ind ^\Q _{K_0 }V_0 [n])}$ is bounded below by a constant $>0$ for $n$ sufficiently large. If $K_0$ contains one complex place, then it follows that the dimension is at least $[K_0 :\Q ]^{-1}\cdot \dim V_0 [n]$. Hence suppose $K_0$ is totally real. Complex conjugation defines an involution of $V_0$ which nontrivially permutes the vector subspaces $\chi _i $. Hence $V_1 :=\oplus _i \chi _i ^{\otimes e}\subset \Sym ^e V_0 $ is stable under conjugation and the minus eigenspace of $V_1$ is nontrivial. Via the action of $\Sym ^\bullet V_0$ on $[L,L]$, the module $\oplus _{e|n}V_0 [n]$ is a module under $\Sym ^\bullet V_1 $, and by Lemma \ref{metabelian_module}, for any nonzero element $x$ of $\Sym ^\bullet V_1 $, the kernel of the action of $x$ on $\oplus _{e|n}V_0 [n]$ is zero. Choosing a nonzero element of $V_1 $ in the minus eigenspace with respect to complex conjugation, we get an injection
\[
V_0 [n]^{c=-1} \hookrightarrow V_0 [n+e]^{c=1}.
\]
Since $\dim V_0 [n]/V_0 [n+e]$ is bounded by a nonzero constant, we see that $V_0 [n]^{c=1} $ is a positive proportion of $V_0 [n]$.
\end{proof}

For a prime $v$ of $K$, let $L_v $ denote the free metabelian Lie algebra generated by $V_v :=V_p \Jac (X_v )$, and $\overline{L}_v$ the associated graded of the Lie algebra of the metabelian quotient of the $\Q _p $-unipotent fundamental group of $X_{v,\overline{\Q }_p }$. Then, by Lemma \ref{metabelian_module}, we have an isomorphism of $G_{\Q _p }$-representations
\begin{equation}\label{eqn:Lv1}
\overline{L}_v \simeq L_v /\Sym ^\bullet (V_v )(1).
\end{equation}
Let $w$ be a prime of $K_0$. The action of $G_{K_0}$ on $\oplus _{v|w}V_v$ induces an action on $\oplus _{v|w}L_v$, giving an isomorphism
\begin{equation}\label{eqn:Lv2}
\Ind ^{K_0 } _K L \simeq \oplus _{v|w}L_v .
\end{equation}

For each $v$, we have a decomposition of $V_v $ as $\oplus _{i=1}^{2g} \chi _{i,v}$, such that the isomorphisms $V_v \simeq \Res V_0$ restrict to isomorphisms between $\Res \chi _{i,v}$ and $\Res \chi _i $ (where the latter restrictions are to a suitably large finite extension of $K_0$). In the same way as for $V_0$, we can define subspaces $V_v [n]$ of $\gr _n (L_v)$ for $n$ divisible by $e$. 

Note that we do not assume that the isogenies between $\Jac (X_v )$ for different $v$ above $w$ respect polarisations, and hence we cannot automatically use \eqref{eqn:Lv1} and \eqref{eqn:Lv2} to identify $\oplus \overline{L}_v $ with the induction of a quotient Lie algebra of $L$. However, since we restrict attention to the spaces $V_v [n]$ and $V_0 [n]$, we can ignore this subtlety by the following lemma.

\begin{lemma}\label{lemma:Vv}
The surjection $L_v \to \overline{L}_v $ restricts to an injection on $\oplus _{e|n}V_v [n]$.
\end{lemma}
\begin{proof}
Let $x$ be an element of $V_v [n]\cap \Ker (L_v \to \overline{L}_v )$. Then there is an element $f$ of $\Sym ^{n-2}(V_v)$ such that $x=f\cdot t$, where $t \in \wedge ^2 V_v $ is a generator of $\Ker (L_v \to \overline{L}_v)$. Lemma \ref{metabelian_module} implies that $x$ is uniquely determined by its Fox differentials. Let $z_i$ be a generator of $\chi _{i,v}$. We may order the characters $\chi _{i,v}$ so that $t$ is a one-dimensional subspace of $\oplus _{i=1}^g \chi _{i,v} \otimes \chi _{g+i,v}$ surjecting onto $\chi _{i,v} \otimes \chi _{g+i,v}$ for all $i$, hence the Fox differential of $x$ with respect to $z_i$ is given by $\alpha _i \cdot f \cdot z_{g+i} $ (when $i\leq g$) or $\alpha _i \cdot x \cdot z_{i-g}$ (when $i>g$) for some nonzero $\alpha _i \in \Q _p $. Since the Fox differential with respect to $z_i $ of an element of $V_v [n]$ must necessarily be in 
\[
\oplus _{(a_i )\in e\cdot \mathbb{Z}_{\geq 0}^{2g}}\Sym ^{a_i -1}(\chi _i )\otimes (\otimes _{j\neq i}\Sym ^{a_j }(\chi _j )),
\]
we see that the congruence conditions force $x$ to be an element of 
\[
\oplus _{(a_i )\in e\cdot \mathbb{Z}_{\geq 0}^{2g}}\Sym ^{a_i -1}(\chi _i )\otimes \Sym ^{a_{i+g} -1}(\chi _i )(\otimes _{j\neq i,g+i}\Sym ^{a_j }(\chi _j ))
\]
for all $1\leq i\leq g$. It follows that $x$ must equal zero.
\end{proof}
Let $Z\subset \Res (X)_{\Q _p }$ be an irreducible subvariety which dominates a factor of $\Res (X)$ at a prime lying above $w$. Let $W\subset \Ind ^{K_0 }_K (V)$ denote the image of $V_p \Alb (Z)$ in $V_p \Res J$. Let $L(W)$ denote the sub-Lie algebra of $L$ generated by $W$. For $n$ divisible by $e$, let $W[n]$ denote the intersection of $\gr _n (L(W))$ with $\oplus _{v|p}V_v [n]$. 

By Lemma \ref{lemma:Vv} and Theorem \ref{CK_thm}, to complete the proof of case (4) of Theorem \ref{finiteness}, it will be enough to show that
\[
\sum _{i\leq n:e|i}\dim (H^1 _f (G_{\Q _p },W[i])/\loc _p H^1 _f (G_{\Q },\Ind ^\Q  _K V[i])\cap H^1 _f (G_{\Q _p },W[i]))\gg n^{2g}
\]
(i.e. that it is bounded below by a nonzero constant times $n^{2g}$). Since 
\[
\lim \dim V_0 [n]/\dim \gr _n L(V_0 )>0,
\] 
it is enough to show that this sum of dimensions makes up a positive proportion of $\sum _{i\leq n:e|i}\dim V_0 [i]$. 

As in the proof of case (3), we reduce to estimating the size of the image of $H^1 _f (G_{\Q _p },W[i])$ in $\oplus _{v|w}H^1 _f (G_{K_v },W[i])$ modulo $M_i$, where
\[
M_i :=\loc _p H^1 _f (G_{\Q },\Ind ^\Q  _K V[i]) \cap \oplus _{v|w}H^1 _f (G_{K_v },W[i]).
\] 
The dimension of this image is at least 
$
\frac{1}{\# H }\dim N_i /M_i ,
$
where $N_i$ is the $H$-module generated by the image of $H^1 _f (G_{\Q _p },W[i])$ in $\oplus _{v|w}H^1 _f (G_{K_v },V_0[i])$.

We use this to reduce to comparing the subspaces $M_i ^H$ and $N_i ^H$ of $\Ind ^{\Q}_{K_0}V_0 [i]\simeq \prod _{v|w}V_v [i]$. As in case (3), we have an isomorphism 
\[
M_i ^H \simeq \loc _p H^1 _f (G_{\Q ,S},\Ind ^{\Q}_{K_0}V_0 [n]) \cap H^1 _f (G_{K_{0,w}},V_0 [n]) \subset H^1 (G_{\Q _p },\Ind ^{\Q}_{K_0} V_0 [n]).
\]
Hence, by Lemma \ref{lemma:positive-const}, we know that for $i$ sufficiently large, $\dim M_i ^H / \h^1 _f (G_{K_{0,w}},V_0 [i])$ is bounded above by a constant strictly less than 1. Hence it will be enough to prove that $\dim N_i ^H / \h^1 _f (G_{K_{0,w}},V_0 [i])$ tends to 1. Since $\dim F^0 (\gr _i (U_i ))=O(i^g )$, we have that $\dim F^0 W[i]$ and $\dim V_0 [i]$ are both $O(i^{g} ).$ Hence, by the $p$-adic comparison theorem, it is enough to prove that $\dim \widetilde{N}_i ^H /\dim V_0 [i]$ tends to $1$, where $\widetilde{N}_i$ is the $H$-module generated by the image of $W[i]$ in $V_0 [i]$ with respect to the composite of the projection to $\oplus _{v|w}V_v [i]$ with the norm map
\[
\Nm :\oplus _{v|w}V_v [i]\simeq \Q _p [H]\otimes V_0 [i]\to V_0 [i].
\]

Hence to complete the proof of case (4) of Theorem \ref{finiteness}, it is enough to prove the following lemma regarding the image of $W[n]$ under the composite of the projection to $\oplus _{v|w}V_v [n]$ and the norm map above.
\begin{lemma}\label{lemma:nonvanishing}
Let $\Nm  W[n]  \subset V_0 [n]$ denote the image of $W[n]$ with respect to the norm map. Then 
\[
\dim \Nm W[n] /\dim V_0 [n]\to 1
\]
as $n\in e\cdot \mathbb{Z}_{>0}$ tends to infinity.
\end{lemma}
Our proof of Lemma \ref{lemma:nonvanishing} is rather elaborate. We first need the following result.
\begin{lemma}\label{lemma:szemeredi}
Let $x_1 ,\ldots ,x_n$ be nonzero elements of a field $F$ such that for all $m>0$ and all $i\neq j$, $x_i ^m \neq x_j ^m$. Let $y_1 ,\ldots ,y_n$ be nonzero elements of $F$. Then, as $m$ tends to infinity, the proportion of $i$ in $\{ 0,\ldots ,m\}$ such that $\sum _j y_j x_j ^i \neq 0$ tends to one uniformly in $m$. More precisely, for any $\epsilon >0$, there is an $m_{\epsilon }$ independent of $x_i$ and $y_i$ such that, for all $m>m_{\epsilon }$, $\sum _{j=1}^n y_j x_j ^i \neq 0$ for at least $(1-\epsilon )m$ elements of $\{1 ,\ldots ,m\}$.
\end{lemma}
\begin{proof}
If $n\leq 2$ the result is clear, so assume $n>2$. Let $S_m \subset \{0 ,\ldots ,m\}$ be the set of $i$ such that $\sum _j y_j x_j ^i =0$. By Szemer\'edi's theorem, it will be enough to show that $S_m$ does not contain any arithmetic progression of length $n$. Let $\{i_1 ,\ldots ,i_n \} \subset \{1 ,\ldots ,m\}$ be an arithmetic progression in $\{1 ,\ldots ,m\}$. Let $d:=i_2 -i_1$. Let $S_0 \subset S$ Then, since the $x_i ^d$ are pairwise distinct by hypothesis, the Vandermonde determinant $\det (x_i ^{d\cdot (j-1)})_{1\leq i,j \leq n}$ is nonzero. Hence the vectors $(x_j ^{i_1 },\ldots  ,x_j ^{i_n} )\in F^n$, for $1\leq j\leq n$, are linearly independent. It follows that there is a $k$ such that
\[
\sum _j y_j x_j ^{i_k }\neq 0.
\]
\end{proof}
Now let $F:=\End ^0 (\Jac (X)_{\overline{K}})\otimes \Q$ be the CM field associated to $X$, and let $\sigma _1 ,\ldots ,\sigma _{2g}$ be the embeddings of $F$ into $\overline{\Q }$.
\begin{lemma}\label{lemma:fermat_propotion}
Let $\lambda _1 ,\ldots ,\lambda _m$ be $m$ elements of $F$ which are not all zero, and are all $e$th powers. Let $
S(n):=\{(a_1 ,\ldots ,a_{2g})\in e\cdot \Z ^{2g}_{\geq 0}:\sum a_j =n\}.
$
Then, as $n\in e\cdot \mathbb{Z}_{>0}$ tends to infinity, the proportion of $(a_i)$ in $S(n)$ such that
\[
F(a_1 ,\ldots ,a_{2g}):=\sum _{i=1}^m \prod _{j=1}^{2g}\sigma _j (\lambda _i )^{a_j }
\]
is nonzero tends to 1.
\end{lemma}
\begin{proof}
Firstly, clearly we can remove all $\lambda _i $ equal to zero, hence we may assume that all $\lambda _i$ are nonzero. We may also rescale so that $\lambda _1 =1$.

We filter the set $\{ 1,\ldots ,m\}$ by subsets as follows. For $1\leq i\leq 2g$, define $T_i $ to be the set of $j$ in $\{1,\ldots ,m\}$ such that $\sigma _k (\lambda _j )=\sigma _1 (\lambda _j )$ for all $1\leq k\leq i$. We prove, by descending induction on $i$, that for all $i\leq m$, as $N$ tends to infinity the proportion of $(a_j )$ in $S(n)$ such that
\[
F_i (a_1,\ldots ,a_{2g}):=\sum _{k\in T_i }\prod _{j=1}^{2g}\sigma _j (\lambda _k )^{a_j }
\]
is nonzero tends to 1. In the case where $i={2g}$, we see that all the $\lambda _k $ for $k$ in $T_{2g}$ are in fact in $\Q $. Since all the $a_i$ are even and all $\lambda _i$ are nonzero, each summand is strictly positive (also, since $\lambda _1 =1$, $T_{2g}$ is nonempty).

Now fix a $k$ between $2$ and $2g$, and suppose that the proportion of $(a_i )$ such that $F_k (a_1 ,\ldots ,a_{2g})$ is nonzero tends to 1 for $(a_i )$ all divisible by $e$. If $T_k =T_{k-1}$ there is nothing to prove in the next inductive step, so assume that there is a $j$ in $T_{k-1}$ such that $\sigma _1 (\lambda _j )\neq \sigma _k (\lambda _j )$. Let $(b_1 ,\ldots ,b_{2g})$ be a tuple such that $F_k (b_1 ,\ldots ,b_{2g})$ is nonzero. Let $R(b_1 ,\ldots ,b_{2g} )$ be the set of tuples $(a_i )\in e\cdot \mathbb{Z}_{\geq 0}^{2g}$ such that $a_i =b_i$ for $i\notin {1,k-1}$, and $a_1 +a_{k-1}=b_1 +b_{k-1}$. Let $\mu _1 ,\ldots ,\mu _\ell $ be the set of (distinct) possible values of $\sigma _k (\lambda _j ^e )/\sigma _1 (\lambda _j ^e )$ for $j$ in $T_{k-1}$, such that $\mu _1 =1$. Then there are constants $c_1 ,\ldots ,c_{\ell }$ such that, for any $(a_i )$ in $R(b_1 ,\ldots ,b_{2g} )$, we can write 
\[
F_{k-1}(a_1 ,\ldots ,a_{2g})=\sum _{i=1}^\ell c_i\cdot \mu _i ^{(a_1 -b_1 )/e}.
\]
In particular, $c_1 =F_k (b_1 ,\ldots ,b_{2g})\neq 0$, and hence the constants $c_i$ are not all zero. Since the $\mu _i$ are all $e$th powers, for all $i>0$ and $j_1 \neq j_2 $,  $\mu ^i _{j_1 }\neq \mu ^i _{j_2 }$. Hence the proportion of $(a_i )$ in $R(b_1 ,\ldots ,b_{2g})$ such that $F_{k-1}(a_1 ,\ldots ,a_{2g})$ is nonzero tends to 1 uniformly in $b_1 +b_{k-1}$, by Lemma \ref{lemma:szemeredi}. Since, for any $n_0$, the proportion of $(b_i )$ in $S(n)$ such that $b_1 +b_{k-1}\leq n_0$ tends to zero, we deduce that the proportion of $(a_i)$ in $S(n)$ such that $F_{k-1}(a_1 ,\ldots ,a_{2g})$ is nonzero tends to 1 as required.
\end{proof}

\begin{proof}[Proof of Lemma \ref{lemma:nonvanishing}]
By embedding $\Q _p $ into $\mathbb{C}$ and viewing $Z$ as a complex subvariety, we can descend the subspace $W$ of $\oplus _{v|w}V_v$ to a subspace $W_{\Q }$ of $\prod H_1 (X_{v,\mathbb{C}},\Q )$. Then we may assume $W$ is a subspace of $V_0 ^m$ of the form $(\lambda _1 ,\ldots ,\lambda _m )\cdot V_0 $, for some $\lambda _i$ in $F\subset \End (V_0 )$. The action of $\lambda \in F$ on $\chi _i $ is given by $\sigma  _i (\lambda )$ (after re-ordering the $\sigma _i $ if necessary). Hence $\Nm W [n]$ consists of elements of the form
\begin{equation}\label{eqn:sigma}
\sum _{i=1}^m  \sigma _k (\lambda _i )  \sigma _{\ell }(\lambda _i )(\prod _{j=1}^{2g} \sigma _j ( \lambda _i )^{a_j })(\prod _{j=1}^{2g}e_j ^{a_j })[e_k ,e_\ell ]
\end{equation}
for $a_i \in S(n)$, where $e_i$ is a generator of $\chi _i $. Hence, to show that the kernel of the induced map $
V_0 [n]\to V_0 [n]
$
has dimension $o(\dim V_0 [n])$, it is enough to show that the proportion of $a_i $ in $S(n)$ such that \eqref{eqn:sigma} is zero is $o(1)$, which follows from Lemma \ref{lemma:fermat_propotion}.
\end{proof}
\bibliography{bib_unlikely}

\begin{thebibliography}{BDCKW18}

\bibitem[Ax71]{ax1971schanuel}
J.~Ax.
\newblock On {S}chanuel's conjectures.
\newblock {\em Annals of mathematics}, pages 252--268, 1971.

\bibitem[Ax72]{ax1972some}
J.~Ax.
\newblock Some topics in differential algebraic geometry {I}: {A}nalytic
  subgroups of algebraic groups.
\newblock {\em American Journal of Mathematics}, 94(4):1195--1204, 1972.

\bibitem[BBBM20]{balakrishnan2020explicit}
J.~Balakrishnan, A.~Besser, F.~Bianchi, and J.~S. Mueller.
\newblock Explicit quadratic chabauty over number fields.
\newblock {\em Israel Journal of Mathematics}, 2020.

\bibitem[BD18]{QC1}
J.~S. Balakrishnan and N.~Dogra.
\newblock Quadratic {C}habauty and rational points, {I}: {$p$}-adic heights.
\newblock {\em Duke Math. J.}, 167, 2018.
\newblock With an appendix by J. Steffen M\"{u}ller.

\bibitem[BD21]{QC2}
J.~S. Balakrishnan and N.~Dogra.
\newblock Quadratic {C}habauty and rational points {II}: {G}eneralised height
  functions on {S}elmer varieties.
\newblock {\em Int. Math. Res. Not. IMRN}, (15):11923--12008, 2021.

\bibitem[BDCKW18]{BDCKW}
J.~S. Balakrishnan, I.~Dan-Cohen, M.~Kim, and S.~Wewers.
\newblock A non-abelian conjecture of {T}ate-{S}hafarevich type for hyperbolic
  curves.
\newblock {\em Math. Ann.}, 372(1-2):369--428, 2018.

\bibitem[BK90]{blochkato}
S.~Bloch and K.~Kato.
\newblock {L}-functions and {T}amagawa numbers of motives, in {T}he
  {G}rothendieck {F}estschrift, {V}ol {I}.
\newblock pages 333--400. Birkh\"auser Boston, 1990.

\bibitem[BLR90]{BLR}
S.~Bosch, W.~Lutkebohmer, and M.~Raynaud.
\newblock N{\'e}ron {M}odels, 1990.

\bibitem[Bor72]{borel}
A.~Borel.
\newblock Cohomologie r\'{e}elle stable de groupes {$S$}-arithm\'{e}tiques
  classiques.
\newblock {\em C. R. Acad. Sci. Paris S\'{e}r. A-B}, 274:A1700--A1702, 1972.

\bibitem[BSCFN21]{BSCFN}
D.~Bl{\'a}zquez-Sanz, G.~Casale, J.~Freitag, and J.~Nagloo.
\newblock A differential approach to the {A}x-{S}chanuel, {I}.
\newblock {\em arXiv preprint arXiv:2102.03384v2}, 2021.

\bibitem[CDC20]{CDC}
D.~Corwin and I.~Dan-Cohen.
\newblock The polylog quotient and the {G}oncharov quotient in computational
  {C}habauty--{K}im {T}heory {I}.
\newblock {\em Int. J. Number Theory}, 16(8):1859--1905, 2020.

\bibitem[CK10]{coates2010selmer}
J.~Coates and M.~Kim.
\newblock Selmer varieties for curves with {C}{M} {J}acobians.
\newblock {\em Kyoto Journal of Mathematics}, 50(4):827--852, 2010.

\bibitem[CS99]{CLS99}
B.~Chiarellotto and B.~Le Stum.
\newblock F-isocristaux unipotents.
\newblock {\em Compositio Math.}, 116:81--110, 1999.

\bibitem[CSM95]{CRS95}
R.~Carter, G.~Segal, and I.~Macdonald.
\newblock {\em Lectures on {L}ie groups and {L}ie algebras}, volume~32 of {\em
  London Mathematical Society Student Texts}.
\newblock Cambridge University Press, Cambridge, 1995.

\bibitem[DC20]{DC}
I.~Dan-Cohen.
\newblock Mixed {T}ate motives and the unit equation {II}.
\newblock {\em Algebra Number Theory}, 14(5):1175--1237, 2020.

\bibitem[DCW16]{DCW}
I.~Dan-Cohen and S.~Wewers.
\newblock Mixed {T}ate motives and the unit equation.
\newblock {\em Int. Math. Res. Not. IMRN}, (17):5291--5354, 2016.

\bibitem[Del89]{Del89}
P.~Deligne.
\newblock Le groupe fondamental de la droite projective moins trois points.
\newblock In {\em Galois groups over {$\mathbf{Q}$}}, volume~16 of {\em Math.
  Inst. Res. Inst. Publ.}, pages 79--297. Springer-Verlag, 1989.

\bibitem[DR15]{DR15}
M.~Dimitrov and D.~Ramakrishnan.
\newblock Arithmetic quotients of the complex ball and a conjecture of {L}ang.
\newblock {\em Doc. Math.}, 20:1185--1205, 2015.

\bibitem[EH17]{ellenberg2017rational}
J.~S. Ellenberg and D.~R. Hast.
\newblock Rational points on solvable curves over $\mathbb{Q}$ via non-abelian
  {C}habauty.
\newblock {\em arXiv preprint arXiv:1706.00525}, 2017.

\bibitem[Fal83]{faltings1983}
G.~Faltings.
\newblock Endlichkeitss{\"a}tze f{\"u}r abelsche variet{\"a}ten {\"u}ber
  zahlk{\"o}rpern.
\newblock {\em Inventiones mathematicae}, 73(3):349--366, 1983.

\bibitem[Fal07]{faltings2007mathematics}
G.~Faltings.
\newblock Mathematics around {K}im’s new proof of {S}iegel’s theorem.
\newblock {\em Diophantine geometry, CRM Series}, 4:173--188, 2007.

\bibitem[FPR94]{FPR94}
J.-M. Fontaine and B.~Perrin-Riou.
\newblock Autour des conjectures de {B}loch et {K}ato: cohomologie galoisienne
  et valeurs de fonctions {$L$}.
\newblock In {\em Motives ({S}eattle, {WA}, 1991)}, volume~55 of {\em Proc.
  Sympos. Pure Math.}, pages 599--706. Amer. Math. Soc., Providence, RI, 1994.

\bibitem[Gre73]{greenberg}
R.~Greenberg.
\newblock The {I}wasawa invariants of ${\Gamma }$-extensions of a fixed number
  field.
\newblock {\em Amer. J. Math.}, 95, 1973.

\bibitem[Gro64]{egaiv}
A.~Grothendieck.
\newblock {\'E}l{\'e}ments de g{\'e}om{\'e}trie alg{\'e}brique: {IV}. {\'e}tude
  locale des sch{\'e}mas et des morphismes de sch{\'e}mas, premi{\`e}re partie.
\newblock {\em Publications Math{\'e}matiques de l'IHES}, 20:5--259, 1964.

\bibitem[Had11]{hadian2011motivic}
M.~Hadian.
\newblock Motivic fundamental groups and integral points.
\newblock {\em Duke Mathematical Journal}, 160(3):503--565, 2011.

\bibitem[Hai87]{hain:higher_alb}
R.~M. Hain.
\newblock Higher {A}lbanese manifolds.
\newblock In {\em Hodge theory ({S}ant {C}ugat, 1985)}, volume 1246 of {\em
  Lecture Notes in Math.}, pages 84--91. Springer, Berlin, 1987.

\bibitem[Hai11]{hain2011rational}
R.~Hain.
\newblock Rational points of universal curves.
\newblock {\em Journal of the American Mathematical Society}, 24(3):709--769,
  2011.

\bibitem[Has18]{hast-thesis}
D.R. Hast.
\newblock Rational points and unipotent fundamental groups.
\newblock {\em PhD Thesis, University of Wisconsin-Madison}, 2018.

\bibitem[Has21]{hast}
D.~R. Hast.
\newblock Functional transcendence for the unipotent {A}lbanese map.
\newblock {\em Algebra Number Theory}, 15(6):1565--1580, 2021.

\bibitem[HZ87]{hain1987unipotent}
R.~M. Hain and S.~Zucker.
\newblock Unipotent variations of mixed {H}odge structure.
\newblock {\em Inventiones mathematicae}, 88(1):83--124, 1987.

\bibitem[Iha86]{ihara}
Y.~Ihara.
\newblock On {G}alois representations arising from towers of coverings of
  {${\bf P}^1- \{0,1,\infty\}$}.
\newblock {\em Invent. Math.}, 86(3):427--459, 1986.

\bibitem[Iwa73]{iwasawa}
K.~Iwasawa.
\newblock On $\mathbb{Z}_l$-extensions of algebraic number fields.
\newblock {\em Ann. of Math.}, 98, 1973.

\bibitem[Jan89]{jannsen1989}
U.~Jannsen.
\newblock On the $l$-adic cohomology of varieties over number fields and its
  {G}alois cohomology.
\newblock In {\em Galois Groups over $\Q $}, pages 315--360. Springer, 1989.

\bibitem[Kim05]{kim:siegel}
M.~Kim.
\newblock The motivic fundamental group of $\mathbf{P}^1 -\{0,1,\infty\}$ and
  the theorem of {S}iegel.
\newblock {\em Invent. Math.}, 161(3):629--656, 2005.

\bibitem[Kim09]{kim:chabauty}
M.~Kim.
\newblock The unipotent {A}lbanese map and {S}elmer varieties for curves.
\newblock {\em Publ. Res. Inst. Math. Sci.}, 45(1):89--133, 2009.

\bibitem[Kim10]{kim2010p}
M.~Kim.
\newblock $p$-adic ${L}$-functions and {S}elmer varieties associated to
  elliptic curves with complex multiplication.
\newblock {\em Annals of mathematics}, pages 751--759, 2010.

\bibitem[Kim12]{kim2012tangential}
M.~Kim.
\newblock Tangential localization for {S}elmer varieties.
\newblock {\em Duke Mathematical Journal}, 161(2):173--199, 2012.

\bibitem[Kli17]{klingler2017hodge}
B.~Klingler.
\newblock Hodge loci and atypical intersections: conjectures.
\newblock {\em arXiv preprint arXiv:1711.09387}, 2017.

\bibitem[KT08]{kim2008component}
M.~Kim and A.~Tamagawa.
\newblock The {$l$}-component of the unipotent {A}lbanese map.
\newblock {\em Math. Ann.}, 340(1):223--235, 2008.

\bibitem[Lab67]{labute:algebre}
J.~P. Labute.
\newblock Alg\`ebres de {L}ie et pro-{$p$}-groupes d\'{e}finis par une seule
  relation.
\newblock {\em Invent. Math.}, 4:142--158, 1967.

\bibitem[Lab95]{labute:free}
J.~P. Labute.
\newblock Free ideals of one-relator graded {L}ie algebras.
\newblock {\em Trans. Amer. Math. Soc.}, 347(1):175--188, 1995.

\bibitem[Mum74]{mumford1974abelian}
D.~Mumford.
\newblock {\em Abelian {V}arieties}, volume 108.
\newblock Oxford University Press, 1974.

\bibitem[NT93]{nakamura1993some}
H.~Nakamura and H.~Tsunogai.
\newblock Some finiteness theorems on {G}alois centralizers in pro-$l$ mapping
  class groups.
\newblock {\em J. reine angew. Math}, 441:115--144, 1993.

\bibitem[NT98]{nakamura-takao}
H.~Namakura and N.~Takao.
\newblock Galois rigidity of pro-$l$ pure braid groups of algebraic curves.
\newblock {\em Transactions of the AMS}, 350(3):1079--1102, 1998.

\bibitem[Ols11]{olsson2011towards}
M.~Olsson.
\newblock Towards non-abelian p-adic {H}odge theory in the good reduction case.
\newblock {\em Mem. Amer. Math. Soc}, 210, 2011.

\bibitem[Ser02]{serre-gc}
J.-P. Serre.
\newblock {\em Galois cohomology}.
\newblock Springer Monographs in Mathematics. Springer-Verlag, Berlin, 2002.
\newblock Translated from the French by Patrick Ion and revised by the author.

\bibitem[SGA71]{SGA1}
{\em Rev\^{e}tements \'{e}tales et groupe fondamental}.
\newblock Lecture Notes in Mathematics, Vol. 224. Springer-Verlag, Berlin-New
  York, 1971.
\newblock S\'{e}minaire de G\'{e}om\'{e}trie Alg\'{e}brique du Bois Marie
  1960--1961 (SGA 1), Dirig\'{e} par A. Grothendieck. Augment\'{e} de deux
  expos\'{e}s de M. Raynaud.

\bibitem[Sie]{siegel}
C.~L. Siegel.
\newblock \"{U}ber einige {A}nwendungen diophantischer {A}pproximationen %
  [reprint of {A}bhandlungen der {P}reuss ischen {A}kademie der %
  {W}issenschaften. {P}hysikalisch-mathematische {K}lasse 1929, 
\newblock In {\em On some applications of {D}iophantine approximations},
  Quad./Monogr.

\bibitem[Sik13]{siksek2013explicit}
S.~Siksek.
\newblock Explicit {C}habauty over number fields.
\newblock {\em Algebra \& Number Theory}, 7(4):765--793, 2013.

\bibitem[Sou79]{soule}
C.~Soul\'{e}.
\newblock {$K$}-th\'{e}orie des anneaux d'entiers de corps de nombres et
  cohomologie \'{e}tale.
\newblock {\em Invent. Math.}, 55(3):251--295, 1979.

\bibitem[Sti10]{stix2010trading}
J.~Stix.
\newblock Trading degree for dimension in the section conjecture: the
  non-abelian {S}hapiro lemma.
\newblock {\em Mathematical Journal of Okayama University}, 52(1), 2010.

\bibitem[Sze75]{szemeredi}
E.~Szemer\'{e}di.
\newblock On sets of integers containing no {$k$} elements in arithmetic
  progression.
\newblock {\em Acta Arith.}, 27:199--245, 1975.

\bibitem[Tri20]{triantafillou2020restriction}
N.~Triantafillou.
\newblock Restriction of scalars {C}habauty and the {S}-unit equation.
\newblock {\em arXiv preprint arXiv:2006.10590}, 2020.

\bibitem[Tsi15]{tsimerman:unlikely}
J.~Tsimerman.
\newblock Ax-{S}chanuel and o-minimality.
\newblock {\em O-Minimality and Diophantine Geometry}, 421:216, 2015.

\end{thebibliography}
\bibliographystyle{alpha}
\end{document}